\documentclass{amsart} 
\usepackage{a4wide}
\usepackage[utf8]{inputenc} 
\usepackage{graphicx}
\usepackage{hyperref}
\usepackage{tikz} \usetikzlibrary{arrows,backgrounds}

\newtheorem{definition}{Definition}[section]
\newtheorem{theorem}{Theorem}[section]
\newtheorem{proposition}{Proposition}[section]
 
\newtheorem{corollary}{Corollary}[section]
\newtheorem{remark}{Remark}[section]
\usepackage{enumitem}

\allowdisplaybreaks \numberwithin{equation}{section}

\newcommand{\R}{\mathbb{R}} \newcommand{\N}{\mathbb{N}}

 \newcommand{\fhi}{\varphi}

 \newcommand{\pt}{\partial_t}
 \newcommand{\px}{\partial_x }

 \newcommand{\F}{{\mathcal F}}

\DeclareMathOperator{\diver}{div}
\DeclareMathOperator{\tv}{TV}

\DeclareMathOperator*{\argmax}{arg\,max}
\DeclareMathOperator{\supp}{Supp}

{%

  \begin{enumerate}}%
  {\end{enumerate}}

{%

  \begin{enumerate}}%
  {\end{enumerate}}

\begin{document}

\title[On the optimization of conservation law models
%traffic flow 
at a junction]
{On the optimization of  conservation law models %\\
at a junction with inflow and flow distribution controls}

\author{F.~Ancona} \address[Fabio Ancona]{\newline Dipartimento
  di Matematica ``Tullio Levi-Civita'', Universit\`a di Padova, Via Trieste 63, 35121 Padova, Italy.}  
  \email[]{ancona@math.unipd.it}
\urladdr{http://www.math.unipd.it/~ancona/}

\author{A.~Cesaroni} \address[Annalisa Cesaroni]{\newline Dipartimento
  di Scienze Statistiche, Universit\`a di Padova, Via Cesare Battisti  141, 35121 Padova, Italy.}  
  \email[]{annalisa.cesaroni@unipd.it}
\urladdr{http://homes.stat.unipd.it/annalisacesaroni/}

\author{G. M. Coclite} \address[Giuseppe Maria Coclite]{\newline
  Dipartimento di Meccanica, Matematica e Management, Politecnico di Bari, Via E.~Orabona 4,  I--70125 Bari, Italy.}  
  \email[]{giuseppemaria.coclite@poliba.it}
\urladdr{https://sites.google.com/site/coclitegm/}

\author{M. Garavello} \address[Mauro Garavello]{\newline Dipartimento
  di Matematica e Applicazioni, Università di Milano Bicocca,
  Via  R. Cozzi 55, I--20125 Milano, Italy.}
\email[]{mauro.garavello@unimib.it}
\urladdr{http://www.matapp.unimib.it/~garavello}

\date{\today}

 \subjclass[2010]{35F25, 35L65, 90B20}

 \keywords{Conservation laws, traffic models, networks, weak solutions.}

 \thanks{The authors are members of the Gruppo Nazionale per l'Analisi Matematica, la Probabilit\`a e le loro Applicazioni (GNAMPA) of the Istituto Nazionale di Alta Matematica (INdAM). The authors were partially supported by PRAT 2013- Traffic Flow on Networks: Analysis and Control, University of Padova}

 \begin{abstract}
 The paper proposes a general framework to analyze control problems 
 for conservation law models on a network.
 Namely we consider a 
 general class of junction distribution controls and inflow controls
 and we establish the compactness in $\mathbf{L}^1$ of a class of flux-traces of solutions. 
 We then derive the existence of solutions for two optimization problems: (I) the maximization of 
 an integral functional depending on the flux-traces of solutions evaluated at points of the incoming and outgoing edges;
 (II) the minimization of the total variation of the  optimal solutions of problem (I). 
 Finally we provide an equivalent variational formulation of the min-max problem (II) and we discuss
 some numerical simulations  for a junction with two incoming and two outgoing edges.
 \end{abstract}

\maketitle

% \tableofcontents
\section{Introduction}

Fluid-dynamic models  of traffic flow on networks,
based on conservation laws, have been intensively investigated 
in the last twenty years.
For a general introduction
% and a survey of recent literature
we refer to~\cite{bcghp, g-h-p, gp2, ho-bo}. 
We recall that the idea of modeling unidirectional car traffic on a single road 
in terms of the scalar conservation law 
\begin{equation}
  \label{conlaw}
  \pt u +\px f(u)=0 
\end{equation}
was first proposed in the seminal papers by Lighthill, Whitham and 
Richards (LWR model)~\cite{li-whi,ri}. Here, the unknown $u=u(t,x)$ denotes the traffic density 
taking values 
in a compact interval~$\Omega$, and the flux has the form $f(u) = u\,v(u)$, where $v(u)$ is the average velocity of cars which is assumed to depend on the density alone.

\subsection*{Setting of the problem}
In order to determine the evolution of vehicular traffic in an entire network of roads  modeled by a directed
graph, one has to further assign a set of suitable boundary conditions at road intersections.
Due to finite propagation speed, it will be sufficient to analyze the local solution in a neighborhood of
each intersection to capture the global behavior on the whole network.
To fix the ideas, let us consider 
a graph composed by a single vertex (or node), 
with $m$
incoming edges $I_i$, $i\in\mathcal{I}=\{1,\dots, m\}$, and $n$ outgoing ones $I_j$,
$j\in\mathcal{O}=\{m+1,\dots,m+n\}$.
We may model each incoming edge  with the half-line $(-\infty, 0)$ and
every outgoing one with the half-line $(0, +\infty)$.  
In this way, the junction is  always sitting at  $x=0$.
Denoting with $u_\ell$ a traffic density in $I_\ell$
($\ell \in \mathcal I \cup \mathcal O$), the conservation law~(\ref{conlaw})
on every incoming and outgoing edge must be supplemented
with some initial condition $\bar u_\ell$, which
thus yields the Cauchy problems
\begin{equation}
  \label{incoming} 
  \begin{cases} \pt u_i +\px f(u_i)=0 \qquad & x<0, \, t>0,\\ 
    u_i(0,x)= \overline u_i(x) & x<0,\end{cases}
\end{equation} 
for every $i\in\mathcal{I}$, and
\begin{equation}\label{outgoing}
  \begin{cases} \pt u_j +\px f(u_j)=0  \qquad & x>0, \, t>0,\\ 
    u_j(0,x)= \overline u_j(x) & x>0,\end{cases}
\end{equation} 
for every $j\in\mathcal{O}$,
where we may assume that the initial conditions
$\overline u_\ell: I_\ell\to \Omega$,
$\ell\in\mathcal{I}\cup\mathcal{O}$,
are given $\mathbf{L}^1$-functions with bounded variation.

The equations in~\eqref{incoming}-\eqref{outgoing} are usually coupled through
transition conditions prescribed at the boundary $x=0$ (also called nodal
conditions, coupling conditions, or junction conditions).
Typically, one introduces such conditions to guarantee that~\eqref{incoming}-\eqref{outgoing} 
admits  a unique solution so to show that the Cauchy problem at a junction
 is well-posed. 
Thus, in particular, every $(m\!+\!n)$-tuple
 of initial data $\overline u_\ell, \ell\in\mathcal{I}\cup\mathcal{O}$
 determines a unique $(m\!+\!n)$-tuple of incoming and outgoing fluxes
$g_\ell\doteq f(u_\ell(\cdot, 0)), \ell\in\mathcal{I}\cup\mathcal{O}$ of the corresponding solution
to~\eqref{incoming}-\eqref{outgoing}.
In this paper, instead, we consider a minimal set of natural coupling conditions  
and treat as control laws both the flow distribution parameters related to such
conditions and  the incoming fluxes at the junction (the {\it inflows})  that are compatible with 
them.
Our intention here is to propose a general set-up to analyze optimal control problems 
on networks
with cost functionals
depending on flow distribution controls and inflow junction controls.

\subsection*{Review of the literature}
Optimization and control issues for 
  network models based on conservation laws have raised an increasing interest
  in the last decade,
motivated by a wide range of applications in various 
research fields. In fact, beside vehicular traffic
(see~\cite{
ba-ha-co-da,bk2,cgr,dm-p-r,ghkl,g-k-l-w,JdWK1,JdWK2,rkdmsg} and references therein), such kind of problems have been addressed
in: air traffic~\cite{brt,mi-ba-to}, 
supply chains~\cite{Da-G-H-P, Da-M-H-P, hfy, LM-A-H-R}, irrigation channels~\cite{dh-pr-co-da-ba}, 
gas pipelines~\cite{cghs},
telecommunication and data~\cite{cmpr}, bio-medical~\cite{bcghp,CoGa}, 
blood circulation~\cite{ca-d-d-m}, 
socio-economical 
and other areas.
In these contexts, a crucial role is played by the design, analysis 
and numerical implementation of 
controls acting at the nodes of the network.
The investigation of optimal control properties of time varying parameters corresponding to junction distribution coefficients
 have been considered for example in~\cite{cdpr1,cdpr2,cmpr,gp3,ghkl,hkk,hk,odmgk} 
while inflow controls have been analyzed in~\cite{cgr,Da-M-H-P,ghz,li-ca-cl}.

To introduce our control framework
let's focus again on the LWR model. 
The 
transition conditions at a junction, in a realistic car traffic model, are determined by: (i)~{\it drivers' turning preferences}  and 
(ii)~{\it relative priority assigned to different incoming roads}.

As described in~\cite{CGP, gp2,  hr}, 
the nodal condition (i)  can be expressed requiring that the flux traces 
of the solutions to~\eqref{incoming}-\eqref{outgoing} at $x=0$ satisfy some, possibly time-varying,
distribution rules. Namely, consider a  $n\times m$ Markov matrix $A(t)=(a_{ji}(t))_{j,i}$, with
\begin{equation}
  \label{distr-rules}
  \qquad
  0\le a_{ji}(t)\le 1\quad\forall~j\in\mathcal{O},i\in\mathcal{I},t>0,\qquad \quad \sum_{j=m+1}^{m+n} a_{ji}(t)=1
  \qquad\forall~i\in\mathcal{I},t>0,
\end{equation}
and impose
the condition
\begin{equation}
  \label{distribution}
  f (u_j(t,0^+)) = \sum_{i=1}^m a_{ji}(t) f(u_i(t,0^-))
  \qquad\textrm{for a.e.} \ \ t>0,
  \qquad \forall\, j \in \mathcal O,
\end{equation}
where $u_\ell(t,0^\pm)\doteq \lim_{x\to 0^\pm} u_\ell(t,x)$. 
Throughout the following we shall simply adopt the notation $u_\ell(t,0)$
for the one-sided 
limit of $u_\ell(t,\cdot)$ at the boundary $x=0$.
Here $a_{ji}(t)$ represents the
fraction of drivers arriving from the $i$-th incoming road that wish to turn
on the $j$-th outgoing road at time~$t$. Notice that, because
of~\eqref{distr-rules}, the condition~\eqref{distribution} in
particular implies a Kirchhoff type formula
\begin{equation*}%\label{conservation}
  \sum_{j=m+1}^{m+n}f(u_j(t,0))=\sum_{i=1}^m f(u_i(t,0))
  \qquad\textrm{for a.e.} \ \ t>0,
\end{equation*}
which expresses the conservation of the total flux of cars through the
node.
Instead, the nodal condition (ii) is expressed in~\cite{CGP, gp2} assigning right of way parameters
$\eta_i(t)\in [0,1]$, $i\in\mathcal{I}$, with $\sum_i\eta_i (t)=1$. For example, in
an intersection
regulated by a traffic signal, the coefficient $\eta_i$ can be interpreted as the
fraction of time in which it is given a green light to cars arriving from the $i$-th road.
\medskip

However, 
the nodal conditions~(i)-(ii) above described in general are satisfied by
infinite many solutions to~\eqref{incoming}-\eqref{outgoing}.
% are not sufficient to select a unique solution to~\eqref{incoming}-\eqref{outgoing}. 
Various approaches have been proposed in the literature to overcome this ill-posedness 
of the Cauchy problem.
A widely used method is based on the introduction of a Junction Riemann Solver, i.e. a  rule 
to construct a unique self-similar solution of the Cauchy problem~\eqref{incoming}-\eqref{outgoing}
when the initial data is constant on each incoming and outgoing edge. 
Once a Junction Riemann Solver
is given, the solution of 
the Cauchy problem at a junction for general initial data is  then 
constructed by 
a standard wave-front tracking technique~\cite{CGP, gp2, gp4, hr}. 
A~different model, which is not based on the construction of a Junction Riemann Solver, was proposed in~\cite{bk1}.
Here,  cars seeking to enter a congested road wait in a buffer of limited capacity and 
 the distribution coefficients $a_{ji}(t)$ in~\eqref{distribution} are regarded
as the boundary values at $x=0$ of passive scalars transported by a semilinear equation coupled with 
the LWR equation~\eqref{conlaw}. In this model the unique solution of 
a general Cauchy problem~\eqref{incoming}-\eqref{outgoing} is obtained 
employing an extension of the Lax-Ole\v{\i}nik formula to the initial 
boundary value problem and determining the
length of queues inside a buffer as the fixed point of a contractive transformation.
In fact, it is shown  in~\cite{br-no} that a specific  Riemann Solver at the junction determines the limiting solution
obtained in presence of a buffer when the buffer's capacity approaches zero.
Instead, a Junction Riemann Solver for 
a source destination model that incorporates a dynamic description of the car path choices
was previously introduced in~\cite{gp5}.

We point out that alternative approaches to establish a well-posedness theory of traffic flow on networks, 
based on the analysis of Hamilton-Jacobi equations at a junction,
 have been developed in the last years, using either PDE methods (see the papers  ~\cite{im,l-s}  
 and references therein) or optimal control interpretation (see 
 \cite{at} and references therein).  
In~\cite{acd, cg} it is pursued an analysis of the well-posedness  of~\eqref{incoming}-\eqref{outgoing} based
on a vanishing viscosity approximation and adopting a similar framework
 of the theory of germs introduced in~\cite{akr} for conservation laws with discontinuous flux. 
Additional models can be found in~\cite{g-h-p}. 
We refer also to~\cite{dag, leb} for traffic engineering references.

\subsection*{The Riemann solver approach}
We remark that all the models without buffer proposed in the literature
provide a procedure to associate to any given $(m+n)$-tuple
of initial conditions $\overline u_\ell,\, \ell\in\mathcal{I}\cup\mathcal{O}$,
a unique $(m+n)$-tuple of boundary values $\widetilde u_\ell,\, \ell\in\mathcal{I}\cup\mathcal{O}$,
so that the resulting solutions $u_\ell(t,x)$, $(t,x)\in  (0,+\infty)\times I_\ell$
 of the  Cauchy-Dirichlet  problems
with initial data $\overline u_\ell$   and boundary data $\widetilde u_\ell$
satisfy the nodal constraint~\eqref{distribution}.
In this paper we adopt a different perspective following a control theoretic approach.
Namely, we propose a general framework to treat optimal control problems at a junction
where one  regards as junction controls the distribution coefficients $a_{ji}(t)$
and the incoming fluxes $ f(u_i(t,0))$
compatible with the transition condition~\eqref{distribution}.

To illustrate our control strategy let's recall the standard construction performed
for a Riemann problem at a junction with a constant matrix distribution  $A=(a_{ji})_{j,i}$  satisfying~\eqref{distr-rules}. Namely, consider the  Cauchy problem
\begin{equation}
  \label{RP-node}
  \begin{cases} \pt u_\ell +\px f(u_\ell)=0 \qquad & x\in I_\ell, \, t>0,\\ 
    u_\ell(0,x)= \overline u_\ell & x\in I_\ell,
  \end{cases}
  \qquad  
  \ell=1,\dots,m+n,
\end{equation}
where $\overline u_1,\dots, \overline u_{m+n}\in \Omega$ are constants.
In connection with every $\overline u_i$, $i\in\mathcal{I}$,
we define
the set $ \Omega_{\overline u_i}^{\scriptscriptstyle{<0}}$ consisting
of all states $\widetilde u_i\in \Omega$ for which the entropy
admissible solution of the classical Riemann problem on the whole real line
\begin{equation}
  \label{RP-classical-1}
  \nonumber
  \left\{ 
    \begin{aligned}
      &\pt u_i +\px f(u_i)=0\qquad  x\in \R, \, t>0,\\
      \noalign{\smallskip} &u_i(0,x)=
      \begin{cases}
        \overline u_i\quad &\textrm{if}\quad\ x<0, \\
        \widetilde u_i\quad &\textrm{if}\quad\ x>0,
      \end{cases}
    \end{aligned}
  \right.
\end{equation}
contains only waves with nonpositive characteristic speeds.
Similarly, for every $\overline u_j$, $j\in\mathcal{O}$,
we let $\Omega_{\overline u_j}^{\scriptscriptstyle{>0}}$ denote the set of
all states $\widetilde u_j\in \Omega$ for which the entropy admissible
solution of the classical Riemann problem
\begin{equation}
  \label{RP-classical-2}
  \nonumber
  \left\{ 
    \begin{aligned}
      &\pt u_j +\px f(u_j)=0\qquad  x\in \R, \, t>0,\\
      \noalign{\smallskip} &u_j(0,x)=
      \begin{cases}
        \widetilde u_j \quad &\textrm{if}\quad\ x<0, \\
        \overline u_j\quad &\textrm{if}\quad\ x>0,
      \end{cases}
    \end{aligned}
  \right.
\end{equation}
contains only waves with nonnegative characteristic speeds.
Then, given
$\overline u =(\overline u_1, \dots, \overline u_{m+n})\in
 \Omega^{m+n}$, consider the closed, convex, not empty set
\begin{equation}
  \label{Gamma-def}
  \Gamma_{\overline u}\doteq
  \left\{ (\gamma_1,\dots,\gamma_m)\in\prod_{i=1}^m
  f( \Omega_{\overline u_i}^{\scriptscriptstyle{<0}})\ : \
  \big(A \cdot (\gamma_1,\dots,\gamma_m)^T\big)^{\!T}\in 
  \prod_{j=m+1}^{m+n}
  f( \Omega_{\overline u_j}^{\scriptscriptstyle{>0}})
  \right\},
\end{equation}
where  $\gamma^T$
denotes the (column) transpose vector of $\gamma$. 
Clearly $\left(0, \cdots, 0\right) \in \Gamma_{\overline u}$ and so
$\Gamma_{\overline u}$ is not empty. Moreover it is closed since it is
the preimage of a closed set through a continuous function, and the convexity
follows from the linearity of the function associated to $A$.
The set
$\Gamma_{\overline u}$ describes the flux-traces at $x=0$ on the incoming
roads $I_i$, $i=1,\dots,m$, of every self-similar entropy admissible
solution to the nodal Riemann problems~\eqref{RP-node} that satisfies
the conditions~\eqref{distribution} for a constant $A$.
In general, the set $\Gamma_{\overline u}$ in~\eqref{Gamma-def} contains more than
one point so, to achieve
uniqueness, an
optimization criterion is usually imposed
for example requiring the maximization of the total flux through
the junction~\cite{CGP, gp2, gp4, hr}. 
An alternative approach to identify a unique Riemann solver at the junction is proposed in~\cite{acd}
in the same spirit of  the theory of germs for conservation laws with discontinuous flux.

Once it is provided a procedure to associate to any $\overline u \in \Omega^{m+n}$
a unique element $\gamma_{\overline u} \in \Gamma_{\overline u}$
that identifies the boundary flux-traces of the solution of~\eqref{RP-node}
selected by the chosen admissibility criterion, 
one can require that
the solution $u(t,x)$ of a general Cauchy
problem at the junction~\eqref{incoming}-\eqref{outgoing}
satisfies the nodal condition
\begin{equation*}%\label{distribution-2}
  \big(f (u_1(t,0)),\dots, f(u_m(t,0))\big)=\gamma_{u(t,0)}
  \in\Gamma_{u(t,0)}
  \qquad\forall~t>0,
\end{equation*}
where 
\[u(t,0)=(u_1(t,0),\dots, u_m(t,0),u_{m+1}(t,0),\dots,  u_{m+n}(t,0))\]  
denotes the
$(m+n)$-tuple of traces at $x=0$ of $u(t,\cdot)$.

\subsection*{Our approach and main result}
In the present paper instead, aiming to perform a control theoretic analysis,
 we consider as {\it admissible solutions of the nodal Cauchy Problem}~\eqref{incoming}-\eqref{outgoing} 
every function $u \doteq (u_\mathcal{I},\, u_\mathcal{O})$
with $u_\mathcal{I}: [0,+\infty)\times [0,+\infty) \to \Omega^m$,
$u_\mathcal{O}: [0,+\infty)\times (-\infty, 0] \to \Omega^n$, 
that for some  $k_\ell:  [0,+\infty) \to \Omega$,
$\ell\in\mathcal{I}\cup\mathcal{O}$, provides an entropy admissible solution of the mixed initial-boundary value problems
\begin{equation}
  \label{bdry-incoming} 
  \begin{cases} \pt u_i +\px f(u_i)=0 \qquad & x<0, \, t>0,\\ 
    u_i(0,x)= \overline u_i(x) & x<0,\\
    u_i(t,0)= k_i(t) & t>0,\end{cases}
    \qquad  i\in\mathcal{I},
\end{equation} 
\begin{equation}
  \label{bdry-outgoing} 
  \begin{cases} \pt u_j +\px f(u_j)=0 \qquad & x>0, \, t>0,\\ 
    u_j(0,x)= \overline u_j(x) & x>0,\\
    u_j(t,0)= k_j(t) & t>0,\end{cases}
    \qquad  j\in\mathcal{O},
\end{equation} 
and satisfies the linear constraint~\eqref{distribution}
(see Definition~\ref{defi}).
Observe that, in general, the boundary data $k_\ell,\, \ell\in\mathcal{I}\cup\mathcal{O}$ are not pointwise attained  by the
entropy admissible solutions of the Cauchy-Dirichlet problems~\eqref{bdry-incoming}, \eqref{bdry-outgoing}
because of the presence of boundary layers.
In fact, for conservation laws the boundary conditions are
enforced in the weak sense of Bardos, Le Roux, N\'ed\'elec~\cite{BLN}.
 Moreover, by uniqueness, the solutions to~\eqref{bdry-incoming}, \eqref{bdry-outgoing} can be equally
 defined as the entropy admissible solutions of the corresponding mixed problems with assigned initial datum~$\overline u_\ell$
 and boundary flux $f(u_\ell(t,0))$ (cfr.~\cite[Proof of Theorem 2.2]{lf}).
Therefore, we may uniquely identify every admissible solutions of the nodal Cauchy Problem~\eqref{incoming}-\eqref{outgoing}
assigning the matrix valued map $A(\cdot)$ satisfying~\eqref{distr-rules}
and the $m$-tuple of boundary incoming flux-traces 
$\gamma=(f (u_1(\cdot,0)),\dots, f(u_m(\cdot,0))\big)$
which, in turn, by~\eqref{distribution} determines also the $n$-tuple of boundary outgoing flux-traces
$(f (u_{m+1}(\cdot,0)),\dots, f(u_{m+n}(\cdot,0))\big)$.
We shall denote by $u(t,x;\, A,\gamma)$ such a solution
where $\gamma$ denotes an $m$-tuple of {\it admissible boundary incoming flux-traces}
(see Definition~\ref{def-adm-bdr-flux}).
 Then, for a fixed initial datum $\overline u\in  \Pi_{\ell=1}^{m+n} \ \mathbf{BV}(I_\ell;  \Omega)$,
 and for a fixed time~$T>0$,
 we will address the following optimization problems.
\begin{itemize}
\item[{\bf (I)}] Given a continuous map $\mathcal{J}: \R^m\to\R$, fix $M>0$
and consider 
\begin{equation}
\label{opt-pb1}
  \sup_{\, A,\,\gamma\ \, }  \int_0^T \mathcal{J}\big(f(u_1(t,0;\, A,\gamma)),\dots, f(u_m(t,0;\, A,\gamma))\big)dt\,,
\end{equation}
where the supremum is taken over all pairs of 
matrix valued maps $A(\cdot)=(a_{ji}(\cdot))_{j,i}\in   \mathbf{BV}([0,T];\R^{m\times n})$ 
satisfying~\eqref{distr-rules}
and  admissible boundary flux-traces $\gamma=(\gamma_1, \dots,\gamma_{m})
%=(f(k_1), \dots, f(k_{m+n}))
\in \mathbf{BV}\big([0,T];  \big(f(\Omega)\big)^{\!m}\big)$,
with  total variation bounded by
$$
\tv\{a_{ji}\}\leq M\quad\forall j\in\mathcal{O}, i\in\mathcal{I}, \qquad\quad \tv\{\gamma_i\}\leq M
\quad\forall i\in \mathcal{I}.
$$ 
\smallskip
\item[{\bf (II)}] The pairs of junction controls $(\widehat A,\,\widehat\gamma\,)$ that optimize~\eqref{opt-pb1}
are in general not unique, as shown in Section~\ref{sec:opt}. 
 Let $\mathcal{U}^M_{max}$ denote the set of such optimal pairs.
 In order to restrict the set of optimal solutions,
 we then consider the minimization problem
\begin{equation*}
  % \label{opt-pb2}
  \inf_{\, (\widehat A(\cdot),\,\widehat\gamma\,)\,\in\,\mathcal{U}^{^{M}}_{max}\, } 
  \sum_{i=1}^m  \tv \left\{f\big(u_i\big(\cdot, 0;\, \widehat A,
    \widehat \gamma\big)\big)\right\}.
\end{equation*}
\end{itemize}
\smallskip
The  goal of the above  problems is to maximize suitable functionals depending on the through flux 
at the junction in a fixed time interval (for example the time-average sum or product of the inflows within a day),
keeping as small as possible the oscillation of the incoming fluxes.
The control parameters determine the percentage of drivers who take the different roads emerging
from the junction ({\it traffic distribution controls}) and regulate the rate at which the vehicles pass through the junction 
({\it inflow controls}). The former can be practically implemented with the use of 
route information panels
giving recommendations to the drivers
 to take one of the outgoing roads, while the latter can be enforced by  ramp meter signals, 
traffic light timing, 
yielding and stop signs, integrated vehicular and roadside sensors so to modulate
the amount of incoming fluxes entering the junction. 
These type of control approaches  are of critical importance in 
traffic management since they allow to improve the performance of traffic system, alleviate congestion, 
reduce pollution and accidents without requiring expensive road constructions to increase road capacity.

%\medskip
Notice that,  an admissible solution of the junction Cauchy problem~\eqref{incoming}-\eqref{outgoing}
defined as above
may exhibit entropy admissible shocks originated at positive times from a point of the line $x=0$ and stationary nonclassical shocks 
(see~\cite{LFnc})
located at $x=0$. Moreover, an admissible solution of a junction Riemann problem may well  be not self-similar, in the sense that it is not constant
along line exiting from the origin.
We may observe that  such features
of the admissible solutions 
considered in the present paper
occur in various 
 models of conservation laws arising in different contexts: like in the case of discontinuous flux~\cite{AMVG1, an-ca}, 
of local~\cite{a-g-s, CoGo} or nonlocal~\cite{a-d-r} constraints,
of LWR model (or also higher order models)
coupled with ODE modeling  a buffer dynamics~\cite{dm-r-s-k-g-b}.

%\medskip
Our main contributions can be summarized as follows.
We shall first establish a compactness result
 for some classes of flux-traces of solutions, which in turn yields
the existence of solutions to the optimization problems {\bf (I)-(II)}
in their general setting.
 
Next, we provide an equivalent variational formulation of the optimization problem {\bf (II)}
which is useful also for numerical investigations of the optimal solutions. 
Finally, we analyze the optimization problem~\eqref{opt-pb1} 
in the case where the supremum is taken over all pairs of 
matrix valued maps $A(\cdot)$ and  admissible boundary flux-traces with arbitrarily large total variation.

%\medskip

We stress that, in our intention,
the approach pursued
 in this paper is a first step towards building a general  strategy
to address control problems on networks, which can be applied in various contexts 
not limited to vehicular traffic models.

\subsection*{Organization of the paper}
Section~\ref{sec:preliminaries} is
devoted to preliminary results on boundary value problems for
conservation laws and on the Cauchy problem for a junction. In Section~\ref{sec:opt} we 
derive the compactness properties for flux-traces of solutions and
analyze the optimization problems {\bf (I)-(II)}.
Finally, in Section~\ref{sec:num}, we discuss
some numerical experiments 
for a junction with two incoming and two outgoing arcs.
%\pagebreak

\section{Preliminaries}
\label{sec:preliminaries}

\subsection{The Dirichlet and the junction
  Cauchy problem for conservation laws}
\label{dir-jct-problems}
Consider a directed graph composed by a single vertex, located at $x=0$,
with $m$ incoming arcs 
$I_i=(-\infty, 0)$, $i\in\mathcal{I}=\{1,\dots, m\}$,
and $n$ outgoing ones $I_j=(0, +\infty)$,
$j\in\mathcal{O}=\{m+1,\dots,m+n\}$.
On each arc, the evolution of the unknown density  $u_\ell(t,x)$ is governed by the
%we consider the evolution of the density of cars $u_\ell(t,x)$ governed by the
scalar conservation law 
\begin{equation}
  \label{conlaw-l}
  \pt u_\ell +\px f(u_\ell)=0,\qquad\quad t\geq 0, \ x \in I_\ell\,.
\end{equation}
The flux function $f:\Omega\to\R$
is defined on a compact interval $\Omega = [0, u^{max}]$
and satisfies the following assumptions:
\begin{enumerate}[label=\textbf{(A)}, align=left]
\item \label{hyp:A}
  \begin{enumerate}[label=\arabic*., align=left]
  \item $f \in \mathbf{C}^1\left(\Omega; \R\right)$ and is strictly
    concave.

  \item $f(0)=f(u^{max})=0$, \ $f'(u^{max})<0<f'(0)$.
  \end{enumerate}
\end{enumerate}

Here $u^{max} > 0$ denotes the maximal possible density inside a road.
We denote by $\theta\in\Omega$ the point of global maximum for $f$, i.e. 
\begin{equation*}
  % \label{theta-def}
  \theta\doteq \argmax_{u\in\Omega} f(u)\,,
%f(\theta)=\max_{p\in \Omega} f(p).
\end{equation*}
and, for every
$u \in \Omega \setminus \left\{\theta\right\}$,
we denote by $\pi(u)$ the unique point in $\Omega$ such that
\begin{equation*}
  % \label{[iu-def}
  f(u)=f(\pi(u))
  \qquad \textrm{ and } \qquad
    \pi(u) \ne u
  \,,
\end{equation*}
while we set  $\pi(\theta) \doteq \theta$.
%Moreover we denote as $f_-, f_+$ the restrictions of $f$ to the intervals $[0, \theta]$
%and $[\theta, u^{max}]$, respectively.
Given an $(m+n)$-tuple of initial data 
$\overline u\in  \Pi_{\ell=1}^{m+n} \  \mathbf{L^\infty}(I_\ell;  \Omega)$
%\mathbf{BV} (I_\ell;  \Omega)$
and boundary data 
$k\in \mathbf{L^\infty} \big([0,+\infty);  \Omega^{m\!+\!n}\big)$,
consider, for every
$i=1,\dots, m$, $j=m+1,\dots,m+n$,
the mixed initial-boundary value problems
\begin{equation}
\label{eq:conlaw1in}
  \begin{cases} \pt u_i +\px f(u_i)=0  \qquad & x<0, \, t>0,\\ 
    u_i(0,x)= \overline u_i(x) & x<0,\\
    u_i(t,0)= k_i(t) & t>0,\end{cases}
\end{equation} 
\begin{equation}
 \label{eq:conlaw1out}
  \begin{cases} \pt u_j +\px f(u_j)=0 \qquad & x>0, \, t>0,\\ 
    u_j(0,x)= \overline u_j(x) & x>0.\\
    u_j(t,0)= k_j(t) & t>0.\end{cases}
\end{equation} 
\begin{remark}
  For simplicity we decide to use the same flux function on each edge of the
  junction, but it is not a restrictive assumption. For real applications,
  it is natural to consider different flux functions on each road of
  the network. It is straightforward the extension of all the results of this
  paper to such a case, provided that each flux function satisfies the
  assumption~\ref{hyp:A}, with possibly different $u^{max}$.
\end{remark}
As observed in the Introduction, the Dirichlet conditions in~\eqref{eq:conlaw1in}, \eqref{eq:conlaw1out}
are, in general, not literally satisfied 
but must be interpreted in the relaxed sense 
of Bardos, Le Roux, N\'ed\'elec~\cite{BLN}. 
In particular, since we assume the strict concavity of the flux function $f$, 
 we may equivalently express 
the  boundary condition in~\eqref{eq:conlaw1in} requiring as  in~\cite{lf} that, for a.e. $t>0$,
there holds:
\begin{equation}
\label{Dir-bc-in}
  \begin{array}{c@{\qquad}c}
    \textrm{either} &  u_i(t,0)=\max\{k_i(t), \theta\}
    \vspace{.1cm}
    \\
    \noalign{\smallskip}
    \textrm{or}  &  \quad f'(u_i(t,0))\geq 0  \   \
    \textrm{ and } \ f(u_i(t,0))\leq f\big(\max\{k_i(t), \theta\}\big),
  \end{array}
\end{equation} 
while the boundary condition in~\eqref{eq:conlaw1out} are equivalent to require that, for a.e. $t>0$,
there holds:
\begin{equation}
\label{Dir-bc-out}
  \begin{array}{c@{\qquad}c}
    \textrm{either} &  u_j(t,0)=\min\{k_j(t), \theta\}
    \vspace{.1cm}
    \\
    \noalign{\smallskip}
    \textrm{or}  &  \quad f'(u_j(t,0))\leq 0  \   \
    \textrm{ and } \ f(u_j(t,0))\leq f\big(\min\{k_i(t), \theta\}\big).
  \end{array}
\end{equation} 
Moreover, one can provide an equivalent formulation of~\eqref{Dir-bc-in}-\eqref{Dir-bc-out} in terms of boundary flux conditions.
Namely, the boundary condition for the problem~\eqref{eq:conlaw1in} on the negative half-line 
can be expressed as
\begin{equation}
\label{Dir-bc-in-flux}
\begin{aligned}
 &\textrm{either}   \qquad\ \  f'(u_i(t,0))< 0\   \
    \textrm{ and } \ f'(k_i(t))<0, \quad f(u_i(t,0))=f(k_i(t)),
    \vspace{.1cm}
    \\
    &\ \ \ \textrm{or}     \qquad\quad\ f'(u_i(t,0))\geq 0  \   \
    \textrm{ and } \ f'(k_i(t))> 0,
     \vspace{.1cm}
    \\
    &\ \ \ \textrm{or}     \qquad\quad\ f'(u_i(t,0))\geq 0  \   \
    \textrm{ and } \  f'(k_i(t))\leq 0, \ \  f(u_i(t,0))\leq f(k_i(t)),
\end{aligned}
\end{equation}
while the boundary condition for the problem~\eqref{eq:conlaw1out} on the positive half-line 
can be expressed as
\begin{equation}
\label{Dir-bc-out-flux}
\begin{aligned}
 &\textrm{either}   \qquad\ \  f'(u_j(t,0))> 0\   \
    \textrm{ and } \ f'(k_j(t))>0, \quad f(u_j(t,0))=f(k_j(t)),
    \vspace{.1cm}
    \\
    &\ \ \ \textrm{or}     \qquad\quad\ f'(u_j(t,0))\leq 0  \   \
    \textrm{ and } \ f'(k_j(t))< 0,
     \vspace{.1cm}
    \\
    &\ \ \ \textrm{or}     \qquad\quad\ f'(u_j(t,0))\leq 0  \   \
    \textrm{ and } \  f'(k_j(t))\geq 0, \ \  f(u_j(t,0))\leq f(k_j(t)).
\end{aligned}
\end{equation}
As stated in the Introduction, $u_\ell(t,0)$ denotes the one-sided  
limit  at  $x=0$ of the solution  $u_\ell(t,\cdot)$ to ~\eqref{eq:conlaw1in}, \eqref{eq:conlaw1out}.
We recall that 
functions of one variable with locally bounded variation admit left
and right limits at every point. Moreover,
since an element  of $\mathbf{BV_{\rm loc}}(I_\ell)$
is an equivalence class of locally integrable functions,
we will always assume that a function $u\in \mathbf{BV_{\rm loc}}(I_\ell)$ 
is left continuous if $\ell\in\mathcal{I}$,
right continuous  if $\ell\in\mathcal{O}$,
possibly modifying its values at its countably
many discontinuity points (see \cite[Chapter 2, Lemma 2.1]{bressan-book}).
Here and throughout the paper, 
%as in Section~\ref{sec:preliminaries}, 
for a function $\psi\in\mathbf{BV}([\tau',\tau''])$
%$\psi\in \big(\mathbf{BV}\cap  \mathbf{L}^1\big)([\tau',\tau''])$ 
%of the time variable, 
we shall always 
define $\tv\{\psi\}$ as the {\it essential variation} of $\psi$ on the open interval $(\tau',\tau'')$
% (pointwise) variation on $[\tau',\tau'']$
which coincides with the {\it pointwise variation} on $(\tau',\tau'')$
of the right continuous or left continuous representative of~$\psi$
(see~\cite[Section 3.2]{AFP}). This implies that, if
$\psi$ is a right continuous (or left continuous) function, there holds
\begin{equation}
  \label{tv-def}
  \tv\{\psi\} =
  pV(\psi)\doteq \sup_{\tau'< t_0<t_1<\cdots<t_N< \tau''}
  \left\{\sum_{\ell=1}^N |\psi(t_\ell) - \psi(t_{\ell-1})|\right\},
\end{equation}
where $pV(\psi)$ denotes the pointwise variation of $\psi$.
%For the total variation on a subinterval $[\tau_1,\tau_2]\subset [0,T]$
For a  function $\psi\in \mathbf{BV}([0,T])$, 
%$\psi\in \big(\mathbf{BV}\cap  \mathbf{L}^1\big)([0,T])$, 
and a subinterval $(\tau',\tau'')\subset [0,T]$, we shall also denote
as $\tv_{(\tau',\tau'')}\{\psi\} \doteq\tv\big\{\psi\big|_{(\tau',\tau'')}\big\}$
the pointwise variation
of the restriction on $(\tau',\tau'')$ of the
right continuous (or left continuous) representative of~$\psi$.
% taken in consideration.

\begin{remark}
  \label{rmk:tv_right_cont}
  Consider a function $\psi\in \mathbf{BV}([0,T])$, let $\psi_1$ be a left continuous
  (or right continuous) representative of $\psi$, and consider any other element $\psi_2$
  of the same equivalent class
  in $\mathbf{L^1}\left((0,T)\right)$ of $\psi$, i.e. such that $\psi_1(t) = \psi_2(t)$ for a.e. 
  $t \in [0,T]$.
  Then, one has
  \begin{equation*}
  pV(\psi_1)\leq pV(\psi_2)\,.
%    \sup_{\tau'= t_0<t_1<\cdots<t_N= \tau''}
%    \left\{\sum_{\ell=1}^N |\psi_1(t_\ell) - \psi_1(t_{\ell-1})|\right\}
%    \le
%    \sup_{\tau'= t_0<t_1<\cdots<t_N= \tau''}
%    \left\{\sum_{\ell=1}^N |\psi_2(t_\ell) - \psi_2(t_{\ell-1})|\right\}.
  \end{equation*}
  In fact, the pointwise variation defined in~(\ref{tv-def}) clearly depends
  on the choice of the 
  representative of $\psi$ and the infimum is achieved by the 
  left continuous or right continuous representatives of  $\psi$ (cfr.~\cite[Section 3.2]{AFP}).
\end{remark}
We shall then adopt the following 
definition  of entropy admissible solution of the mixed initial-boundary
value problem for a conservation law~\eqref{conlaw-l} with concave flux
(see~\cite{BLN, cr, daf-book,kruz,lax,Ot,vo}).
\begin{definition}
  \label{def-ent-sol-ibvp}
  Given $\overline u_\ell\in  \mathbf{L^\infty} (I_\ell;  \Omega)$ and
  $k_\ell\in \mathbf{L^\infty} \big((0,+\infty);  \Omega\big)$,
  $\ell \in\mathcal{I}\cup\mathcal{O}$,
  we say that a function
  $u_\ell\in\mathbf{C}\big([0,+\infty);\, \mathbf{L}^1_{{\rm loc}}
  (I_\ell;  \Omega)\big)$
  is an entropy admissible weak solution
  of~\eqref{eq:conlaw1in} (resp. of~\eqref{eq:conlaw1out}) if :
  \begin{itemize}
  \item[(i)] For every $t > 0$,
    $u_\ell(t) \in \mathbf{BV_{\rm loc}}\left(I_\ell; \Omega\right)$ and admits the one-sided limit at $x=0$.

  \item[(ii)] $u_\ell$ is a weak entropy solution
    of~\eqref{conlaw-l} with initial data $\overline u_\ell$, i.e. for all 
    $\varphi\in\mathcal{C}_c^1\big(\R\times I_\ell;  \R\big)$ there holds
    \begin{equation}
      \label{entr-weak-sol}
      \int_0^{+\infty}\!\!\int_{I_\ell}\Big[u_\ell\,  \pt\varphi + f(u_\ell)\, 
      \px\varphi \Big](t,x)~dxdt + \int_{I_\ell}\overline u_\ell (x)
      \varphi(0,x)dx\, = 0\,,
    \end{equation}
    and the Lax entropy condition~\cite{lax} is satisfied 
    \begin{equation*}
      % \label{lax-entr}
      u_\ell(t,x^-)\leq u_\ell(t,x^+)\qquad\quad
      t>0, \ x \in I_\ell\,.
    \end{equation*}
  \item[(iii)] for a.e. $t>0$ the boundary condition at $x=0$ is
    verified in the sense of~\eqref{Dir-bc-in-flux}
    (resp. of~\eqref{Dir-bc-out-flux}).
  \end{itemize}
\end{definition}

\begin{remark}
  Notice that, because of the assumption {\bf (A)} on the flux function, an entropy admissible weak solution $u_\ell$
  of~\eqref{conlaw-l} satisfies the one-side  Ole\v{\i}nik
inequality on the decay of negative waves~\cite{Oleinik} which, together with the well-known $\mathbf{L^\infty}$ bounds on  $u_\ell$, yields 
 uniform BV-bounds on $u_\ell(t)$ at any fixed time $t>0$.
  %\eqref{eq:conlaw1in} (or of~\eqref{eq:conlaw1out}) 
  Moreover, for every
  $t > 0$ and $\ell \in \mathcal I$,
  $u_\ell(t)$ admits the one-sided limit
  $u_\ell(t, 0)\doteq\lim_{x\to 0^-} u_\ell(t, x)$. 
    Indeed, if 
    %$t > 0$, $\ell \in \mathcal I$ and
 $\lim_{x\to 0^-} u_\ell(t, x)$
 %, $t > 0$, $\ell \in \mathcal I$, 
 does not exist, then one can find two sequences
  $x_n$ and $y_n$ of negative numbers converging to $0$ such that
  \begin{equation*}
    v_1 = \lim_{n \to +\infty} u_\ell \left(t, x_n\right),
    \qquad
    v_2 = \lim_{n \to +\infty} u_\ell \left(t, y_n\right),
    \qquad
    v_1 < v_2.
  \end{equation*}
  Since $f$ is concave, then $f'(v_1) > f'(v_2)$ and so, for $n$ sufficiently
  large, we would deduce that
  $f'\left(u_\ell(t, x_n)\right) > f'\left(u_\ell(t, y_n)\right)$.
  By tracing the backward generalized characteristics~\cite{d,daf-book} starting at $(t,  x_n)$, 
  $(t,  y_n)$, we would then conclude that such
  %using the concept of backward characteristics (see~\cite{d}), we deduce
  %that backwards 
  characteristics intersect in $I_\ell$ at some positive time,
  %inside the domain, 
  which is
  not possible. A similar argument holds for the
  existence of $\lim_{x\to 0^+} u_\ell(t, x)$, $t > 0$, $\ell \in \mathcal O$.
  Therefore, condition \textup{(i)} of Definition~\ref{def-ent-sol-ibvp}
  is indeed a consequence of condition \textup{(ii)}.
  We have included it in the definition for
  sake of clarity.
\end{remark}

\begin{remark}
\label{BV-flux}
For every fixed point $x\in I_\ell$, $\ell\in \mathcal I \cup \mathcal O$,
the flux map $t\mapsto f\big(u_\ell(t, x)\big)$ of an entropy admissible weak solution $u_\ell$
 of~\eqref{conlaw-l}  is  nonincreasing  in presence of shock discontinuities. Hence, since
one may derive Ole\v{\i}nik-type 
inequalities on the positive variation of $f\big(u_\ell(\cdot, x)\big)$
%in presence of rarefaction waves,
as for the negative variation of $u_\ell(t, \cdot )$,
 it follows that $f\big(u_\ell(\cdot, x)\big)\in  \mathbf{BV_{\rm loc}} \left((0, +\infty); \Omega\right)$.
\end{remark}

\noindent
As observed in~\cite{lf}, it is not restrictive to assume that 
the boundary data have characteristics 
 always entering the domain, i.e. that 
\begin{equation}
\label{char-bdry-data}
\begin{aligned}
k_i(t)\geq \theta &\qquad\text{if}\quad\ i\in\mathcal{I},
\\
\noalign{\smallskip}
k_j(t)\leq \theta &\qquad \text{if}\quad\ j\in\mathcal{O}.
\end{aligned}
\end{equation}
In fact, 
%if~\eqref{char-bdry-data} is not satisfied, 
the entropy admissible weak solutions $u_\ell$
of \eqref{eq:conlaw1in}, \eqref{eq:conlaw1out}
with boundary data $k_\ell$, $\ell\in\mathcal{I}\cup\mathcal{O}$,
can be as well obtained 
replacing $k_\ell$ with the {\it normalized boundary data}
\begin{equation*}
  % \label{nornal-lbdry}
  \widetilde k_\ell(t)\doteq 
  \begin{cases}
    \max\big\{u_\ell(t,0),\, \pi(u_\ell(t,0))\big\} & \textrm{if} \quad\ \ell\in\mathcal{I},
    \\
    \noalign{\smallskip}
    \min\big\{u_\ell(t,0),\, \pi(u_\ell(t,0))\big\}& \textrm{if} \quad\ \ell\in\mathcal{O},
  \end{cases}
\end{equation*}
which satisfy~\eqref{char-bdry-data}
and
\begin{equation*}
  % \label{normal-bdry-sol-trace}
  f\big(\,\widetilde k_\ell(t)\big)=f\big(u_\ell(t,0)\big)
  \qquad\forall~t > 0\,.
\end{equation*}
\medskip

\noindent
We next recall  a general property of  weak solutions of~\eqref{conlaw-l}  
that will be useful later. 
\begin{proposition}
  \label{divprop}
  Given $\overline u_\ell\in   \mathbf{BV}(I_\ell;  \Omega)$
  and  $k_\ell\in \mathbf{L^\infty} \big((0,+\infty);  \Omega\big)$,
  $\ell \in\mathcal{I}\cup\mathcal{O}$,
  let  $u_\ell$ be an entropy admissible weak solution of~\eqref{eq:conlaw1in}
  if $\ell \in \mathcal I$
  or \eqref{eq:conlaw1out} if $\ell \in \mathcal O$.
  Then, 
  for every $0 \le t_1 < t_2$, $x_0\in I_\ell$, 
  one has
  \begin{equation}
    \label{div-thm}
    \int_{t_1}^{t_2} f\left(u_\ell(t,0)\right) dt - 
    \int_{t_1}^{t_2} f\left(u_\ell(t,x_0^\pm)\right) dt 
    =  \int_{x_0}^0 u_\ell(t_1,x)dx - \int_{x_0}^0 u_\ell (t_2,x)dx.
  \end{equation}
\end{proposition}

\begin{proof}
  % ($i$) 
  For simplicity we consider only the case $\ell \in \mathcal I$, the other
  case being completely similar; thus $x_0 < 0$.
  Notice first that
  \begin{equation}
    \label{eq:integral-identity}
    \int_{t_1}^{t_2} f(u_\ell(t,x_0^+))dt = \int_{t_1}^{t_2} f(u_\ell(t,x_0^-))dt.
  \end{equation}
  Indeed, by the Rankine-Hugoniot condition, the function 
  $f(u_\ell(t,x))$ does not admit discontinuities of zero slope.
  Moreover if $t_0$ is a point of continuity
  for $t \to f(u_\ell(t,x_0))$, then also $x_0$ must be a point of continuity
  for $x \to f(u_\ell(t_0,x))$. Therefore
  \begin{equation*}
    \left\{ t>0\ | f(u_\ell(t,x_0^-))\neq f(u_\ell(t,x_0^+)) \right\}
    \subset \left\{t>0\ | f(u_\ell(t^+,x_0))\neq f(u_\ell(t^-,x_0)) \right\}.
  \end{equation*}
  The latter set has Lebesgue measure zero,
  since $f(u_\ell (\cdot, x_0))$ is a function of bounded variation
  (see Remark~\ref{BV-flux}).
  Thus also the first set has Lebesgue measure zero,
  which implies~(\ref{eq:integral-identity}).

  Fix now the domain $D \doteq (t_1, t_2) \times (x_0, 0)$ and a sequence
  $\fhi_\nu \in C^1_c \left(\R \times I_\ell; \R\right)$,
  $\nu \in \N$, such that
  $\supp \left(\fhi_\nu\right)\subseteq D$ for every $\nu \in \N$
  and $\fhi_\nu$ converges to the characteristic function of $D$ as
  $\nu \to +\infty$.
  The Divergence Theorem to the BV vector
  field $(u_\ell \fhi_\nu, f(u_\ell) \fhi_\nu)$ on $D$ in combination
  with the integral equality~\eqref{entr-weak-sol} implies that
  \begin{align*}
    0 
    & = \int \int_D \diver \left(u_\ell(t, x) \fhi_\nu(t,x),
      f(u_\ell(t, x)) \fhi_\nu(t,x)\right) dx\, dt
    \\
    & = \int \int_D \diver \left(u_\ell(t, x),
      f(u_\ell(t, x)) \right) \fhi_\nu(t,x) dx\, dt
  \end{align*}
  for every $\nu \in \N$. Passing to the limit as $\nu \to + \infty$, we
  deduce that
  \begin{equation*}
    \int \int_D \diver \left(u_\ell(t, x),
      f(u_\ell(t, x))\right) dx\, dt = 0
  \end{equation*}
  which implies~(\ref{div-thm}).
\end{proof}

Since we consider finite horizon
optimization problems (see Section~\ref{sec:opt}),
we now introduce the concept of solution to the Cauchy problem for a node of
a network on $[0, T]$.
\begin{definition}
  \label{defi}
  Let  $T>0$ and $A(\cdot)=(a_{ji}(\cdot))_{j,i}\in 
  \mathbf{L}^\infty((0,T);\R^{m\times n})$ be a
  matrix valued map satisfying~\eqref{distr-rules}.
  Given $\overline u\in  \Pi_{\ell=1}^{m+n} \ \mathbf{L^\infty}
  % \mathbf{BV} 
  (I_\ell;  \Omega)$,
  we say that
  an entropy admissible weak solution to the nodal Cauchy
  problem~\eqref{incoming}-\eqref{outgoing}  on $[0,T]$,
  is a function $u\in \mathbf{C}\big([0, T];\
  \Pi_{\ell=1}^{m+n} \ \mathbf{L}^1_{\rm loc}(I_\ell;  \Omega)\big)$
  %\mathbf{L}^1(I_\ell;  \Omega)\big)$ 
  that, for some $k\in \mathbf{L^\infty} \big((0,T);  \Omega^{m\!+\!n}\big)$,
%  with $f(k) \in \mathbf{BV_{\rm loc}} \big([0,T];  \Omega^{m\!+\!n}\big)$,
  enjoys the following
  properties.
  \begin{enumerate}
  \item for every $i=1,\cdots, m$, the $i$-th component $u_i$ of $u$ is an
    entropy admissible weak solution of~\eqref{eq:conlaw1in} on
    $[0,T]\times I_i$, in the sense of Definition~\ref{def-ent-sol-ibvp};

  \item for every $j=m+1,\cdots, m+n$, the $j$-th component $u_j$ of $u$ is an
    entropy admissible weak solution of~\eqref{eq:conlaw1out} on
    $[0,T]\times I_j$, in the sense of Definition~\ref{def-ent-sol-ibvp};

  \item for a.e. $t\in (0,T]$ 
    condition \eqref{distribution} is satisfied.
  \end{enumerate}
\end{definition}

\section{Optimization problems}\label{sec:opt}

In this section we first introduce the general notation 
of admissible controls and corresponding solutions.
Then we
establish a compactness property for a class of flux-traces of solutions. 
Next, we analyze the optimization problems ({\bf I})-({\bf II})
described in the Introduction.
We shall first consider the maximization of a functional, defined in
equation~\eqref{max-J}, related to the integral of the flux at the
junction. Then, among all solutions which maximize this functional, we
choose the ones whose flux at the junction enjoy minimal total variation
(see the minimization problem~\eqref{min-max-J}). Finally, we provide
also an equivalent variational formulation of this min-max problem in
terms of the functional \eqref{IM-def}, which will be useful in
the numerical analysis of the optimal solutions.

\subsection{General setting}
Fix $T>0$.
Given  $\overline u\in  \Pi_{\ell=1}^{m+n} \ \mathbf{L^\infty}
%\mathbf{BV} 
(I_\ell;  \Omega)$,
for every $i \in \mathcal I$ and $j \in \mathcal O$,
recalling Definition~\ref{def-ent-sol-ibvp}, consider the sets
\begin{equation}
  \label{Fset-def}
  \begin{split}
    \F_{i} \doteq 
    \F_{i}(\overline u_i) 
    & \doteq 
    \left\{f(u_i(\cdot,0))\ \left| \ 
        \begin{array}{c}
          u_i \ \textrm{is a weak entropy admiss. sol. of~\eqref{eq:conlaw1in}
          on } 
          [0,T]\times I_i
          \\
          \textrm{with boundary data} \ k_i\in \mathbf{L^\infty} \big((0,T);  \Omega\big)
%          \textrm{ with } \ f(k_i)\in \mathbf{BV_{\rm loc}}
%         \big([0,T];  \Omega\big)
        \end{array}
      \right.
    \right\}
    \\
    \F_{j} \doteq \F_{j}(\overline u_j )
    & \doteq \left\{f(u_j(\cdot ,0))\ \left| 
        \begin{array}{c}
          u_j \ \textrm{is a weak entropy admiss. sol.
          of~\eqref{eq:conlaw1out} on } 
          [0,T]\times I_j
          \\
          \textrm{with boundary data} \ k_j\in \mathbf{L^\infty} \big((0,T);  \Omega\big)
%          \textrm{ with } \ f(k_j)\in \mathbf{BV_{\rm loc}}
%          \big([0,T];  \Omega\big)
        \end{array}
      \right.
    \right\},
  \end{split}
\end{equation}
which consists of the boundary flux-traces of all possible entropy
admissible weak solutions to~(\ref{eq:conlaw1in}) and
to~(\ref{eq:conlaw1out}) with initial data $\overline u_i$, $\overline u_j$, respectively.  
We recall that the flux-traces 
are defined as the one sided-limit $f(u_i(t,0))\doteq \lim_{x\to 0^-}f(u_i(t,x))$
for the incoming arcs $I_i$, $i\in\mathcal{I}$, and as 
the one sided-limit $f(u_j(t,0))\doteq \lim_{x\to 0^+}f(u_j(t,x))$
for the outgoing arcs $I_j$, $j\in\mathcal{O}$.
Then, define the set of admissible matrix valued maps 
fulfilling~\eqref{distr-rules} for a.e. $t$
\begin{align}
  \label{eq:A}
  \mathcal{A} 
    &\doteq \big\{A=(a_{ji}(\cdot))_{j,i}\in \mathbf{BV}([0,T];\R^{m\times n}) 
    %\mathbf{L}^1([0,T];\R^{m\times n})
      \ \big|\ \ 
      \textrm{condition~(\ref{distr-rules})
        holds for a.e. $t\in(0,T]$ } 
    \Big\},
    \\
    \noalign{\smallskip}
  \label{A-M}
  \mathcal{A}^M &\doteq \Big\{A=(a_{ji}(\cdot))_{j,i}\in \mathcal{A} \ | \ \tv\{a_{ji}\}\leq
  %\tv_{[0, T]}a_{ji}\leq
  M\ \ \ \forall~i,j \Big\}\qquad\quad M>0.
\end{align}

\noindent
\textbf{Admissible controls}.  
As observed in the Introduction,  we may uniquely identify every
admissible solutions of the nodal
Cauchy Problem~\eqref{incoming}-\eqref{outgoing}
by assigning the matrix valued map $A\in\mathcal{A}$ and the $m$-tuple of 
%{\it normalized boundary incoming flux-traces }
{boundary incoming flux-traces }
%$\gamma=(\gamma_1, \dots,\gamma_{m})$
%$\gamma=\big(f\big(\!\max\{u_1(\cdot,0), \theta\} \big),\dots, f\big(\!\max\{u_m(\cdot,0), \theta\} \big)\big)$
$\gamma=\big(f(u_1(\cdot,0)),\cdots, f(u_m(\cdot,0))\big)$
which, in turn, by conditions~\eqref{distribution},
determines also the $n$-tuple of boundary outgoing flux-traces
$(f (u_{m+1}(\cdot,0)),\cdots,f(u_{m+n}(\cdot,0))\big)$.
We then regard $(\gamma,A)$ as a pair of {\it junction controls} 
and we shall 
%denote by $u(t,x;\, A,\gamma)$ the corresponding solution of ~\eqref{incoming}-\eqref{outgoing}.
%With the above notations, we shall then 
adopt the following definition.
\begin{definition}
  \label{def-adm-bdr-flux}
  Given $\overline u\in  \Pi_{\ell=1}^{m+n} \ \mathbf{L}^\infty
  %\mathbf{BV} 
  (I_\ell;  \Omega)$ and
  $A=(a_{ji}(\cdot))_{j,i}\in \mathcal{A}$,
  we say that
  \begin{equation*}
    \gamma=(\gamma_1,\cdots,\gamma_m)\in  \mathbf{L}^\infty
    %\mathbf{BV} 
    \big((0,T);\, \left[f \left(\Omega\right)\right]^m\big)
  \end{equation*}
  %\mathbf{L}^1\big([0,T];\, \Omega^m\big)$
  is an
  $m$-tuple of $A$-admissible boundary inflow
  controls if there exists a boundary datum
  $k \in \mathbf{L^\infty}\left((0,T); \Omega^{m+n}\right)$ 
  % with $f(k) \in \mathbf{BV_{\rm loc}}\left([0,T]; \Omega^{m+n}\right)$ 
  such that
  the entropy admissible weak solutions $u_i$, $i\in \mathcal I$, and $u_j$, $j\in\mathcal J$, 
  of~\eqref{eq:conlaw1in} and of~\eqref{eq:conlaw1out}, respectively, satisfy:
  \begin{enumerate}
  \item $f\left(u_\ell(\cdot, 0\right)) \in \mathcal F_\ell$
    for every $\ell \in \mathcal I \cup \mathcal O$;

  \item $\gamma_i(t) = f(u_i(t, 0))$, for a.e. $t \in [0,T]$
    and all $i \in \mathcal I$,

  \item $\sum_{i=1}^m a_{ji}(t) \, \gamma_i(t) = f(u_j(t, 0))$,
    for a.e. $t \in [0,T]$ and all $j \in \mathcal O$.
  \end{enumerate}
  We denote by $u(\cdot, \cdot;\, A,\gamma)$ the entropy admissible weak
  solution of~\eqref{incoming}-\eqref{outgoing} determined by $\gamma$
  and~$A$.
  The components
  $u_\ell(\cdot, \cdot;\, A,\gamma)$, $\ell \in\mathcal{I}\cup \mathcal{O}$,
  are entropy admissible weak solution
  of \eqref{eq:conlaw1in}, \eqref{eq:conlaw1out},
  with normalized boundary data
  $k\in \mathbf{L^\infty} \big((0,T);  \Omega^{m\!+\!n}\big)$  defined by
  \begin{equation*}
    % \label{normal-bdry-data25}
    k_\ell(t)\doteq 
    \begin{cases}
      f_+^{-1}\big(\gamma_\ell(t)\big) & \textrm{if} \quad\ \ell\in\mathcal{I},
      \\
      \noalign{\medskip}
      f_-^{-1}\big(\gamma_\ell(t)\big) & \textrm{if} \quad\ \ell\in\mathcal{O},
    \end{cases}
  \end{equation*}
  where we denote as $f_-, f_+$ the restrictions of $f$
  to the intervals $[0, \theta]$
  and $[\theta, u^{max}]$, respectively.
\end{definition}
\smallskip 
 
 \noindent
We then define the sets of admissible controls as
\begin{equation}
  \label{admiss-controls}
  \begin{split}
    \mathcal{U}\!\doteq\! \mathcal{U}(\overline u)\!\doteq 
    &
    \left\{\!(A,\gamma)\!\in\! \mathcal{A}\!\times\!  \mathbf{L}^1
    %\mathbf{BV}
    ((0,T);\R^{m})
      \ \left|\ \
        \begin{array}{c}
          \gamma  \ \textrm{is an
          $m$-tuple of $A$-admissible}
          \\
          \textrm{
          boundary inflow controls}
        \end{array}
      \right.
    \right\},
    \\
    \mathcal{U}^M\!\doteq\! \mathcal{U}^M(\overline u)\!\doteq
    &
    \left\{\big((a_{ji})_{j,i},\gamma\big)\in \mathcal{U}(\overline u)
      \ \left|\ \
        \begin{array}{ll}
          \tv\{\gamma_i\}\leq M 
          & \forall~i \in \mathcal I
          \\ 
          \tv\{a_{ji}\}\leq  M 
          & 
            \forall~i \in \mathcal I, \forall~j \in \mathcal O
        \end{array}
      \right.
    \right\},
  \end{split}
\end{equation}
for all $M \ge 0$.
\begin{remark}
  \label{nonempty-control-set}
  One can easily verify that the sets of admissible controls defined by~\eqref{admiss-controls} are nonempty. Indeed, 
  given $\overline u\in  \Pi_{\ell=1}^{m+n} \ \mathbf{L}^\infty
  (I_\ell;  \Omega)$, consider the $(m+n)$-tuple of boundary data
  $k_\ell \in \mathbf{L^\infty}\left([0,T]; \Omega\right)$, $\ell =1,\dots, m+n$, defined by
  \begin{equation*}
  \begin{aligned}
  k_i(t)&= u^{max}\qquad \forall~t\in [0,T]\,,\qquad \forall~i\in\mathcal I\,,
  \\
  \noalign{\smallskip}
  k_j(t)&= 0\qquad\quad\ \ \forall~t\in [0,T]\,,\qquad \forall~j\in\mathcal O\,,
  \end{aligned}
  \end{equation*}
  and let $u_i, i\in \mathcal I$, $u_j, j\in\mathcal O$, denote the corresponding
  entropy admissible weak solution of~\eqref{eq:conlaw1in} and \eqref{eq:conlaw1out}, respectively.
  Observe that, because of the assumption {\bf (A)} on the flux function, the boundary conditions~\eqref{Dir-bc-in-flux},
  \eqref{Dir-bc-out-flux} imply
  \begin{equation*}
    u_\ell(t,0) \in\{0, u^{max}\}\qquad \forall~t\in [0,T]\,,\qquad
    \forall~\ell\in\mathcal I\cup\mathcal O\,.
  \end{equation*}
   Therefore,  we have
  \begin{equation*}
  \begin{aligned}
  \gamma_i(t)\doteq f(u_i(t,0))&=0 \qquad \forall~t\in [0,T]\,,\qquad \forall~i\in\mathcal I
  \\
  \noalign{\smallskip}
   f(u_j(t,0))&=0 \qquad \forall~t\in [0,T]\,,\qquad \forall~j\in\mathcal O\,,
   \end{aligned}
  \end{equation*}
  which shows that the $m$-tuple $\gamma\doteq (\gamma_1,\dots,\gamma_m)$
  satisfies condition \textup{(3)} of
  Definition~\ref{def-adm-bdr-flux}, for every matrix $A\doteq (a_{ji}(\cdot))_{j,i}\in \mathcal{A}$,
 or $A\doteq (a_{ji}(\cdot))_{j,i}\in \mathcal{A}^M$, and that $ \tv\{\gamma_i\}\leq M$ for any $i \in \mathcal I$.
This proves that if $A\in\mathcal A$ one has $\big(A,\,\gamma\big)\in \mathcal U(\overline u)$,
 while if $A\in\mathcal A^M$ one has $\big(A,\,\gamma\big)\in \mathcal U^M(\overline u)$,
 for every $M>0$.
  \end{remark}

\noindent
\textbf{Admissible flux traces}.  
 In connection with the above sets of admissible controls 
we introduce now the classes of flux-traces of solutions that we will
analyze in the optimization problems considered in 
this paper. Namely, we first define the set of all junction flux-traces 
in the incoming arcs for the  entropy weak solutions of~\eqref{incoming}-\eqref{outgoing}
associated to the above classes of admissible controls:
\begin{align}
  \label{eq:G}
  \mathcal{D} 
  & \doteq \mathcal{D}(\overline u) \doteq
    \left\{(g_1,\dots,g_m)\in \mathbf{L}^1((0,T);\R^m)
      % \prod_{i=1}^m \F_{i}
    \ \left|
    \ \ 
    \begin{array}{l}
      \exists~(A,\gamma)\in\mathcal{U}(\overline u)
      \\
      \textrm{s.t.}
      \ \ \,g_i= f(u_i(\cdot,0;\, A,\gamma)) \ \ \
      \forall  \ i\in\mathcal{I} 
    \end{array}
  \right.
  \right\},
  \\
  \label{g-M}
  \mathcal{D}^M 
  & \doteq \!\mathcal{D}^M(\overline u)\!\doteq\!
    \left\{(g_1,\dots,g_m)\in \mathbf{L}^1((0,T);\R^m)\
      % \prod_{i=1}^m \F_{i}\
    \left| \ \ 
    \begin{array}{l}
      \exists~(A,\gamma)\in\mathcal{U}^M(\overline u)
      \\
      \textrm{s.t.}
      \ \ \,g_i= f(u_i(\cdot,0;\, A,\gamma)) \ \ \
      \forall  \ i\in\mathcal{I}
    \end{array}
  \right.
  \right\}.
\end{align}

Next, in view of considering more general optimization problems,
see Remark~\ref{opt-funct-outgoing}, 
for every given $x_j\in I_j$, $j\in\mathcal{O}$, we consider the sets
\begin{equation}
  \label{f-M-x}
  \F_{j,x_j} \doteq \F_{j,x_j}(\overline u_j )
  \doteq \left\{f(u_j(\cdot ,x_j))\ \left| 
      \begin{array}{l}
        u_j \ \textrm{is a weak entropy admiss. sol.
        of~\eqref{eq:conlaw1out}} 
        \\
        \noalign{\medskip}
        \textrm{on } [0,T]\times I_j \
        \textrm{ for some} \ \ k_j\in \mathbf{L^\infty} \big((0,T);  \Omega\big)
      \end{array}
    \right.
  \right\}.
\end{equation}
Then, given every $x^\mathcal{O}\doteq (x_{m+1},\dots,x_{m+n})\in [0,+\infty)^n$,
we define
the sets of flux-traces of solutions evaluated at the points $x_j$
of the outgoing arcs:
\begin{align}
  \label{eq:G-x}
  \mathcal{G}_{x^\mathcal{O}} 
  & \doteq \mathcal{G}_{x^\mathcal{O}}(\overline u) \doteq
    \left\{(g_{m+1},\dots,g_{m+n})\in \prod_{j=m+1}^{m+n} \F_{j,x_j}\ 
    \left|
    \begin{array}{l}
      \exists~(A,\gamma)\in\mathcal{U}(\overline u)
      \ \ \ \ \text{s.t.}
      \\
      g_j= f(u_j(\cdot,x_j;\, A,\gamma)) \
      \forall \ j\in\mathcal{O}
    \end{array}
  \right.
  \right\},
  \\
  \label{g-M-x}
  \mathcal{G}^M_{x^\mathcal{O}} 
  & \doteq \mathcal{G}^M_{x^\mathcal{O}}(\overline u) \doteq
    \left\{(g_{m+1},\dots,g_{m+n})\in \prod_{j=m+1}^{m+n} \F_{j,x_j}\ \left|
    \begin{array}{l}
      \exists~(A,\gamma)\in\mathcal{U}^M(\overline u)
      \ \ \ \ \text{s.t.}
      \\
      g_j= f(u_j(\cdot,x_j;\, A,\gamma)) \
      \forall \ j\in\mathcal{O}
    \end{array}
  \right.
  \right\}.
\end{align}
\subsection{Compactness of admissible flux traces}
We provide here
the compactness of the sets $\mathcal{D}^M$, $\mathcal{G}^M_{x^\mathcal{O}}$  
with respect to the $\mathbf{L}^1$
topology, which is the standard setting for one dimensional
conservation laws.
With the same techniques, one can achieve compactness in $\mathbf{L}^p$,
with $p > 1$, if required by particular models. 
This property then yields the existence of optimal solutions for cost functionals  
depending on the  flux-traces  of solutions to
the nodal Cauchy Problem~\eqref{incoming}-\eqref{outgoing}
evaluated at the intersection $x=0$ or at points $x_j$ of the outgoing arcs $I_j$, $j\in\mathcal{O}$.

\begin{theorem}
  \label{compactness-G-M}
  Fix $\overline u\in  \Pi_{\ell=1}^{m+n} \  \mathbf{L}^\infty
  %\mathbf{BV} 
  (I_\ell;  \Omega)$.
  Then, for every $M>0$ the sets $\mathcal{D}^M$ ,
  $\mathcal{G}^M_{x^\mathcal{O}}$ 
  in~\eqref{g-M}, \eqref{g-M-x} are compact
  subsets of $\mathbf{L}^1((0,T);\R^m)$, $\mathbf{L}^1((0,T);\R^{n})$,
  respectively.
\end{theorem}
\begin{proof}
We shall first show the compactness of $\mathcal{D}^M$ and then derive as a consequence
the compactness of $\mathcal{G}^M_{x^\mathcal{O}}$.
\medskip

\noindent{\bf 1.}
Observe that the set $\mathcal A^M$ in~(\ref{A-M}) is compact in the $\mathbf{L}^1$
  topology by Helly's  Theorem. 
  On the other hand, setting
  \begin{equation}
  \label{FMset-def}
   \F_{\ell}^M \doteq
   % \F_{i}^M(\overline u_i)\doteq 
   \Big\{
  f(u_\ell(\cdot,0))\in \F_{\ell} \ | \ \tv \big\{f\big(u_\ell(\cdot,0)\big)
  \big\} \leq M
  \Big\},\qquad \ell\in\mathcal{I}\cup\mathcal{O}\,,
  \end{equation}
  where $ \F_{\ell}$ are the sets defined in~\eqref{Fset-def},
  by Definition~\ref{def-adm-bdr-flux} and
  definitions~\eqref{Fset-def}, \eqref{admiss-controls}, \eqref{g-M}, we have
  \begin{equation*}
    % \label{DM-def-2}
  \mathcal{D}^M =
  \bigg\{(g_1,\dots,g_m)\in \prod_{i=1}^m \F_{i}^M\ \
     \big| \ \ \exists~(a_{ji})_{ji}\in\mathcal{A}^M\quad\text{s.t.}\quad
   \sum_{i=1}^m a_{ji}\,g_i\in \F_j \ \ \
      \forall~j\in\mathcal{O} \bigg\}.
  \end{equation*}
  Since $\mathcal A^M$ is compact, in order to establish
  the compactness of $\mathcal{D}^M$ 
  it will be sufficient to show that the sets
  $\mathcal{F}^M_i$, $i \in \mathcal I$,
  are compact and that the sets $\mathcal{F}_j$,
  $j \in \mathcal O$, are closed with respect to the $\mathbf{L}^1$
  topology. We shall provide only the proof of the compactness of the set 
  $\mathcal{F}^M_i$, $i \in \mathcal I$, in~\eqref{FMset-def}, 
  the proof of the
  closureness of the sets $\mathcal{F}_j$, $j \in \mathcal O$,
  being entirely similar.
  \medskip

  \noindent{\bf 2.}
  Fix $i \in \mathcal I$ and
  let $(g^\nu)_\nu$ be a sequence in
  $\mathcal{F}^M_i$. Then, for every $\nu$ one has
  \begin{equation}
    \label{gnu-ip}
    g^\nu (t)= f\big(u_i^\nu(t,0)\big)
    % f\big(\max\{u_i^\nu(t,0), \theta\} \big)
    \qquad a.e. ~t\in [0,T]\,,
  \end{equation}
  for some weak entropy admissible 
  solution $u_i^\nu$  to~\eqref{eq:conlaw1in} 
  with boundary data $k_i^\nu\in \mathbf{L^\infty} \big((0,T);  \Omega\big)$.
  Moreover, 
  there holds
  \begin{equation*}
    % \label{tvbound-gtildenu}
    \tv\big\{ g^\nu\big\}\leq M
    \qquad \forall~\nu\,.
  \end{equation*}
  Since all $g^\nu$ take values in the bounded set $f(\Omega)$,
  applying Helly's Compactness Theorem we deduce that there exists
  a  function $\psi\in \mathbf{BV}([0,T];\, f\left(\Omega\right))$,
  with 
  \begin{equation}
  \label{tvboundpsi}
  \tv\{\psi\} \leq M, 
  \end{equation}
  so that (up to a subsequence) there holds
  \begin{equation}
    \label{flux-trace-conv}
    \begin{aligned}
 %     u_i^\nu(\cdot ,0) \quad &\to \quad \psi,\\
      % \quad
       g^\nu \quad &\to \quad \psi
      %f\circ \psi,
    \end{aligned}
    \qquad\textrm{in}\quad \mathbf{L}^1((0,T);\R)\,.
  \end{equation}
  On the other hand, consider the map
  $\widetilde k_i\in \mathbf{L^\infty} \big([0,T];  \Omega\big)$
  defined by
  \begin{equation}
  \label{tildek-def}
  \widetilde k_i(t) = f_+^{-1}(\psi(t))
  %\tilde k_i(t) = \max\big\{ u\in\Omega \ \big| \, f(u) = \psi(t)\big\},
  \qquad \forall~t\in[0,T]\,.
  \end{equation}
  %
  %  which belongs to $\mathbf{BV} \big([0,T];  \Omega\big)$.
  %$\tilde k_i\in  \mathbf{BV} \big([0,T];  \Omega\big)$.
  Then, letting $u_i\in\mathbf{C}\big([0,T];\, \mathbf{L}^1((-\infty, 0);
  \Omega)\big)$
  be the entropy admissible weak solution of~\eqref{eq:conlaw1in}
  with boundary data $\widetilde k_i$, by definition~\eqref{Fset-def} we have
  $f(u_i(\cdot,0))\in \mathcal{F}_i$.
  Therefore,
  in order to
  complete the proof of the compactness of $\mathcal F_i^M$ it will be
  sufficient to show that (up to a subsequence)
  \begin{equation}
    \label{flux-trace-conv-3}
    \begin{aligned}
      f(u_i^\nu(\cdot ,0)) \quad &\to \quad f(u_i(\cdot ,0)),
    \end{aligned}
    \qquad\textrm{in}\quad \mathbf{L}^1((0,T);\R).
  \end{equation}
In fact, \eqref{gnu-ip}, \eqref{flux-trace-conv},
%\eqref{flux-trace-conv-b}, 
\eqref{flux-trace-conv-3}
together imply
\begin{equation}
 \label{flux-trace-conv-3b}
  f(u_i(t ,0))=\psi (t) \qquad \text{for \ a.e.} \ \ t\in [0,T]\,.
\end{equation}
Hence, recalling Remark~\ref{rmk:tv_right_cont}, we deduce from~\eqref{tvboundpsi}, \eqref{flux-trace-conv-3b},  
that the essential variation of $\psi$ satisfies the bound
  \begin{equation*}
    % \label{bdry-traces-id}
   \tv \big\{f\big(u_i(\cdot,0) \big)\big\}= \tv\{\psi\} \leq M\,,
  \end{equation*}
which  implies  $f(u_i(\cdot,0))\in \mathcal{F}^M_i$ by definition~\eqref{FMset-def}.    
  \medskip

\noindent{\bf 3.}
Towards a proof of~\eqref{flux-trace-conv-3} observe first that,
 as recalled in Subsection~\ref{dir-jct-problems},
  we may  identify $u_i^\nu$ as the entropy admissible 
  solution to~\eqref{eq:conlaw1in} with normalized boundary data
  \begin{equation*}
    % \label{tildeknu-def}
  \widetilde k_i^\nu(t) = \max\{u_i^\nu(t,0), \pi(u_i^\nu(t,0))\} \qquad \forall~t\in [0,T]\,,
  \qquad\forall~\nu\,.
  \end{equation*}
  Then, by~\eqref{gnu-ip}, one has
  \begin{equation}
  \label{normal-bdry-sol-trace-2}
  f\big(\,\widetilde k_i^\nu(t)\big) = g^\nu(t)\qquad \forall~t\in [0,T],
  \qquad\forall~\nu\,.
  \end{equation}
  Hence,  \eqref{flux-trace-conv}, \eqref{tildek-def}, \eqref{normal-bdry-sol-trace-2} together imply
  \begin{equation}
    \label{flux-trace-conv-2}
    \begin{aligned}
      f\big(\widetilde k_i^\nu\big) \quad &\to \quad f\big(\widetilde k_i\big)
         \end{aligned}
    \qquad\textrm{in}\quad \mathbf{L}^1((0,T);\R).
  \end{equation}
  Therefore, 
  by virtue of the $\mathbf{L}^1$ Lipschitz
  continuous dependence of the solution to~\eqref{eq:conlaw1in} on the
  flux of the normalized boundary data (cfr.~\cite[Theorem~4]{am1}), and because
  of~\eqref{flux-trace-conv-2}, one has
  \begin{equation}
    \label{sol-conv}
    u_i^\nu(t,\cdot) \quad \to \quad u_i(t,\cdot)
    \qquad\textrm{in}\quad \mathbf{L}^1_{loc}(I_i; \Omega)
    \qquad\forall~t\in [0,T].
  \end{equation}
   On the other hand, recalling that $\Omega=[0, u^{max}]$,
 notice that, for any solution to~\eqref{eq:conlaw1in}, the backward generalized characteristics (see~\cite{d}) issuing from points $(t,x)$,
   $x< f'(u^{max})  \cdot T$,
%   $x< -\displaystyle{\max_{p\in \Omega}  |f'(p)|} \cdot T$, 
$t\in [0,T]$,  reach the $x$-axis without crossing the boundary
  line $x=0$. Therefore, all solutions $u_i^\nu$ are uniquely determined by the 
  the initial data~$\overline u$ on the region $[0,T]\times (-\infty, f'(u^{max})  \cdot T)$.
  Hence, one has 
  \begin{equation}
    \label{u=unu}
    u_i^\nu(t,x)=u_i(t,x)\qquad\forall~t  \in [0,T], \ x<f'(u^{max})
    %-\displaystyle{\max_{p\in \Omega} |f'(p)|} 
    \cdot T,\qquad\forall~\nu.
  \end{equation}
Hence, invoking Proposition~\ref{divprop} for every $u_i^\nu$,
and and thanks to~\eqref{sol-conv}, \eqref{u=unu},
we find that for all  $x_0< f'(u^{max})
%-2\max_{p\in \Omega} |f'(p)| 
\cdot T$ and  $t_1, t_2\in (0,T)$, there holds 
  \begin{equation}
    \label{div-thm-nu}
    \begin{aligned}
%      \int_{t_1}^{t_2} f\left(\psi(t)\right) dt & = \lim_\nu
%      \int_{t_1}^{t_2} g^\nu(t)dt
%      \\
\lim_\nu  \int_{t_1}^{t_2} f(u_i^\nu(t,0))dt
      &= \lim_\nu\bigg[\int_{t_1}^{t_2} f(u_i^\nu(t,x_0))dt
      +\int_{x_0}^0 u_i^\nu(t_1,x)dx- \int_{x_0}^0
      u_i^\nu(t_2,x)dx\bigg]
      \\
      &= \int_{t_1}^{t_2} f(u_i(t,x_0))dt+\int_{x_0}^0
      u_i(t_1,x)dx- \int_{x_0}^0 u_i(t_2,x)dx
      \\
      &= \int_{t_1}^{t_2} f(u_i(t,0))dt.
    \end{aligned}
  \end{equation}
By the arbitrariness of $t_1, t_2 \in (0,T)$, and since all $f(u_i^\nu(\cdot ,0))$
take values in the bounded set $f(\Omega)$, we recover from~\eqref{div-thm-nu} the
convergence~\eqref{flux-trace-conv-3} completing the proof of the compactness of $\mathcal{D}^M$.
\medskip

\noindent{\bf 4.}
With the same analysis of the previous points we deduce the compactness of the set
$\mathcal{G}^M_{x^\mathcal{O}}$ when $x^\mathcal{O} = 0_n\doteq (0, \dots, 0)\in\R^n$.
Next, consider $x^\mathcal{O} = (x_{m+1},\dots,x_{m+n}) \neq 0_n$
and observe that by definitions~\eqref{f-M-x}, \eqref{eq:G-x}, \eqref{FMset-def}, we have
\begin{equation}
  \label{DM-def-4}
  \mathcal{G}^M_{x^\mathcal{O}} =
  \left\{
    (g_{m+1}, \cdots\!, g_{m+n}) \in \mathbf{L}^1((0,T);\R^n)\ \left|
      \begin{array}{l}
        \forall~j\in\mathcal{O}\ \ \exists~u_j 
        %, j\in\mathcal{O}, 
        \textrm{ weak entr. admiss. soln. of~\eqref{eq:conlaw1out}}
        \vspace{.1cm} \\
        \textrm{on } [0,T] \!\times\! [0,+\infty) \ \textrm{ for some} \ \
        k_j\!\in\! \mathbf{L^\infty} \big((0,T); \! \Omega\big)
        \vspace{.1cm} 
        \\
        \textrm{ s.t.} \ \ g_j= f(u_j(\cdot,x_j))
         \ \ \text{and}
        \vspace{.1cm} \\
        \big(f(u_{m+1}(\cdot,0)), \dots, f(u_{m+n}(\cdot,0))\big)\in 
        \mathcal{G}^M_{0_n}
      \end{array}\!\!\!\!
    \right.
  \right\}.
\end{equation}
Then, letting $(g^\nu)_\nu$ be a sequence in $\mathcal{G}^M_{x^\mathcal{O}}$,
there will be 
a sequence $(u^\nu)_\nu$ of $n$-tuples of  entropy admissible weak solutions
of~\eqref{eq:conlaw1out}
with boundary data $k^\nu\in \mathbf{L^\infty} \big((0,T);  \Omega^n\big)$,
so that, for every $\nu$, one has
\begin{equation}
  \label{gnu-fnu-x-0}
  g_j^\nu = f(u_j^\nu(\cdot, x_j))
  \quad \forall~j\in\mathcal{O}, \qquad
  \big(f(u_{m+1}^\nu(\cdot,0)), \dots, f(u_{m+n}^\nu(\cdot,0))\big)\in 
  \mathcal{G}^M_{0_n} \,.
\end{equation}
By the compactness of $\mathcal{G}^M_{0_n}$ 
and relying on the analysis performed at previous points,
it follows that there exists an $n$-tuple $(u_{m+1},\dots,u_{m+n})$ of
entropy admissible weak solutions of~\eqref{eq:conlaw1out}
with boundary data $k\in \mathbf{L^\infty} \big([0,T];  \Omega^n\big)$,
such that 
\begin{equation}
  \label{limit-in-GM}
  \big(f(u_{m+1}(\cdot,0)), \dots, f(u_{m+n}(\cdot,0))\big)\in 
  \mathcal{G}^M_{0_n}
\end{equation}
and a subsequence
$\big((f(u_{m+1}^{\nu'}(\cdot,0)), \dots, f(u_{m+n}^{\nu'}(\cdot,0)))\big)_{\nu'}$
so that there holds
\begin{equation*}
  \begin{gathered}
 f(u_j^{\nu'}(\cdot ,0)) \quad \to \quad f(u_j(\cdot ,0)),
    \qquad\textrm{in}\quad \mathbf{L}^1((0,T);\R),
    \\
    \noalign{\medskip}
    u_j^{\nu'}(t,\cdot) \quad \to \quad u_j(t,\cdot)
    \qquad\textrm{in}\quad \mathbf{L}^1_{loc}(I_j; \Omega)
    \qquad\forall~t\in [0,T],
\end{gathered}
\qquad\quad\forall~j\in\mathcal{O}\,.
\end{equation*}
Hence, invoking Proposition~\ref{divprop} for every $u_j^{\nu'}$,
we find
  \begin{equation}
    \label{div-thm-nu-2}
    \begin{aligned}
\lim_{\nu'}  \int_{t_1}^{t_2} f(u_j^{\nu'}(t,x_j))dt
      &= \lim_{\nu'}\bigg[\int_{t_1}^{t_2} f(u_j^{\nu'}(t,0))dt
      +\int_0^{x_j} u_j^{\nu'}(t_1,x)dx- \int_0^{x_j}
      u_j^{\nu'}(t_2,x)dx\bigg]
      \\
      &= \int_{t_1}^{t_2} f(u_j(t,0))dt+\int_0^{x_j}
      u_j(t_1,x)dx- \int_0^{x_j} u_j(t_2,x)dx
      \\
      &= \int_{t_1}^{t_2} f(u_j(t,x_j))dt.
    \end{aligned}
  \end{equation}
By the arbitrariness of $t_1, t_2 \in (0,T)$, and since all $f(u_i^{\nu'}(\cdot ,x_j))$
take values in the bounded set $f(\Omega)$, we derive from~\eqref{gnu-fnu-x-0}, \eqref{div-thm-nu-2} the
convergence 
  \begin{equation*}
    % \label{flux-trace-conv-42}
    \begin{aligned}
      g_j^{\nu'} \quad &\to \quad f(u_j(\cdot ,x_j)),
    \end{aligned}
    \qquad\textrm{in}\quad \mathbf{L}^1((0,T);\R)
    \qquad\forall~j\in\mathcal{O}.
  \end{equation*}
On the other hand, because of~\eqref{DM-def-4}, \eqref{limit-in-GM}, 
we have
\begin{equation*}
  % \label{limit-in-GM-42}
\big(f(u_{m+1}(\cdot,x_j)), \dots, f(u_{m+n}(\cdot,x_))\big)\in 
      \mathcal{G}^M_{x^\mathcal{O}},
\end{equation*}
completing the proof of the compactness of $ \mathcal{G}^M_{x^\mathcal{O}}$.
\end{proof}
\begin{remark}
  \label{pointw-conv-flux}
  By the proof of Theorem~\ref{compactness-G-M}
  and relying on Helly's compactness theorem it follows that,
  if $\{\big(A^\nu,f(u_1^\nu(\cdot,0)),\dots f(u_m^\nu(\cdot,0))\big)\}_\nu$
  is a sequence
  in $\mathcal{A}^M \times\mathcal{D}^M$,
  then letting $\psi\in \mathbf{BV}([0,T];\, f\left(\Omega\right))$
  be a map such that
  \begin{equation*}
    f(u_i^\nu(\cdot ,0)) \quad \to \quad \psi
    \qquad\textrm{in}\quad \mathbf{L}^1((0,T);\R)\,,
  \end{equation*}
  and denoting with $u_i$, $i\in\mathcal{I}$,
  the entropy admissible weak solutions of~\eqref{eq:conlaw1in}
  with boundary data $\widetilde k_i$
  defined according with~\eqref{tildek-def},
  there exist $A\in\mathcal{A}^M$ so that
  (up to a subsequence) one has
  \begin{equation*}
    % \label{pointw-conv-Anu}
    A^\nu(t) \quad \to \quad A(t)
    \qquad\quad \textrm{for \ a.e.} \ \ t\in[0,T]\,,
  \end{equation*}
  and
  \begin{equation*}
    % \label{pointw-conv-flux-unu}
    f(u_i^\nu(t ,0)) \quad \to \quad f(u_i(t ,0))\qquad
    \text{for \ a.e.} \ \ t\in [0,T]\,.
  \end{equation*}
\end{remark}
%
%\smallskip

%
\begin{remark}
  \label{lowersemic-totvar}
  We underline that, in order to achieve the compactness of
  a set of flux-traces,
  one can alternatively consider
  a class of admissible controls defined as a set of uniformly
  $\mathbf{BV}$ bounded boundary data. 
  Unfortunately this choice makes the analysis more involved, due to the
  lack of convergence of the trace of solutions.
  % it is more natural to consider a class of
  % admissible junction controls with a uniform $\mathbf{BV}$
  % bound on the flux of
  % the boundary data.
  In fact,  
  if we consider a sequence $(u^\nu)_\nu$ of solutions to~\eqref{eq:conlaw1in}
  that converge in $\mathbf{L}^1_{loc}$  to some solution $u$,
  then the boundary traces of $u^\nu$ do not converge in general to the
  boundary trace of the limiting solution $u$.

  % However, with a finer analysis of the structure of
  % the boundary traces, one can  show that it still holds
  % the lower semicontinuity property of the total variation
  % of the traces of solutions with respect to the convergence of the solutions
  % in their domain,
  % as illustrated in the following example.
  As an example, given the flux function $f(u)=u(1-u)$, $u\in [0,1]$, 
  consider the solutions $u_i^\nu$ to~\eqref{eq:conlaw1in} with initial data 
  \begin{equation}
    \label{indata-nonconv-trace}
    \overline u_i(x) = 
    \begin{cases}
      \frac{1}{8}\quad &\text{if}\quad x<-1\,,
      \\
      \noalign{\smallskip}
      \frac{1}{4}\quad &\text{if}\quad -1<x<0\,,
    \end{cases}
  \end{equation}
  and boundary data
  \begin{equation*}
    k_i^\nu(t) = \frac{3}{4}+\frac{1}{\nu}\qquad\forall~t\,,
    \quad\forall~\nu>8\,.
  \end{equation*}
  By a direct computation, see Figure~\ref{fig:rmk_3.3},
  one can verify that 
  \begin{equation}
    \label{uinu-sol-ex-nonconv-traces}
    u_i^\nu(t,x)=
    \begin{cases}
      \ \frac{1}{8}\qquad\text{if}\qquad 
      \begin{cases}
        x< -1+\frac{5t}{8}, \qquad t\le \frac{8\nu}{5\nu+8}&
        \\
        \noalign{\smallskip}
        \qquad\qquad\  \text{or}&
        \\
        \noalign{\smallskip}
        x< -\frac{8}{5\nu+8}+\frac{(\nu-8)\big((5\nu+8)t-8\nu\big)}
        {8\nu(5\nu+8)},& \frac{8\nu}{5\nu+8}\le t\le
        \frac{8\nu^2}{(5\nu+8)(\nu-8)}
        \\
        \noalign{\smallskip}
        \qquad\qquad\  \text{or}&
        \\
        \noalign{\smallskip}
        x<0, \qquad t\ge \frac{8\nu^2}{(5\nu+8)(\nu-8)}&
      \end{cases}
      \\
      \noalign{\bigskip}
      \ \frac{1}{4}\qquad\text{if}\qquad 
      -1+\frac{5t}{8} <x< -\frac{t}{\nu}, \qquad t< \frac{8\nu}{5\nu+8}
      \\
      \noalign{\bigskip}
      \ \frac{3}{4}+\frac{1}{\nu}\quad\text{if}\quad 
      \begin{cases}
        -\frac{t}{\nu}<x< 0, \qquad t< \frac{8\nu}{5\nu+8}&
        \\
        \noalign{\smallskip}
        \qquad\qquad\  \text{or}&
        \\
        \noalign{\smallskip}
        \frac{(\nu-8)\big((5\nu+8)t-8\nu\big)}{8\nu(5\nu+8)}
        -\frac{8}{5\nu+8}< \! x \! <0,& \! \frac{8\nu}{5\nu+8}< \! t \!<
        \frac{8\nu^2}{(5\nu+8)(\nu-8)}\,.
      \end{cases}
    \end{cases}
  \end{equation}
  \begin{figure}
    \centering
    \includegraphics[]{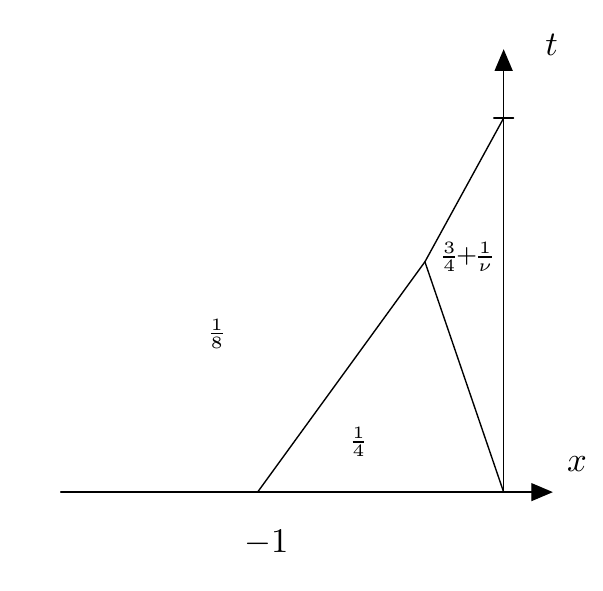}
    \caption{The representation of the function $u_i^\nu$
      in (\ref{uinu-sol-ex-nonconv-traces}) of Remark~\ref{lowersemic-totvar}.}
    \label{fig:rmk_3.3}
  \end{figure}
  The sequence $(u_i^\nu)_\nu$ converges in
  $\mathbf{L}^1_{loc}\big((0,+\infty)\times(-\infty, 0)\big)$
  to the solution $u_i$ of~\eqref{eq:conlaw1in} with initial
  data~\eqref{indata-nonconv-trace} 
  and boundary data
  \begin{equation*}
    k_i(t) = \frac{3}{4}\qquad\forall~t\,,
  \end{equation*}
  since $(k_i^\nu)_\nu$ converges to $k_i$ in $\mathbf{L}^1_{loc}(0,\infty)$.
  Notice that
  \begin{equation*}
    % \label{ui-sol-ex-nonconv-traces}
    u_i(t,x)=
    \begin{cases}
      \ \frac{1}{8}\qquad\text{if}\qquad  
      \begin{cases}
        x< -1+\frac{5t}{8}, \qquad t\le \frac{8}{5}&
        \\
        \noalign{\smallskip}
        \qquad\qquad\  \text{or}&
        \\
        \noalign{\smallskip}
        x< 0, \qquad  t\ge \frac{8}{5}&
      \end{cases}
      \\
      \noalign{\bigskip}
      \ \frac{1}{4}\qquad\text{if}\qquad 
      -1+\frac{5t}{8} <x< 0, \qquad t< \frac{8}{5}\,
    \end{cases}
  \end{equation*}
  and hence its boundary trace is
  \begin{equation*}
    u_i(t,0)=
    \begin{cases}
      \ \frac{1}{4}\quad &\text{if}\qquad 
      t< \frac{8}{5}
      \\
      \noalign{\medskip}
      \ \frac{1}{8}\quad &\text{if}\qquad 
      t\ge \frac{8}{5}\,.
    \end{cases}
  \end{equation*}
  On the other hand the  boundary traces of $u_i^\nu$ are
  \begin{equation*}
    u_i^\nu(t,0)=
    \begin{cases}
      \ \frac{1}{4}\quad &\text{if}\qquad 
      t=0
      \\
      \noalign{\medskip}
      \ \frac{3}{4}+\frac{1}{\nu}\quad &\text{if}\qquad 
      0<t< \frac{8\nu^2}{(5\nu+8)(\nu-8)}
      \\
      \noalign{\medskip}
      \ \frac{1}{8}\quad &\text{if}\qquad 
      t\ge \frac{8\nu^2}{(5\nu+8)(\nu-8)}
    \end{cases}
  \end{equation*}
  which pointwise converge to the function
  \begin{equation*}
    \psi(t)=
    \begin{cases}
      \ \frac{1}{4}\quad &\text{if}\qquad 
      t=0
      \\
      \noalign{\medskip}
      \ \frac{3}{4}\quad &\text{if}\qquad 
      0<t< \frac{8}{5}
      \\
      \noalign{\medskip}
      \ \frac{1}{8}\quad &\text{if}\qquad 
      t\ge \frac{8}{5}\,.
    \end{cases}
  \end{equation*}
  Instead, the sequence of flux-traces $(f(u_i^\nu(\cdot, 0)))_\nu$
  clearly converges to $f(u_i(\cdot, 0))$.
  % Furthermore, the essential variation
  % of $u_i(\cdot,0)$ on $(0,+\infty)$ satisfies the bound
  % % 
  % \begin{equation*}
  %   \frac{1}{8}=\tv  \big\{u_i(\cdot,0)\big\} 
  %   <\lim_\nu \tv  \big\{u_i^\nu(\cdot,0)\big\}=
  %   \frac{9}{8}\,.
  % \end{equation*}
\end{remark}
\smallskip

\subsection{Maximization of flux depending functionals}

Let $\mathcal{J} : \R^m \to\R$ be a continuous map. 
Given $\overline u\in  \Pi_{\ell=1}^{m+n} \ \mathbf{L}^\infty
  % \mathbf{BV} 
  (I_\ell;  \Omega)$,
and $T, M>0$,
we  consider the  optimization problem:
\begin{equation}
  \label{max-J}
   \sup_{(A,\,\gamma)\,\in\, \mathcal{U}^{^M}(\overline u)\,}  \int_0^T \mathcal{J}\big(f^{\mathcal{I}}(u(t,0;\, A,\gamma))\big)dt\,,
    \end{equation}
where $\mathcal{U}^M(\overline u)$ is the set of admissible controls defined in~\eqref{admiss-controls}
and  we set
\begin{equation}
  \label{fI-vector-def}
  f^{\mathcal{I}}(u(t,0;\, A,\gamma))\doteq \big(f(u_1(t,0;\, A,\gamma)),
  \cdots, f(u_m(t,0;\, A,\gamma))\big),
 \end{equation}
with the notations of Definition~\ref{def-adm-bdr-flux}.
Typical examples of  the map $\mathcal{J}$ are 
\begin{equation*}
  % \label{eq:sum-prod-map}
  \mathcal{J}_1(\gamma_1,\dots,\gamma_m) = \sum_{i = 1}^m \gamma_i =
  \gamma_1+\cdots+\gamma_m,
  \qquad\quad
  \mathcal{J}_2(\gamma_1,\dots,\gamma_m) = \prod_{i = 1}^m \gamma_i =
  \gamma_1\cdot\, \cdots \, \cdot\gamma_m\,,
\end{equation*}
which are  commonly considered in the optimization rules introduced
for the definition of various Riemann solvers (see~\cite{b-y, gp2}).
Thanks to Theorem~\ref{compactness-G-M} we immediately deduce that the
optimization problem~\eqref{max-J} admits a solution.
\begin{theorem}
  \label{exmas}
  Given $\overline u\in  \Pi_{\ell=1}^{m+n} \ \mathbf{L}^\infty
  %\mathbf{BV} 
  (I_\ell;  \Omega)$,
and $T>0$,
  for every fixed $M>0$
  %, the maximization problem~(\ref{max-J}) admits a solution. More precisely,
   there exists $(\,\widehat A, \widehat \gamma\,)\in
  \mathcal{U}^M(\overline u)$ such that
  \begin{equation}
    \label{max}
    \int_0^T \mathcal{J}\big(f^{\mathcal{I}}(u(t,0;\, \widehat A,\,\widehat
    \gamma))\big)\,dt
    = \sup_{(A,\,\gamma)\,\in\, \mathcal{U}^{^M}(\overline u)\,} 
    \int_0^T \mathcal{J}\big(f^{\mathcal{I}}(u(t,0;\, A,\gamma))\big)\,dt\,.
  \end{equation} 
\end{theorem}
\begin{proof}
  Since the set $\mathcal{D}^M$ is compact in $\mathbf{L}^1$ by
  Theorem~\ref{compactness-G-M}, 
  the conclusion follows observing
  that the map $g\mapsto \int_0^T \mathcal{J}(g(t))dt$ from
  $(\mathbf{L}^1((0, T);\R))^m$ to $\R$ is continuous.
\end{proof}

\noindent
%Recalling Definition~\ref{def-adm-bdr-flux}, t
The optimal solutions provided by Theorem~\ref{exmas} 
determine the  solutions of the nodal
Cauchy problem~\eqref{incoming}-\eqref{outgoing} 
whose boundary incoming flux-traces solve the
optimization problem~\eqref{max-J}.
Namely, an immediately consequence of Theorem~\ref{exmas}  is the following
\begin{corollary}
  \label{optsol-nodal-optsol-max}
  Given $\overline u\in  \Pi_{\ell=1}^{m+n} \ \mathbf{L}^\infty
  (I_\ell;  \Omega)$, and
  $T,M>0$, let $((\widehat a_{ji})_{ji},\, \widehat \gamma\,)\in
  \mathcal{U}^M(\overline u)$ be an optimal pair of controls
  for the maximization problem~\eqref{max}.
  Then, $\widehat u\doteq u(\cdot,\cdot\,; (\widehat a_{ji})_{ji},
  \widehat \gamma\,)$ is
  an entropy admissible weak solution  of the nodal
  Cauchy problem~\eqref{incoming}-\eqref{outgoing}  on $[0,T]$ 
  that satisfies 
  \begin{equation*}
    \label{max-nodal-sol}
    \begin{gathered}
      f\big(\widehat u_i(t,0)\big)=\widehat \gamma_i (t) 
      \qquad \textrm{a.e.}~t\in[0,T]\,,\quad \forall~i\in\mathcal{I}\,,
      \\
      \noalign{\smallskip}
      f(\widehat u_j(t,0))= \sum_{i=1}^m \widehat a_{ji}\,f(\widehat u_i(t,0))
      \qquad \textrm{a.e.}~t\in[0,T]\,,\quad \forall~j\in\mathcal{O}\,.
    \end{gathered}
  \end{equation*}
\end{corollary}

\begin{remark}
  \label{opt-funct-outgoing}
  Relying on the compactness of the set
  $\mathcal{G}^M_{x^\mathcal{O}}$ in~\eqref{g-M-x} established in
  Theorem~\ref{compactness-G-M}, we may derive the existence of
  optimal solutions for more general cost functionals than the one
  considered in~\eqref{max-J} related both to junction fluxes of the
  incoming edges and to the fluxes of solutions at fixed points of the
  outgoing edges. Namely, given a continuous map
  $\mathcal{J} : \R^{m+n} \to\R$, initial data
  $\overline u\in \Pi_{\ell=1}^{m+n} \ \mathbf{L}^\infty
  % \mathbf{BV}
  (I_\ell; \Omega)$,
  and $T, M>0$,
  $x^\mathcal{O} = (x_{m+1},\cdots,x_{m+n})\in [0,+\infty)^n$, setting
  \begin{equation*}
    % \label{fO-vector-def}
    f^{\mathcal{O}}(u(t,x^{\mathcal{O}};\, A,\gamma))\doteq
    \big(f(u_{m+1}(t,x_{m+1};\, A,\gamma)), \cdots,
    f(u_{m+n}(t,x_{m+n};\, A,\gamma))\big),
  \end{equation*}
  consider the optimization problem
  \begin{equation}
    \label{max-J-inc-out}
    \sup_{(A,\,\gamma)\,\in\, \mathcal{U}^{^M}(\overline u)\,}  \int_0^T \mathcal{J}\Big(f^{\mathcal{I}}(u(t,0;\, A,\gamma)),\,
    f^{\mathcal{O}}(u(t,x^{\mathcal{O}};\, A,\gamma))\Big)dt\,,
  \end{equation}
  where $f^{\mathcal{I}}$ is defined as in~\eqref{fI-vector-def}. With
  the same arguments of Theorem~\ref{exmas} we deduce
  that~\eqref{max-J-inc-out} admits an optimal solution and the
  supremum is achieved as a maximum.
\end{remark}

\subsection{Minimization of the total variation of optimal solutions}
It is quite easy to realize that 
the solution of the optimization problem~(\ref{max-J}) is in general not unique.
This is illustrated by an  example discussed in~\cite{accg}
where it is considered the special case of a junction with one incoming and one outgoing arc
and two solutions of~\eqref{max-J} are provided for the functional $\mathcal J (\gamma)=\gamma$.

On the other hand, in many applications (e.g. see~\cite{cgr} for vehicular traffic modeling) it would be 
desirable to select those optimal solutions that keep as small as possible the
total variation of the incoming  fluxes, i.e. that minimize
$\sum_{i=1}^m  \tv \{f(u_i(\cdot,0))\}$.
Therefore, for every $M > 0$, we define the  set of optimal pairs
\begin{equation*}
  % \label{eq:over-GM}
  \mathcal{U}^M_{max} = 
  \left\{(\widehat A, \widehat \gamma\,) \in {\mathcal{U}}^M:\
   \ \text{\eqref{max} holds}\right\},
\end{equation*}
which is nonempty because of Theorem~\ref{exmas},
and we  consider the  optimization problem
\begin{equation}
  \label{min-max-J}
  \inf_{\, (\widehat A,\,\widehat\gamma\,)\,\in\,\mathcal{U}^{^{M}}_{max}\ } 
  \sum_{i=1}^m  \tv \left\{f\big(u_i\big(\cdot, 0;\,
    \widehat A,\widehat \gamma\big)\big)\right\}\,.
\end{equation}
In the same spirit of~\eqref{min-max-J}, we also address an optimization problem
that includes in the same cost 
the integral functional in~\eqref{max-J} and the total variation of the flux in~\eqref{min-max-J}.
Namely, for every fixed $\delta>0$, consider the maximization problem
\begin{equation}
\label{max-Jdelta}
 \sup_{(A,\,\gamma)\,\in\, \mathcal{U}}
 %(\overline u)}
  \left(
   \int_0^T \mathcal{J}\big(f^{\mathcal{I}}(u(t,0;\, A,\gamma))\big)dt
   -\delta \, \sum_{i=1}^m \tv \left\{f\big(u_i\big(\cdot, 0;\, A, \gamma\big)\big)\right\}
   -\delta\tv\big\{A\big\}
   \right)
   \end{equation}
where $\tv\big\{A\big\}=\tv\big\{(a_{ji})_{j,i}\big\}\doteq\sum_{j,i} \tv\{a_{ji}\}$.
Here, the admissible controls are $ \mathbf{BV}$ functions with  arbitrarily large total variation.
\begin{theorem}
  \label{prop:min-sup}
  Given $\overline u\in \Pi_{\ell=1}^{m+n} \ \mathbf{L}^\infty
  % \mathbf{BV}
  (I_\ell; \Omega)$, and $T>0$, the following hold.
  \begin{itemize}
  \item[($i$)] For every fixed $M>0$ there exists
    $(A^*, \gamma^*)\in \mathcal{U}^M_{max}(\overline u)$ such that
    \begin{equation}
      \label{eq:min-sup}
      \sum_{i=1}^m  \tv \left\{f\big(u_i\big(\cdot, 0;\, A^*, \gamma^*\big)\big)\right\}
      = \inf_{\, (\widehat A,\,\widehat\gamma\,)\,\in\,\mathcal{U}^{^{M}}_{max}(\overline u)\ } 
      \sum_{i=1}^m  \tv \left\{f\big(u_i\big(\cdot, 0;\, \widehat A,\widehat \gamma\big)\big)\right\}\,.
    \end{equation} 
    \medskip
  \item[($ii$)] For every fixed $\delta>0$ there exists
    $(A^{\delta}, \gamma^{\delta})\in \mathcal{U}(\overline u)$ such
    that
    \begin{equation}
      \label{eq:max-Jdelta}
      \begin{aligned}
        & \int_0^T \mathcal{J} \big(f^{\mathcal{I}}(u(t,0;\,
        A^{\delta}\gamma^{\delta}))\big)dt -\delta \, \sum_{i=1}^m \tv
        \left\{f\big(u_i\big(\cdot, 0;\, A^{\delta},
          \gamma^{\delta}\big)\big)\right\}
        -\delta\tv\big\{A^{\delta}\big\}=
        \\
        &\, = \sup_{(A,\,\gamma)\,\in\, \mathcal{U}(\overline u)}
        \left( \int_0^T \mathcal{J}\big(f^{\mathcal{I}}(u(t,0;\,
          A,\gamma))\big)dt -\delta \, \sum_{i=1}^m \tv \left\{
            f\big(u_i\big(\cdot, 0;\, A, \gamma\big)\big)\right\}
          -\delta\tv\big\{A\big\} \right).
      \end{aligned}
    \end{equation}
  \end{itemize}
\end{theorem}

\noindent
\textit{Proof.}
\smallskip

\noindent  {(${ i}$)} 
  Consider a minimizing sequence $\big((A^\nu,g^\nu)\big)_{\!\nu}$ 
  for~\eqref{min-max-J} such that
  \begin{gather}
  g^\nu_i= f\big(u_i(\cdot,0;\ A^\nu, \gamma^\nu)\big)\qquad\quad 
  (A^\nu,\,\gamma^\nu)\,\in\,\mathcal{U}^{M}_{max}\,,
  \qquad i\in\mathcal{I}\,,  \qquad\forall~\nu\,,
  \\
  \noalign{\medskip}
  \label{gnu-opt}
  \int_0^T \mathcal{J}\big(g^\nu_1(t),\dots, g^\nu_m(t)\big)\,dt
   =\max_{(A,\,\gamma)\,\in\, \mathcal{U}^{^M}\,}  \int_0^T \mathcal{J}\big(f^{\mathcal{I}}(u(t,0;\, A,\gamma))\big)\,dt,
   \qquad\forall~\nu\,,
  \\
  \noalign{\medskip}
  \label{gnu-tv-inf}
  \sum_{i=1}^m \tv \big\{g_i^\nu \big\}\quad
    \to \quad \inf_{\, (\widehat A,\,\widehat\gamma\,)\,\in\,\mathcal{U}^{^{M}}_{max}\ } 
\sum_{i=1}^m  \tv \left\{f\big(u_i\big(\cdot, 0;\, \widehat A,\widehat \gamma\big)\big)\right\}.
  \end{gather}
   Since  by~\eqref{g-M}
   %by~\eqref{Dm-Um} 
   we have $(A^\nu, g^\nu) \in{\mathcal{D}}^M$ for all $\nu$,
  applying Theorem~\ref{compactness-G-M} and relying on Remark~\ref{pointw-conv-flux},
  we deduce that there exists a subsequence, again denoted by
  $\big((A^\nu,g^\nu)\big)_{\!\nu}$, and an element $(A^*, g^*) \in
  {\mathcal{D}}^M$, with $g^*= f^\mathcal{I}(u(\cdot,0;\, A^*,\gamma^*))$,
  $(A^*,\gamma^*)\in\mathcal{U}^M$,
   such that 
  \begin{equation}
    \label{gnu-conv-57}
%    A^\nu \quad \to \quad A^*,\qquad
      g^\nu \quad \to \quad g^*\quad
    \qquad\textrm{in}\quad \mathbf{L}^1\,.
  \end{equation}
   Hence, \eqref{gnu-opt} together with~\eqref{gnu-conv-57} yields
   \begin{equation*}
    \int_0^T \mathcal{J}\big(g^*(t)\big)\,dt
   =\max_{(A,\,\gamma)\,\in\, \mathcal{U}^{^M}\,}  \int_0^T \mathcal{J}\big(f^{\mathcal{I}}(u(t,0;\, A,\gamma))\big)\,dt\,,
  \end{equation*} 
  showing that $(A^*,\gamma^*) \in \mathcal{U}^M_{max}$.
  On the other hand, by the lower semicontinuity property of the essential variation 
  with respect to the  $\mathbf{L}^1_{{\rm loc}}$-topology (see~\cite[Sections 3.1-3.2]{AFP}),
  and by the property of the 
  liminf  operation with respect to the sum, we derive
  \begin{equation}
  \begin{aligned}
  \label{gnu-conv-577}
  \sum_{i=1}^m \tv \big\{g^*_i\big\} &\le  \sum_{i=1}^m\liminf_{\nu} \tv\big\{g^\nu_i\big\}
  %
%  \\
%  \noalign{\smallskip}
  \le \liminf_{\nu}\bigg(\sum_{i=1}^m  \tv\big\{g^\nu_i\big\}\bigg)
  \\
  \noalign{\smallskip}
  &=\lim_{\nu} \bigg(\sum_{i=1}^m  \tv\big\{g^\nu_i\big\}\bigg) =
  \inf_{\, (\widehat A,\,\widehat\gamma\,)\,\in\,\mathcal{U}^{^{M}}_{max}\ } 
\sum_{i=1}^m  \tv \left\{f\big(u_i\big(\cdot, 0;\, \widehat A,\widehat \gamma\big)\big)\right\}.
  \end{aligned}
  \end{equation}
  Since $(A^*,\gamma^*) \in \mathcal{U}^M_{max}$ it follows from~\eqref{gnu-conv-577} that
  \eqref{eq:min-sup} holds concluding the proof of ($i$).
  \medskip
  
  \noindent  {(${ ii}$)} 
  We shall provide a proof of the existence of $(A^{\delta},  \gamma^{\delta})\in
  \mathcal{U}(\overline u)$
  satisfying~\eqref{eq:max-Jdelta} only in the case where
  the supremum in~\eqref{max-Jdelta} is strictly positive. In fact, in the case where such a supremum is non positive,  
  we may always consider the problem~\eqref{max-Jdelta} with a new integrand function $\mathcal{J}^\sharp=\alpha+\mathcal{J}$,
  $\alpha$ being a positive constant chosen so that the supremum in~\eqref{eq:max-Jdelta}  be
  strictly positive. Clearly, a maximizer for~\eqref{max-Jdelta} with $\mathcal{J}^\sharp$ in place of $\mathcal{J}$
  provides also a maximizer for the original problem~\eqref{max-Jdelta}.

  Then, assume that the supremum in~\eqref{max-Jdelta}  is strictly positive and 
  consider a maximizing sequence $\big((A^\nu,g^\nu)\big)_{\!\nu}$ 
  for~\eqref{max-Jdelta}  such that, for all $\nu$, there holds
  \begin{equation}
  \label{gnu-def-79}
  g^\nu_i= f\big(u_i(\cdot,0;\ A^\nu, \gamma^\nu)\big)\qquad\quad 
  (A^\nu,\,\gamma^\nu)\,\in\,\mathcal{U}\,,
%  \quad\tv\{A^\nu\}\le M\,,
  \qquad i\in\mathcal{I}\,, 
  \end{equation}
  and satisfying
  \begin{equation}
  \label{gnu-conv-70}
   \begin{aligned}
    & \lim_\nu \bigg(\int_0^T \mathcal{J}\big(g^\nu_1(t),\dots, g^\nu_m(t)\big)\,dt
   -\delta \, \sum_{i=1}^m  \tv \big\{g_i^\nu \big\}-\delta \tv\big\{A^\nu\big\}\bigg)=
    \\
    &\qquad\quad =\sup_{(A,\,\gamma)\,\in\, \mathcal{U}}
    %_{(A,\,\gamma)\,\in\, \mathcal{U}\,}
    \left(
   \int_0^T \mathcal{J}\big(f^{\mathcal{I}}(u(t,0;\, A,\gamma))\big)dt
   -\delta \, \sum_{i=1}^m \tv \left\{f\big(u_i\big(\cdot, 0;\, A, \gamma\big)\big)\right\}
   -\delta \tv\big\{A\big\}
   \right).
    \end{aligned}
  \end{equation}
  Since all $g^\nu$ take values in the bounded set $\left[f(\Omega)\right]^m$
  and $ \mathcal{J}$ is continuous,
  and by virtue of the strictly positive assumption on the supremum in~\eqref{max-Jdelta},
 we may find a constant $C_1>0$ such that,
  for all $\nu$ sufficiently large,
  there hold
  \begin{equation}
  \label{bound-tv-maximizer-1}
  \begin{gathered}
  \int_0^T \mathcal{J}\big(g^\nu_1(t),\dots, g^\nu_m(t)\big)\,dt\le C_1\,,
  \\
  \int_0^T \mathcal{J}\big(g^\nu_1(t),\dots, g^\nu_m(t)\big)\,dt
   -\delta \, \sum_{i=1}^m  \tv \big\{g_i^\nu \big\}-\delta \tv\big\{A^\nu\big\}>0\,.
  \end{gathered}
  \end{equation}
  Thus, \eqref{bound-tv-maximizer-1} implies that, for $\nu$ sufficiently large, one has
  \begin{equation}
  \label{bound-tv-maximizer-2}
  \tv\big\{ g_i^\nu\big\}\leq \frac{C_1}{\delta}
\qquad \forall~i\in\mathcal{I}\,,
\qquad\qquad\quad
\tv\big\{a_{ji}\big\}\leq \frac{C_1}{\delta}
\qquad \forall~j,i\,.
  \end{equation}
  Recalling that by Definition~\ref{def-adm-bdr-flux} and because of~\eqref{gnu-def-79}
  we have 
  \begin{equation*}
   g_i^\nu=\gamma_i^\nu\qquad\quad\forall~i\in\mathcal{I},\qquad
   \forall~\nu\ \ \text{large}\,,
  \end{equation*}
  we deduce
  from~\eqref{bound-tv-maximizer-2} that $(A^\nu, g^\nu) \in{\mathcal{D}}^{M_0}$,
  with $M_0\doteq\frac{C_1}{\delta}$, for all $\nu$
  sufficiently large.
  Therefore,  applying Theorem~\ref{compactness-G-M} and relying on Remark~\ref{pointw-conv-flux},
  we deduce that there exist a subsequence, again denoted by
  $\big((A^\nu,g^\nu)\big)_{\!\nu}$, and an element $(A^{\delta},  g^{\delta}) \in
  \mathcal{D}^{M_0}\subset \mathcal{D}$, with $g^{\delta}= f^\mathcal{I}(u(\cdot,0;\, A^{\delta},\gamma^{\delta}))$,
  $(A^{\delta},\gamma^{\delta})\in\mathcal{U}^{M_0}\subset \mathcal{U}$,
   such that 
  \begin{equation}
      \label{Anu-conv-78}
   A^\nu(t) \quad \to \quad A^{\delta}(t)
   \qquad\quad\forall~t\in[0,T]\,,
   \end{equation}
    \begin{equation}
    \label{gnu-conv-77}
%    A^\nu \quad \to \quad A^{\delta_1,\delta_2},\qquad
      g^\nu \quad \to \quad g^{\delta}\quad
    \qquad\textrm{in}\quad \mathbf{L}^1\,.
  \end{equation}
Then, relying on~\eqref{Anu-conv-78}, \eqref{gnu-conv-77}, 
we deduce as in the proof of point ($i$) that
\begin{gather}
\label{gnu-conv-88}
 \int_0^T \mathcal{J}\big(g^{\delta}_1(t),\dots, g^{\delta}_m(t)\big)\,dt
=\lim_\nu  \int_0^T \mathcal{J}\big(g^\nu_1(t),\dots, g^\nu_m(t)\big)\,dt\,,
\\
\noalign{\smallskip}
\label{gnu-conv-89}
\sum_{i=1}^m  \tv \big\{g_i^{\delta} \big\}
\le \liminf_\nu \sum_{i=1}^m  \tv \big\{g_i^\nu \big\}\,,
\\
\noalign{\smallskip}
\label{gnu-conv-90}
      \sum_{j,i=1}^m  \tv \big\{a_{ji}^{\delta} \big\}
\le \liminf_\nu \sum_{j,i=1}^m  \tv \big\{a_{ji}^\nu \big\}\,.
\end{gather}
Hence, by the property of the liminf and limsup operations with respect to the sum and by virtue of~\eqref{gnu-conv-70}, \eqref{gnu-conv-88}-\eqref{gnu-conv-90}, we find
\begin{equation*}
  % \label{gnu-conv-91}
\begin{aligned}
&\sup_{(A,\,\gamma)\,\in\, \mathcal{U}}
    %_{(A,\,\gamma)\,\in\, \mathcal{U}\,}
    \left(
   \int_0^T \mathcal{J}\big(f^{\mathcal{I}}(u(t,0;\, A,\gamma))\big)dt
   -\delta \, \sum_{i=1}^m \tv \left\{f\big(u_i\big(\cdot, 0;\, A, \gamma\big)\big)\right\}
   -\delta \tv\big\{A\big\}
   \right)\le
   \\
   &\qquad\le \lim_\nu  \int_0^T \mathcal{J}\big(g^\nu_1(t),\dots, g^\nu_m(t)\big)\,dt
   +\limsup_\nu\left(
   -\delta \, \sum_{i=1}^m  \tv \big\{g_i^\nu \big\}-\delta \tv\big\{A^\nu\big\}
   \right)
   \\
   &\qquad\le \lim_\nu  \int_0^T \mathcal{J}\big(g^\nu_1(t),\dots, g^\nu_m(t)\big)\,dt
   -\delta \, \sum_{i=1}^m  \liminf_\nu \tv \big\{g_i^\nu \big\}-\delta \sum_{j,i=1}^m \liminf_\nu\tv\big\{a_{ji}^\nu\big\}
   \\
   &\qquad= \int_0^T \mathcal{J}\big(g^{\delta}_1(t),\dots, g^{\delta}_m(t)\big)\,dt
    -\delta\sum_{i=1}^m  \tv \big\{g_i^{\delta} \big\}
    - \delta  \sum_{j,i=1}^m  \tv \big\{a_{ji}^{\delta} \big\}\,,
\end{aligned}
\end{equation*}
proving that $(A^{\delta},  \gamma^{\delta})\in\mathcal{U}$ satisfies~\eqref{eq:max-Jdelta},
completing the proof of the theorem.
\qed

\smallskip

\subsection{Equivalent variational formulations}

We present here an equivalent variational formulation of the
optimization problem~\eqref{min-max-J}
which may be useful also for numerical investigations 
of the optimal solutions as discussed in Section~\ref{sec:num}.
Namely, for every fixed $M>0$, consider the function
\begin{equation}
  \label{IM-def}
  I_M(\delta)= 
  \max_{(A,\,\gamma)\,\in\, \mathcal{U}^M(\overline u)}
  \left(
    \int_0^T \mathcal{J}\big(f^{\mathcal{I}}(u(t,0;\, A,\gamma))\big)dt
    -\delta \, \sum_{i=1}^m \tv
    \left\{f\big(u_i\big(\cdot, 0;\, A, \gamma\big)\big)\right\}
  \right)
\end{equation}
which is  well defined for $\delta>0$ by 
the same arguments of the proof of Theorem~\ref{prop:min-sup}-($ii$).
We shall analyze its limit when the argument vanishes.
\begin{theorem}
  \label{thmdelta}
  Given $\overline u\in  \Pi_{\ell=1}^{m+n} \  \mathbf{L}^\infty
  %\\mathbf{BV} 
  (I_\ell;  \Omega)$, and
  $T>0$,  
  for every fixed $M>0$ one has
  \begin{equation}
    \label{lim-IM}
    \lim_{\delta\to0} I_M(\delta)=
    \max_{(A,\,\gamma)\,\in\, \mathcal{U}^{^M}(\overline u)\,}
    \int_0^T \mathcal{J}\big(f^{\mathcal{I}}(u(t,0;\, A,\gamma))\big)\,dt.
  \end{equation}
  Moreover, given any sequence
  $(A^\nu, \gamma^\nu)\in\mathcal{U}^M(\overline u)$
  of maximizers for $I_M(\delta^\nu)$, with $\delta^\nu \to 0$, 
  there exist a subsequence  $((A^{\nu'}, \gamma^{\nu'}))_{\nu'}$, and 
  $(A^\sharp,\gamma^\sharp)\in\mathcal{U}^M_{max}(\overline u)$, satisfying
  \begin{equation}
    \label{conv-maximizers}
    f^{\mathcal{I}}\big(u(\cdot,0;\ A^{\nu'}, \gamma^{\nu'})\big)
    \quad \to \quad f^{\mathcal{I}}\big(u(\cdot,0;\ A^\sharp, \gamma^\sharp)\big)
    \qquad\textrm{in}\qquad \mathbf{L}^1
  \end{equation}
  and
  \begin{equation}
    \label{AM-opt-sol}
    \sum_{i=1}^m  \tv \left\{f\big(u_i
      \big(\cdot, 0;\, A^\sharp, \gamma^\sharp\big)\big)\right\}=
    \min_{\, (\widehat A,\,\widehat\gamma\,)\,\in\,\mathcal{U}^{^{M}}_{max}
      (\overline u)\ } 
    \sum_{i=1}^m  \tv \left\{f\big(u_i\big(\cdot, 0;\, \widehat A,\widehat
      \gamma\big)\big)\right\}.
  \end{equation}
\end{theorem}
% 
%\begin{proof} 
\noindent
\begin{proof}
%We provide a separate proof of the existence of an optimal solution for~\eqref{min-max-J} and 
%for~\eqref{max-Jdelta}.
  In order to establish
  % ~\eqref{min-max-J}
  the theorem  it will be sufficient to show that, 
  given $\delta^\nu \to 0$ and every sequence
  $(A^\nu, \gamma^\nu)\in\mathcal{U}^M$
   %,with $A^\nu\in\mathcal{A}^M$,  
   such that
   \begin{equation}
   \label{Id-maximizer}
   I_M(\delta^\nu)=
   \int_0^T \mathcal{J}\big(f^{\mathcal{I}}(u(t,0;\, A^\nu,\gamma^\nu))\big)dt
   -\delta^\nu \, \sum_{i=1}^m \tv \left\{f\big(u_i\big(\cdot, 0;\, A^\nu, \gamma^\nu\big)\big)\right\}
   \qquad\forall~\nu\,,
   \end{equation}
   %
   % with $\delta^\nu \to 0$, 
   there exist a subsequence,
  again denoted by $((A^\nu, \gamma^\nu))_\nu$, and 
  $(A^{\sharp},\gamma^{\sharp})\in\mathcal{U}^M_{max}$,
  such that there hold~\eqref{conv-maximizers}, \eqref{AM-opt-sol}
  and
  \begin{gather}
  \label{lim-Idnu}
  \lim_\nu I_M(\delta^\nu)= \max_{(A,\,\gamma)\,\in\, \mathcal{U}^{^M}(\overline u)\,}  \int_0^T \mathcal{J}\big(f^{\mathcal{I}}(u(t,0;\, A,\gamma))\big)\,dt.
%  %
  \end{gather}
  Towards this goal, in connection with the sequence $(A^\nu,\gamma^\nu)\in\mathcal{U}^M$
  satisfying~\eqref{Id-maximizer}, set
  \begin{equation*}
  g^\nu_i= f\big(u_i(\cdot,0;\ A^\nu, \gamma^\nu)\big)\qquad\quad 
   i\in\mathcal{I}\,,  \qquad\forall~\nu\,,
   \end{equation*}
  and consider the sequence of $(A^\nu,g^\nu)\in\mathcal{D}^M$.
  Then, invoking Theorem~\ref{compactness-G-M} and Remark~\ref{pointw-conv-flux} we deduce that
  there exists a subsequence, again denoted by
  $\big((A^\nu,g^\nu)\big)_{\!\nu}$, and an element $(A^\sharp, g^\sharp) \in
  {\mathcal{D}}^M$, with $g^\sharp= f^\mathcal{I}(u(\cdot,0;\, A^\sharp,\gamma^\sharp))$,
  $(A^\sharp,\gamma^\sharp)\in\mathcal{U}^M$,
  so that we have
  \begin{equation}
    \label{gnu-conv-571}
%    A^\nu \quad \to \quad A^*,\qquad
      g^\nu \quad \to \quad g^\sharp\quad
    \qquad\textrm{in}\quad \mathbf{L}^1\,.
  \end{equation}
  Next, relying on Theorem~\ref{prop:min-sup}-($i$),
  consider $(A^*,\gamma^*)\in\mathcal{U}^M_{max}$
   such that~\eqref{eq:min-sup} holds and set $g^*= f^\mathcal{I}(u(\cdot,0;\, A^*,\gamma^*))$. 
   Since $(A^*,\gamma^*)\in\mathcal{U}^M$, 
   by definition~\eqref{IM-def}
   we have
   \begin{equation}
    \label{eq:delta-nu-ineq}
    \int_0^T \mathcal{J}(g^*(t))dt-\delta_\nu
    \sum_{i=1}^m \tv \big\{g^*_i \big\}
     \leq \int_0^T
    \mathcal{J}(g^\nu(t))dt-\delta_\nu \sum_{i=1}^m
    \tv \big\{g^\nu_i\big\}\qquad\forall~\nu\, .
  \end{equation}
  Notice that $\tv\{g_i^\nu\}\le M$ for all $i\in\mathcal{I}$ and for all $\nu$ because
   $(A^\nu,g^\nu)\in\mathcal{D}^M$. Hence, taking the limit in~\eqref{eq:delta-nu-ineq} when $\delta^\nu\to 0$,
   and relying on~\eqref{gnu-conv-571},
   we derive
      \begin{equation*}
    \int_0^T \mathcal{J}(g^*(t))dt 
    \leq \int_0^T \mathcal{J}(g^\sharp(t))dt
      \end{equation*}
    showing that also $(A^\sharp,\gamma^\sharp)\in\mathcal{U}^M_{max}$.
    Thus, one has
    \begin{equation}
    \label{int-gsharp-1}
    \int_0^T \mathcal{J}(g^\sharp(t))dt
    = \max_{(A,\,\gamma)\,\in\, \mathcal{U}^{^M}(\overline u)\,}  \int_0^T \mathcal{J}\big(f^{\mathcal{I}}(u(t,0;\, A,\gamma))\big)\,dt.
    \end{equation}
    With the same arguments, relying on~\eqref{gnu-conv-571}
    and taking the limit in~\eqref{Id-maximizer},  we deduce
    \begin{equation}
    \label{int-gsharp-2}
    \lim_\nu I_M(\delta^\nu) = \int_0^T \mathcal{J}(g^\sharp(t))dt.
    \end{equation}
    Therefore we recover  \eqref{conv-maximizers}, \eqref{lim-Idnu}
    from~\eqref{gnu-conv-571}, \eqref{int-gsharp-1}, \eqref{int-gsharp-2}. In order to complete the proof of the theorem it remains to establish~\eqref{AM-opt-sol}.
     To this end notice that, since $(A^\nu,\gamma^\nu)\in\mathcal{U}^M$
     for all $\nu$ and
      $(A^*,\gamma^*)\in\mathcal{U}^M_{max}$, we have
      \begin{equation*}
      \int_0^T \mathcal{J}(g^\nu(t))dt\le
    \int_0^T \mathcal{J}(g^*(t))dt\qquad\forall~\nu\,.
      \end{equation*}
  which in turn, because of~\eqref{eq:delta-nu-ineq}, implies 
      \begin{equation}
      \label{gnu-tv-bound-81}
      \sum_{i=1}^m
    \tv \big\{g^\nu_i\big\}\le \sum_{i=1}^m \tv \big\{g^*_i \big\} \qquad\forall~\nu\, .
      \end{equation}
      Thus, by the property of the liminf operation and
      relying on~\eqref{gnu-conv-571} as in the proof of Theorem~\ref{prop:min-sup},
      we derive from~\eqref{gnu-tv-bound-81} the estimates
      \begin{equation}
      \label{gsharp-tv-bound}
      \begin{aligned}
      \sum_{i=1}^m
    \tv \big\{g^\sharp_i\big\}&\le \sum_{i=1}^m\liminf_\nu \, \tv \big\{g^\nu_i\big\}
    \\
    \noalign{\smallskip}
    &\le \liminf_\nu \bigg(\sum_{i=1}^m \tv \big\{g^\nu_i\big\}\bigg)
    \le \sum_{i=1}^m \tv \big\{g^*_i \big\}\,.
      \end{aligned}
      \end{equation}
    Since \eqref{eq:min-sup} holds for  $(A^*,\gamma^*)$, we deduce from~\eqref{gsharp-tv-bound} 
    that  $(A^\sharp,\gamma^\sharp)$ verifies \eqref{AM-opt-sol}, completing the proof of the theorem.
     \end{proof}

\noindent
We next consider the function
\begin{equation}
  \label{I-def}
  I(\delta) = \!\!\!\!\!
  \max_{(A,\,\gamma)\,\in\, \mathcal{U}(\overline u)}
  \!\left(
    \int_0^T \!\!\!\mathcal{J}\big(f^{\mathcal{I}}(u(t,0;\, A,\gamma))\big)dt
    -\delta \, \sum_{i=1}^m \tv \left\{f\big(u_i\big(\cdot, 0;\, A,
      \gamma\big)\big)\right\}
    -\delta\tv\big\{A\big\}
  \right),
\end{equation}
which is  well defined for $\delta>0$ by Theorem~\ref{prop:min-sup}-(ii).
We shall analyze its relation with the optimization problem
\begin{equation}
  \label{max-M-J}
%   \sup_{\ M>0\ } \sup_{(A,\,\gamma)\,\in\, \mathcal{U}^{^M}(\overline u)\,}  
\sup_{(A,\,\gamma)\,\in\, \mathcal{U}(\overline u)\,} 
   \int_0^T \mathcal{J}\big(f^{\mathcal{I}}(u(t,0;\, A,\gamma))\big)dt\,.
    \end{equation}
Here, we are maximizing the integral functional of the flux-traces among all admissible controls without imposing a uniform bound on the total variation.
%  with
% arbitrarily large total variation.
%
\begin{theorem}
  \label{thmdelta-sup}
  Given $\overline u\in  \Pi_{\ell=1}^{m+n} \  \mathbf{L}^\infty
  %\mathbf{BV} 
  (I_\ell;  \Omega)$, and
$T>0$,  one has
\begin{equation}
 \label{lim-I-sup}
  \lim_{\delta\to 0} I(\delta) = \sup_{(A,\,\gamma)\,\in\, \mathcal{U}(\overline u)\,}
  %= \sup_{\ M>0\ } \sup_{(A,\,\gamma)\,\in\, \mathcal{U}^{^M}(\overline u)\,}  
   \int_0^T \mathcal{J}\big(f^{\mathcal{I}}(u(t,0;\, A,\gamma))\big)dt\,.
  \end{equation}
More precisely, one of the following two cases holds:
\begin{itemize}
 \item[($i$)] there exists $M_0>0$ such that
\begin{equation*}
  % \label{lim-I-sup-2}
  \lim_{\delta\to 0} I(\delta) 
  = \max_{(A,\,\gamma)\,\in\, \mathcal{U}^{M}(\overline u)\,}
   \int_0^T \mathcal{J}\big(f^{\mathcal{I}}(u(t,0;\, A,\gamma))\big)dt
    \qquad\forall~M>M_0\,,
%  = \max_{(A,\,\gamma)\,\in\, \mathcal{U}(\overline u)\,}
%   \int_0^T \mathcal{J}\big(f^{\mathcal{I}}(u(t,0;\, A,\gamma))\big)dt\,.
  \end{equation*}
  and the supremum in~\eqref{max-M-J} is attained as a maximum;
  \item[($ii$)] 
\begin{equation*}
  % \label{lim-I-sup-3}
  \lim_{\delta\to 0} I(\delta) 
  > \max_{(A,\,\gamma)\,\in\, \mathcal{U}^{M}(\overline u)\,}
   \int_0^T \mathcal{J}\big(f^{\mathcal{I}}(u(t,0;\, A,\gamma))\big)dt
  \qquad\forall~M>0\,,
  \end{equation*}
  and the optimization problem~\eqref{max-M-J} does not admit a maximum.
\end{itemize}
\end{theorem}
\begin{proof}
  Recalling the definitions~\eqref{admiss-controls} and
  because of Theorem~\ref{exmas}
  it follows that
\begin{equation}
\label{sup-J-max-M-J-1}
\sup_{(A,\,\gamma)\,\in\, \mathcal{U}(\overline u)\,} 
   \int_0^T \mathcal{J}\big(f^{\mathcal{I}}(u(t,0;\, A,\gamma))\big)dt
= \sup_{\ M>0\ } \max_{(A,\,\gamma)\,\in\, \mathcal{U}^{^M}(\overline u)\,}  
   \int_0^T \mathcal{J}\big(f^{\mathcal{I}}(u(t,0;\, A,\gamma))\big)dt\,,
\end{equation}
and that
\begin{equation*}
\label{sup-J-max-M-J-2}
M \quad \mapsto \quad \max_{(A,\,\gamma)\,\in\, \mathcal{U}^{^M}(\overline u)\,}  
   \int_0^T \mathcal{J}\big(f^{\mathcal{I}}(u(t,0;\, A,\gamma))\big)dt
\end{equation*}
is a non decreasing map. Hence,
%By~\eqref{sup-J-max-M-J-1}-\eqref{sup-J-max-M-J-2}
 in order to establish the theorem it will be sufficient to prove~\eqref{lim-I-sup}.
To this end notice first that the map $\delta\to I(\delta)$ is non increasing. In fact, let $0<\delta_1<\delta_2$
and, by Theorem~\ref{prop:min-sup}-($ii$), consider a maximizer $(A^{\delta_2}, \gamma^{\delta_2})\in\mathcal{U}$ for $I(\delta_2)$.
Then, by definition~\eqref{I-def}, we find
\begin{equation*}
  % \label{I-decr}
\begin{aligned}
I(\delta_2)&=
\int_0^T \mathcal{J}\big(f^{\mathcal{I}}(u(t,0;\, A^{\delta_2},\gamma^{\delta_2}))\big)dt
   -\delta_2 \, \sum_{i=1}^m \tv \left\{f\big(u_i\big(\cdot, 0;\, A^{\delta_2}, \gamma^{\delta_2}\big)\big)\right\}
   -\delta_2\tv\big\{A^{\delta_2}\big\}
   \\
   \noalign{\smallskip}
   &\le
   \int_0^T \mathcal{J}\big(f^{\mathcal{I}}(u(t,0;\, A^{\delta_2},\gamma^{\delta_2}))\big)dt
   -\delta_1 \, \sum_{i=1}^m \tv \left\{f\big(u_i\big(\cdot, 0;\, A^{\delta_2}, \gamma^{\delta_2}\big)\big)\right\}
   -\delta_1\tv\big\{A^{\delta_2}\big\}
   \\
   \noalign{\smallskip}
   &\le I(\delta_1)\,.
\end{aligned}
\end{equation*}
Therefore, the limit in~\eqref{lim-I-sup} exists.
Next observe that by definitions~\eqref{IM-def}, \eqref{I-def}
we have
\begin{equation*}
I(\delta) \geq I_M(\delta)\qquad\forall~M, \delta>0\,.
\end{equation*}
Hence,
taking first the limit as $\delta\to 0$ in both sides of the inequality, next considering the supremum over $M>0$
in the right-hand side
and relying on~\eqref{lim-IM}, \eqref{sup-J-max-M-J-1}, we find
\begin{equation}
\label{lim-I-sup-42}
\lim_{\delta\to 0} I(\delta) \geq 
\sup_{(A,\,\gamma)\,\in\, \mathcal{U}(\overline u)\,} 
   \int_0^T \mathcal{J}\big(f^{\mathcal{I}}(u(t,0;\, A,\gamma))\big)dt\,.
%\max_{(A,\,\gamma)\,\in\, \mathcal{U}^{^M}(\overline u)\,}  \int_0^T \mathcal{J}\big(f^{\mathcal{I}}(u(t,0;\, A,\gamma))\big)\,dt\qquad\forall~M>0\,,
\end{equation}
On the other hand, by definitions~\eqref{IM-def}, \eqref{I-def} there hold
\begin{equation}
\label{Id-IMd}
I(\delta)\le \sup_M I_M(\delta)\,,
\end{equation}
\begin{equation}
\label{IM-42}
I_M(\delta) \leq \max_{(A,\,\gamma)\,\in\, \mathcal{U}^{^M}(\overline u)\,}  \int_0^T \mathcal{J}\big(f^{\mathcal{I}}(u(t,0;\, A,\gamma))\big)\,dt.
\qquad\forall~M, \delta>0\,.
\end{equation}
Thus, taking first the supremum over $M>0$ in both sides of~ \eqref{IM-42}
and relying on~\eqref{sup-J-max-M-J-1}, \eqref{Id-IMd},
next considering the
limit as $\delta\to 0$ in the left-hand side,
 we find
\begin{equation*}
  %\label{lim-I-sup-43}
\lim_{\delta\to 0} I(\delta) \leq 
\sup_{(A,\,\gamma)\,\in\, \mathcal{U}(\overline u)\,} 
   \int_0^T \mathcal{J}\big(f^{\mathcal{I}}(u(t,0;\, A,\gamma))\big)dt\,,
\end{equation*}
which together with~\eqref{lim-I-sup-42} yields~\eqref{lim-I-sup}, completing the proof of the theorem.
\end{proof}
\medskip

\section{Numerical simulations for a node\\
  with two incoming
  and two outgoing arcs}\label{sec:num} 

This section is devoted to present few numerical simulations in the case of a
node $J$ with two incoming and two outgoing arcs.
These numerical results seem to indicate the fact that 
the maximisation problem~\eqref{max-J}
may well have no optimal solution within the
ones  constructed by the standard Junction Riemann Solver
%Riemann solvers at $J$ 
\cite{b-y, CGP, gp2, gp4, hr}.
%, in the sense they maximize functional~(\ref{max-J}).

In every numerical example, we model the incoming roads
$I_1$ and $I_2$ through the real interval
$(-5,0)$, while the outgoing roads $I_3$ and $I_4$ are described by the
interval $(0,5)$, so that the node is located at $x=0$;
see Figure~\ref{fig:junction-2x2}.
\begin{figure}
  \centering
    \begin{tikzpicture}[line cap=round,line join=round,>=triangle
    45,x=1cm,y=1cm]
    \clip(0,0) rectangle (11.1,3.1);
    % \draw[help lines,xstep=1,ystep=1] (0,0) grid (11,8);

    \draw[fill] (5.5, 1.5) circle (.12cm);

    \draw (4., 2.25) -- (5.5, 1.5) -- (4., 0.75);

    \draw (8.5, 3.) -- (7., 2.25);
    \draw (8.5, 0.) -- (7., 0.75);

    \draw[->] (2.5, 3.) -- (4., 2.25);
    \draw[->] (2.5, 0.) -- (4., .75);

    \draw[<->] (7., 2.25) -- (5.5, 1.5) -- (7., 0.75);
    \node[inner sep = 0, anchor = north] at (3.5, 2.3) {$I_1$};
    \node[inner sep = 0, anchor = south] at (3.5, .7) {$I_2$};

    \node[inner sep = 0, anchor = north] at (7.5, 2.3) {$I_3$};
    \node[inner sep = 0, anchor = south] at (7.5, .7) {$I_4$};

    \node[inner sep = 0, anchor = west] at (0.5, 1.5) {$\mathcal I = 
      \left\{1, 2\right\}$};
    \node[inner sep = 0, anchor = east] at (10.5, 1.5) {$\mathcal O = 
      \left\{3, 4\right\}$};
  \end{tikzpicture}
  \caption{The junction considered in Section~\ref{sec:num}.
    The incoming roads are denoted by $I_1$ and $I_2$, while the outgoing
    ones are $I_3$ and $I_4$.}
  \label{fig:junction-2x2}
\end{figure}
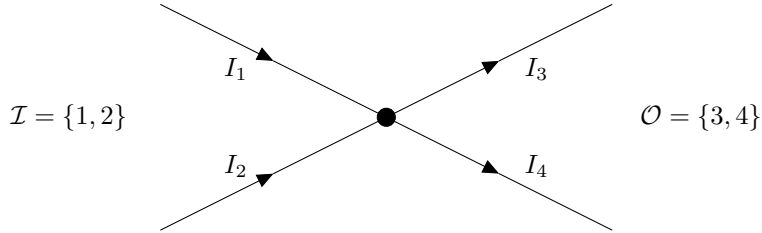

In each arc, the Lighthill-Whitham-Richards model~(\ref{conlaw})
is considered with the
flux function given by $f(u) = 4 u (1-u)$, so that the
set $\Omega$ of all the possible densities  is the interval $[0,1]$.
The solution in each arc is computed by the Godunov method;
see for example~\cite[Section~12.1]{leveque}, \cite[Chapter~III]{god-rav}
or~\cite[Chapter~3]{holden-risebro}. We use a uniform spatial mesh
with length $\Delta x = 0.05$ and a non-uniform time mesh with length
$\Delta t$, calculated in such a way the classic
CFL condition is satisfied; see~\cite{c-f-l}.
Note that it is also possible using a uniform time mesh. However,
in view of future numerical studies, we
choice a non uniform one for improving convergence.
The simulations are done
in the time interval $(0,T)$.

%As regards the numeric 
Concerning the optimization of~(\ref{max-J}),
we assume that the distributional matrix $A$ satisfying~(\ref{distr-rules})
is a priori fixed.
In such a way, in the optimization problem~(\ref{max-J})
we only regard  as junction controls
the incoming fluxes $g = \left(g_1, g_2\right)$.
In order to numerically approximate an 
%reasonable 
optimal control,
we perform a heuristic recursive procedure based on the following steps.
\begin{enumerate}
\item Consider a piecewise constant initial control $g = (g_1, g_2)$,
  with a fixed number, namely $M$, of points of discontinuity.

\item On each time interval where the control is constant, we perform a set of
  variations. For each variation, we compute the corresponding 
  solution and cost~(\ref{eq:numeric-cost}).

\item Among all the possible variations, we select one that
  maximizes~(\ref{eq:numeric-cost}). In this way, we obtain
  a piecewise constant control
  $\overline{g}$, which acts not worse than $g$.
\end{enumerate}
According to the variational formulation of the min-max problem given
by Theorem~\ref{thmdelta}, for finding a solution which maximizes~(\ref{max-J})
and minimizes~(\ref{min-max-J}), we  shall maximize the cost
\begin{equation}
  \label{eq:numeric-cost}
  \int_0^T \mathcal J(f(u_1(t, 0)),f(u_2(t, 0))) dt 
  - \delta \sum_{i=1}^2 \tv_{[0,T]}f(u_i(\cdot, 0))
\end{equation}
with $\delta>0$ sufficiently small or $\delta = 0$.
%  and with 
% $\mathcal J(f_1, f_2) = f_1 + f_2$.

\subsection{Rarefaction and shock approaching the node (case $1$)}
\label{sse:shock-rar}

% 2x2/simulations/shock-rarefaction/delta02/sum/ramp3

Here we consider the case of a rarefaction and a shock
in an incoming arc interacts
with the node. The initial data are given by
\begin{equation*}
  \overline u_1(x) = \left\{
    \begin{array}{l@{\quad}l}
      0.47 & \textrm{ if } x < -2.1,
      \\
      0.25 & \textrm{ if } -2.1 < x < -1,
      \\
      0.5 & \textrm{ if } x > -1,
    \end{array}
  \right. \qquad
  \overline u_2(x) = 0.5, \qquad
  \overline u_3(x) = 0.1, \qquad
  \overline u_4(x) = 0.1, 
\end{equation*}
the distribution matrix $A$ is given by
\begin{equation*}
  A = \left(
    \begin{array}{cc}
      0.5 & 0.3
      \\
      0.5 & 0.7
    \end{array}
  \right),
\end{equation*}
and $\mathcal J(f_1, f_2) = f_1 + f_2$.
Moreover, we consider the following choice of parameters:
$T = 5$, $M = 20$ and $\delta = 0.2$.
In Figure~\ref{fig:s4.1:optimal-sol} the numerical optimal solution is
represented. Note that the solution is qualitatively different from 
that obtained with the classical Riemann solver
at the node; see Figure~\ref{fig:s4.1:RS-sol}. 
Indeed at about $t \sim 1$ a shock wave with negative speed
in the arc $I_1$ is generated. This wave has almost zero speed, at time
$t \sim 2.5$ enters a little more in the domain and comes back to the
boundary at time $t \sim 4$.
The numeric cost~(\ref{eq:numeric-cost}) for the optimal solution
is approximately $8.415$, since
\begin{equation*}
    \int_0^T \!\!\mathcal J(f(u_1(t, 0)),f(u_2(t, 0))) dt \sim 8.498
    \quad
    \tv_{[0,T]}f(u_1(\cdot, 0)) \sim 0.250
    \quad
    \tv_{[0,T]}f(u_2(\cdot, 0)) \sim 0.163,
\end{equation*}
while the cost for the solution obtained with the corresponding Riemann
solvers, see Figure~\ref{fig:s4.1:RS-sol}, is approximately $8.389$.
In Figures~\ref{fig:ss3:optimal-fluxes-sol} and~\ref{fig:ss3:RS-fluxes}
respectively the optimal fluxes and the fluxes obtained with the Riemann
solver are drawn. Note in particular that the optimal fluxes have less
total variation than the fluxes obtained with the Riemann solver.
\begin{figure}
  \centering
  \includegraphics[width=7cm, height = 3cm]{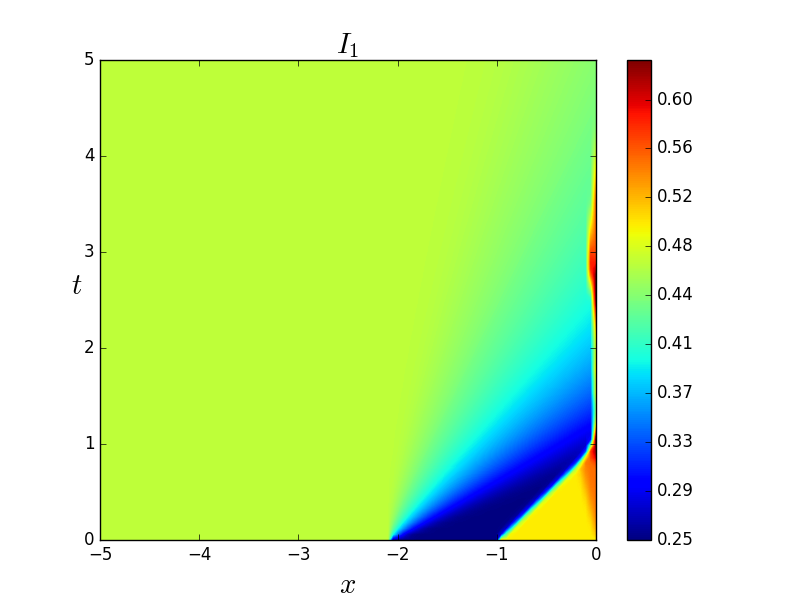}
  \includegraphics[width=7cm, height = 3cm]{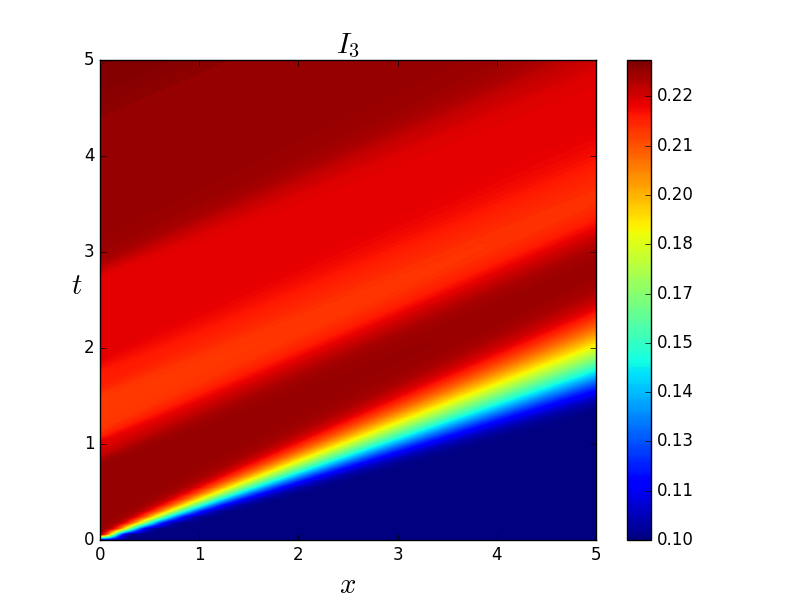}
  \\
  \includegraphics[width=7cm, height = 3cm]{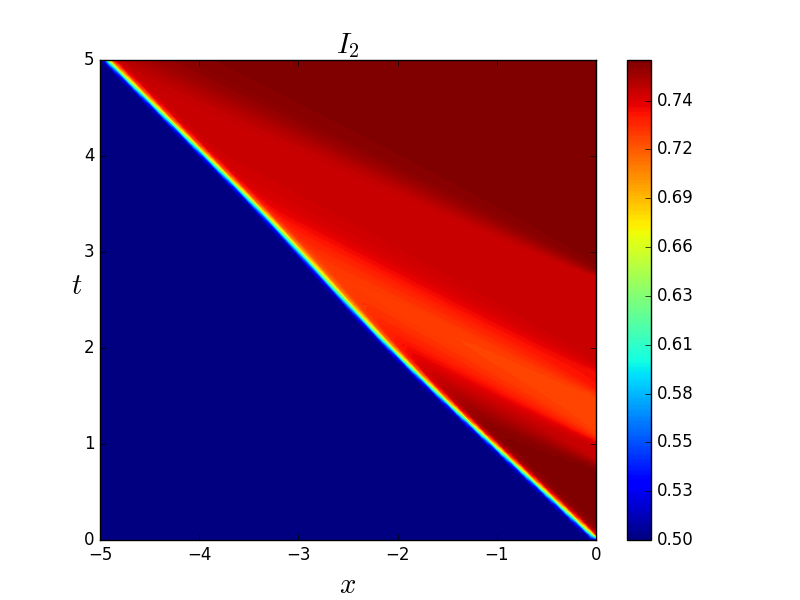}
  \includegraphics[width=7cm, height = 3cm]{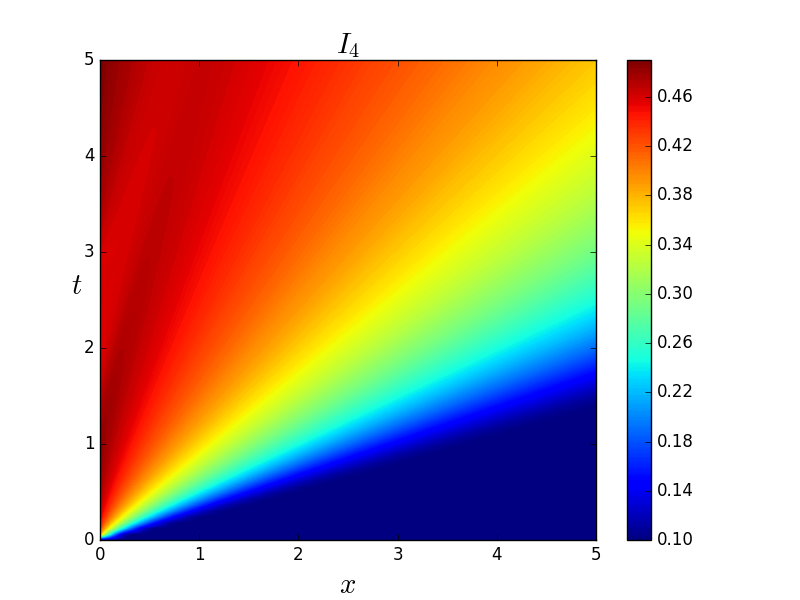}
  \caption{The numerical optimal solution obtained in
    Subsection~\ref{sse:shock-rar}. In the first line, the solutions in
    $I_1$ and $I_3$. In the second line, the solutions in
    $I_2$ and $I_4$. (Color online)}
  \label{fig:s4.1:optimal-sol}
\end{figure}
% \begin{figure}
%   \centering
%   \includegraphics[width=7cm, height = 3cm]{figures/4.1/Contour-I1-range.png}
%   \includegraphics[width=7cm, height = 3cm]{figures/4.1/Contour-I3-range.png}
%   \\
%   \includegraphics[width=7cm, height = 3cm]{figures/4.1/Contour-I2-range.png}
%   \includegraphics[width=7cm, height = 3cm]{figures/4.1/Contour-I4-range.png}
%   \caption{The numerical optimal solution obtained in
%     Subsection~\ref{sse:shock-rar}. In the first line, the solutions in
%     $I_1$ and $I_3$. In the second line, the solutions in
%     $I_2$ and $I_4$. (Same range)}
%   \label{fig:s4.1:optimal-sol}
% \end{figure}
\begin{figure}
  \centering
  \includegraphics[width=7cm, height = 3cm]{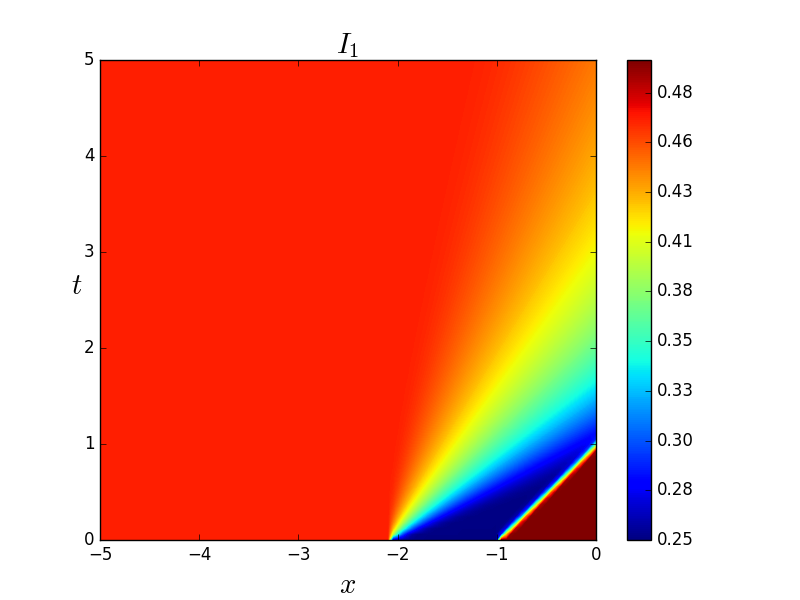}
  \includegraphics[width=7cm, height = 3cm]{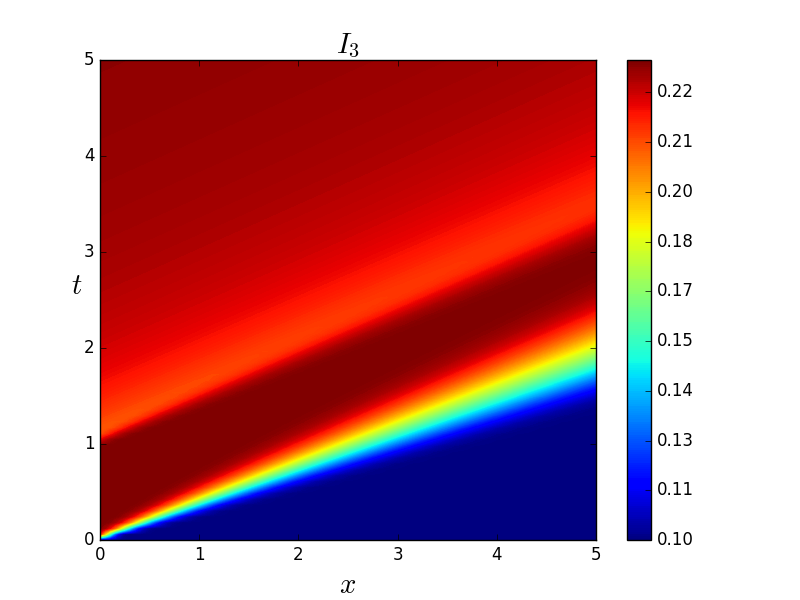}
  \\
  \includegraphics[width=7cm, height = 3cm]{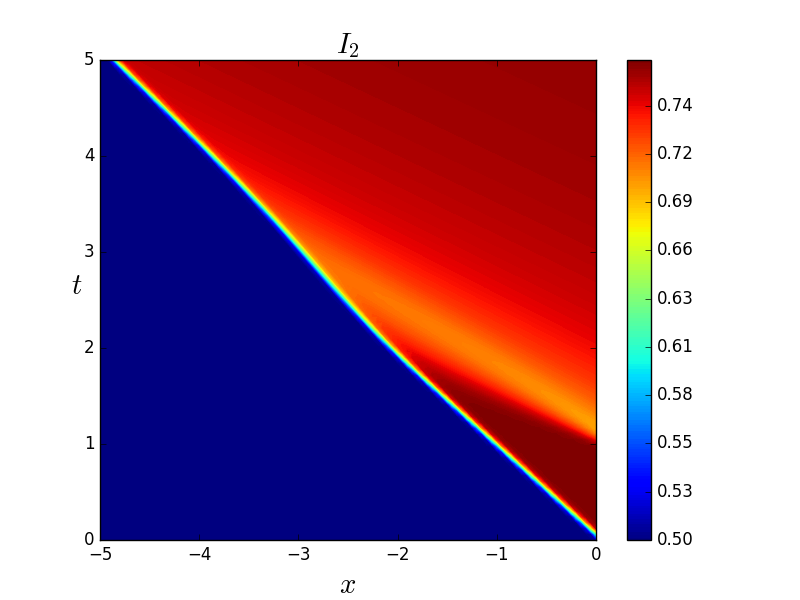}
  \includegraphics[width=7cm, height = 3cm]{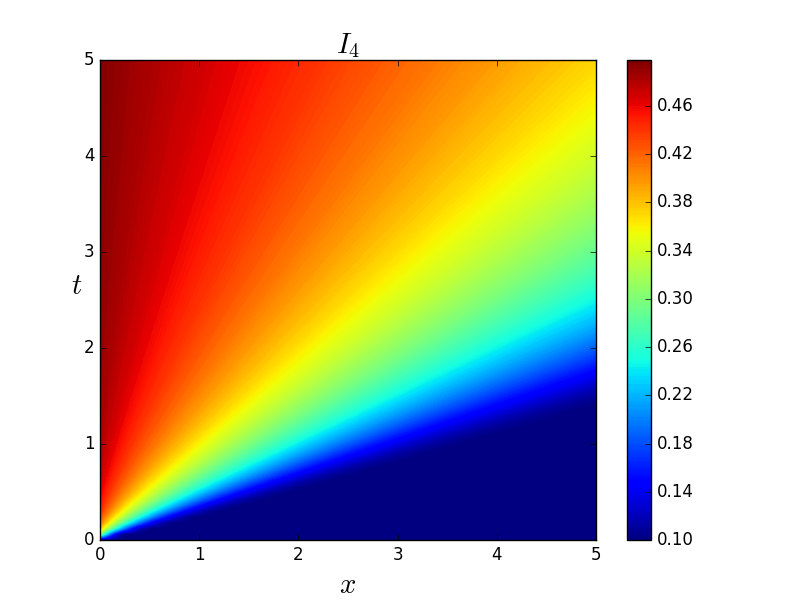}
  \caption{The solution, for Subsection~\ref{sse:shock-rar},
    obtained with the Riemann solver. In the first line, the solutions in
    $I_1$ and $I_3$. In the second line, the solutions in
    $I_2$ and $I_4$.  (Color online)}
  \label{fig:s4.1:RS-sol}
\end{figure}
% \begin{figure}
%   \centering
%   \includegraphics[width=7cm, height = 3cm]{figures/4.1/Contour-classic-I1-range.png}
%   \includegraphics[width=7cm, height = 3cm]{figures/4.1/Contour-classic-I3-range.png}
%   \\
%   \includegraphics[width=7cm, height = 3cm]{figures/4.1/Contour-classic-I2-range.png}
%   \includegraphics[width=7cm, height = 3cm]{figures/4.1/Contour-classic-I4-range.png}
%   \caption{(SAME RANGE) The solution, for Subsection~\ref{sse:shock-rar},
%     obtained with the Riemann solver. In the first line, the solutions in
%     $I_1$ and $I_3$. In the second line, the solutions in
%     $I_2$ and $I_4$.}
%   \label{fig:s4.1:RS-sol}
% \end{figure}
\begin{figure}
  \centering
  \includegraphics[width=5cm, height = 2cm]{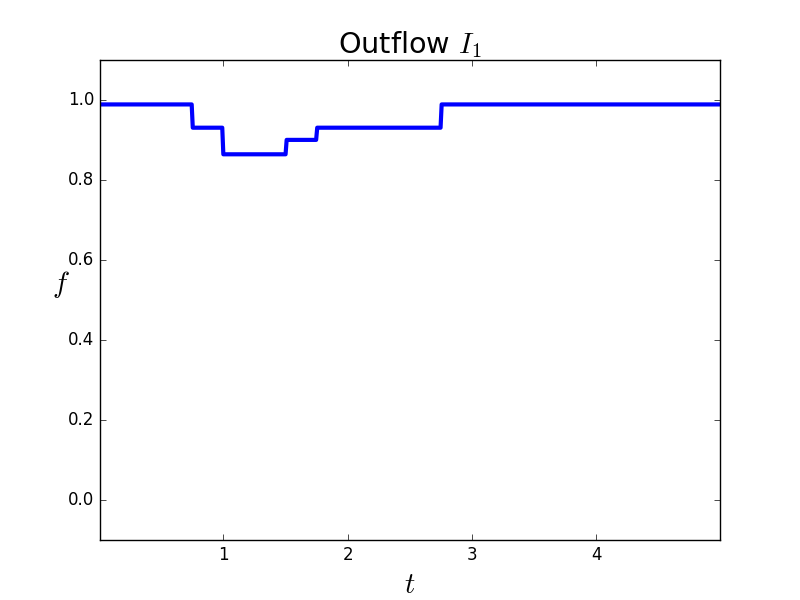}
  \includegraphics[width=5cm, height = 2cm]{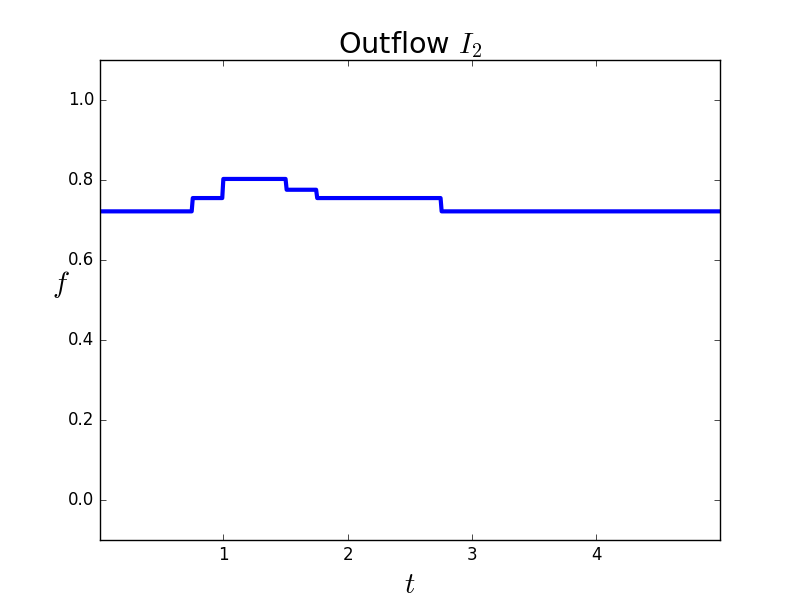}
  \caption{The numerical optimal fluxes at the node in the incoming arcs $I_1$
    and $I_2$, obtained in
    Subsection~\ref{sse:shock-rar}.}
  \label{fig:ss3:optimal-fluxes-sol}
\end{figure}
\begin{figure}
  \centering
  \includegraphics[width=5cm, height = 2cm]{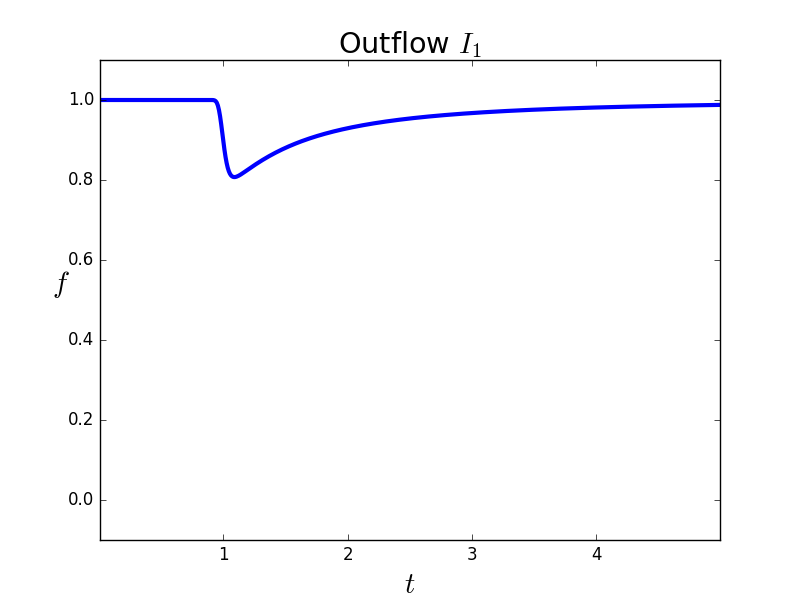}
  \includegraphics[width=5cm, height = 2cm]{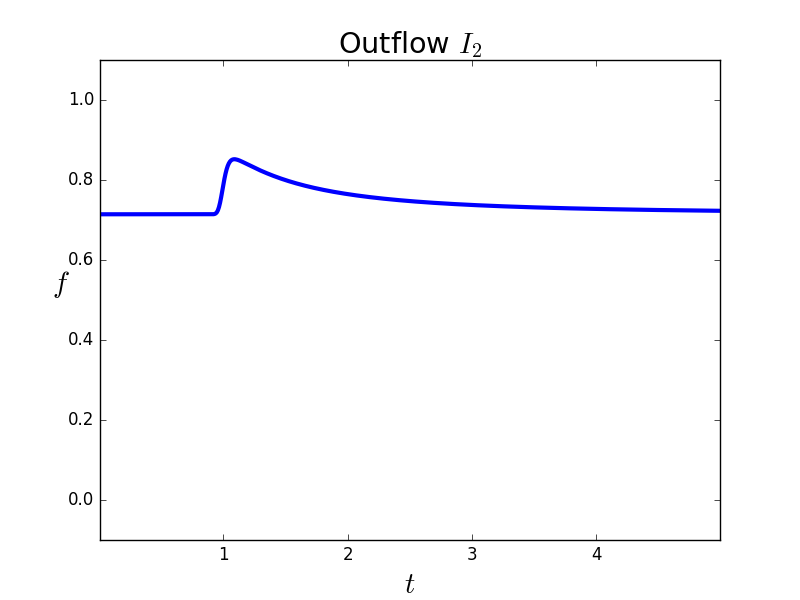}
  \caption{The fluxes at the node in the incoming arcs of solution,
    for Subsection~\ref{sse:shock-rar},
    obtained with the Riemann solver.}
  \label{fig:ss3:RS-fluxes}
\end{figure}

\subsection{Rarefaction and shock approaching the node (case $2$)}
\label{sse:constant-ic}

% 2x2/simulations/shock-rarefaction/delta04/product/ramp2
% 2x2/simulations/test (new one)

As in Subsection~\ref{sse:shock-rar}, we present the case of a rarefaction
and a shock approaching the node from the incoming arc $I_1$.
More precisely, we have that the initial data are given by
\begin{equation*}
  \overline u_1(x) = \left\{
    \begin{array}{ll}
      0.25 & \textrm{ if } x < -2. \textrm{ or } -1.5 < x < -1,
      \\
      0. & \textrm{ if } -2. < x < -1.5 \textrm{ or } -1. < x < -.5,
      \\
      0.5 & \textrm{ if } x > -.5,
    \end{array}
  \right. \qquad
  \begin{array}{l}
    \overline u_2(x) = 0.5, 
    \\
    \overline u_3(x) = 0.1, 
    \\
    \overline u_4(x) = 0.1, 
  \end{array}
\end{equation*}
the distribution matrix is given by
\begin{equation*}
  A = \left(
    \begin{array}{cc}
      0.5 & 0.3
      \\
      0.5 & 0.7
    \end{array}
  \right)
\end{equation*}
We here maximize the product of the incoming fluxes, i.e 
$\mathcal J(f_1, f_2) = 50 \cdot f_1 f_2$.
Moreover, we consider the following parameters:
$T = 1.1$, $M = 20$ and $\delta = 0.4$.
In Figure~\ref{fig:s3.2:optimal-fluxes} the numerical optimal solution is
represented. Note that two shocks are generated at the junction at
times $t \sim 0.2$ and $t \sim 0.4$.
These shocks interact with shocks and rarefaction waves generated by the
initial datum and come back to the junction.
The numeric cost~(\ref{eq:numeric-cost}) for this solution
is approximately $32.39$, since
\begin{equation*}
    \int_0^T \!\!\mathcal J(f(u_1(t, 0)),f(u_2(t, 0))) dt \sim 32.68
    \quad
    \tv_{[0,T]}f(u_1(\cdot, 0)) \sim 0.91
    \quad
    \tv_{[0,T]}f(u_2(\cdot, 0)) \sim 0.53,
\end{equation*}
while the cost for the solution obtained with the corresponding Riemann
solvers, see Figure~\ref{fig:s4.2:RS-sol}, is approximately $30.73$.
In Figures~\ref{fig:ss1:optimal-fluxes-sol} and~\ref{fig:ss1:RS-fluxes}
respectively the optimal fluxes and the fluxes obtained with the Riemann
solver are drawn. Again the optimal fluxes have less
total variation than the fluxes obtained with the Riemann solver.
We observe that, in this case compared with case 1,
the optimal cost has a
larger difference from the cost associated to the solution constructed with
the classical Riemann solver. This is due both to the choice of $\mathcal J$
and to the choice of the initial datum, producing a more complex wave pattern.
\begin{figure}
  \centering
  \includegraphics[width=7cm, height = 3cm]{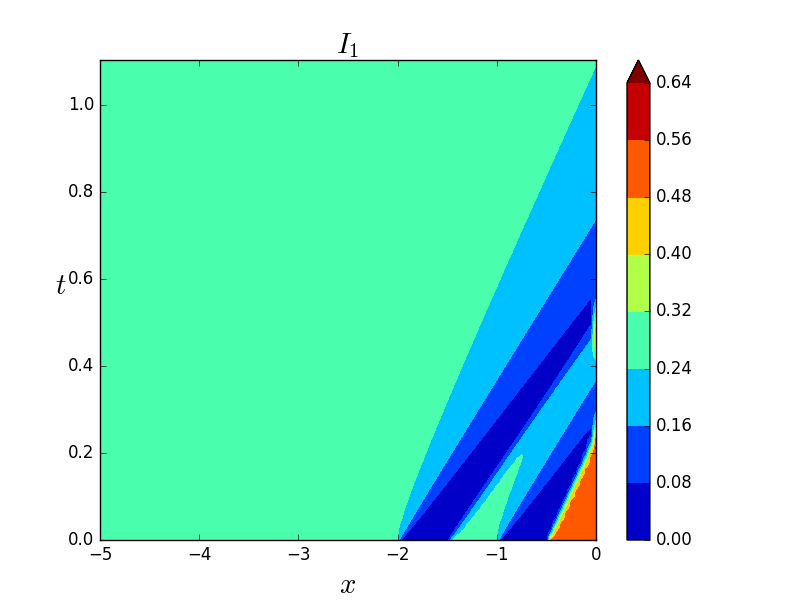}
  \includegraphics[width=7cm, height = 3cm]{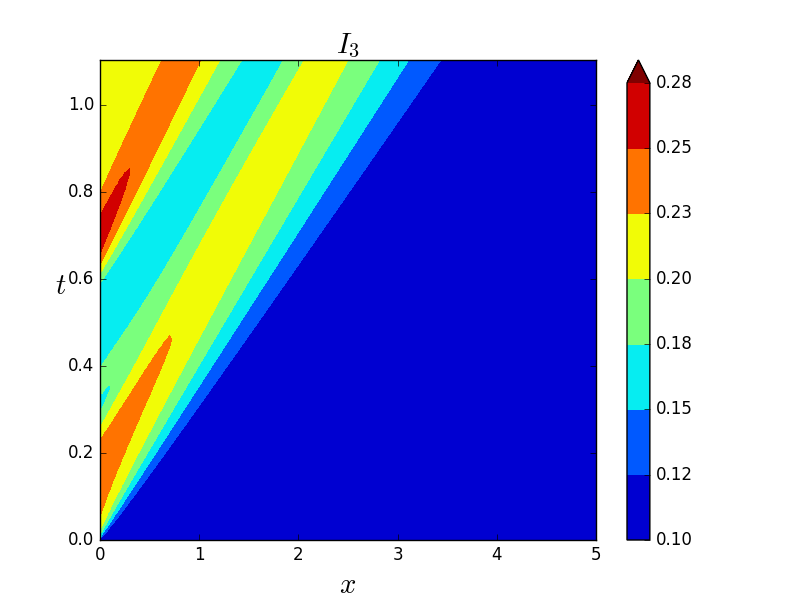}
  \\
  \includegraphics[width=7cm, height = 3cm]{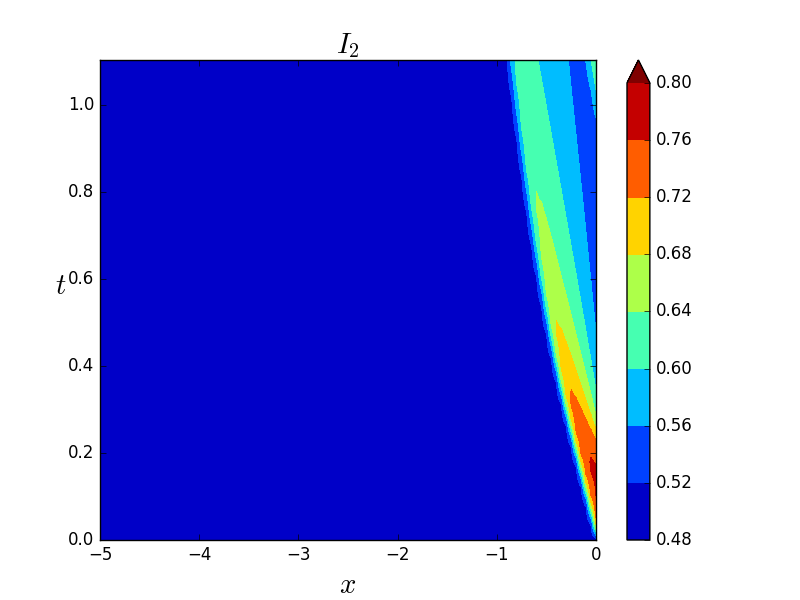}
  \includegraphics[width=7cm, height = 3cm]{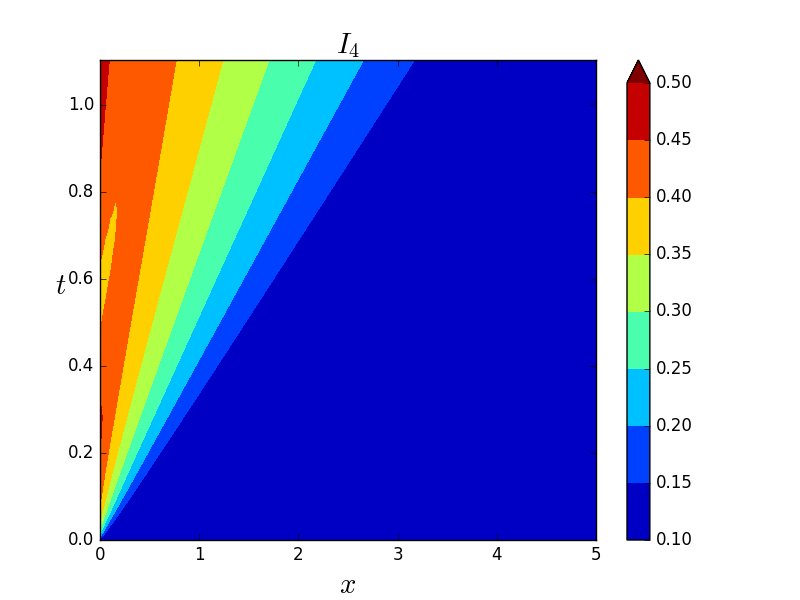}
  \caption{The numerical optimal solution obtained in
    Subsection~\ref{sse:constant-ic}. In the first line, the solutions in
    $I_1$ and $I_3$. In the second line, the solutions in
    $I_2$ and $I_4$. (Color online)}
  \label{fig:s3.2:optimal-fluxes}
\end{figure}
% \begin{figure}
%   \centering
%   \includegraphics[width=7cm, height = 3cm]{figures/4.2/Contour-I1-range.png}
%   \includegraphics[width=7cm, height = 3cm]{figures/4.2/Contour-I3-range.png}
%   \\
%   \includegraphics[width=7cm, height = 3cm]{figures/4.2/Contour-I2-range.png}
%   \includegraphics[width=7cm, height = 3cm]{figures/4.2/Contour-I4-range.png}
%   \caption{(SAME RANGE) The numerical optimal solution obtained in
%     Subsection~\ref{sse:constant-ic}. In the first line, the solutions in
%     $I_1$ and $I_3$. In the second line, the solutions in
%     $I_2$ and $I_4$.}
%   \label{fig:s3.2:optimal-fluxes}
% \end{figure}
\begin{figure}
  \centering
  \includegraphics[width=7cm, height = 3cm]{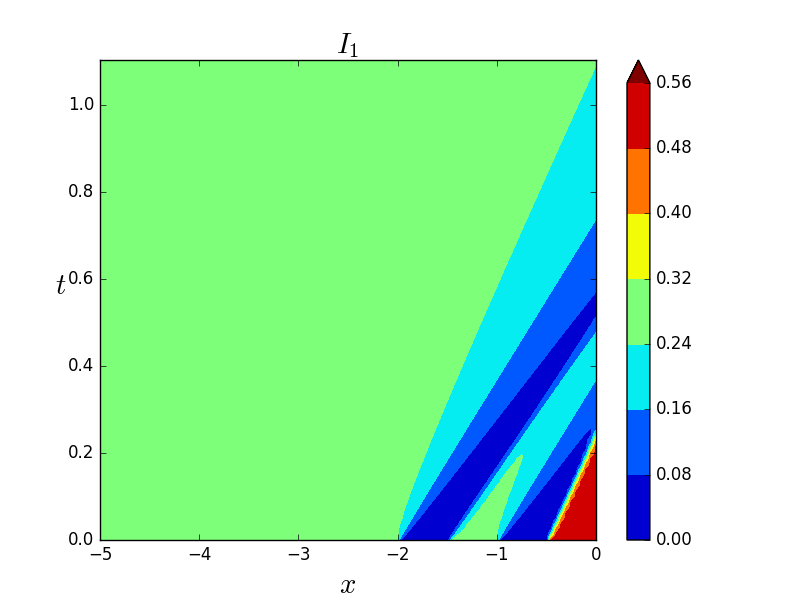}
  \includegraphics[width=7cm, height = 3cm]{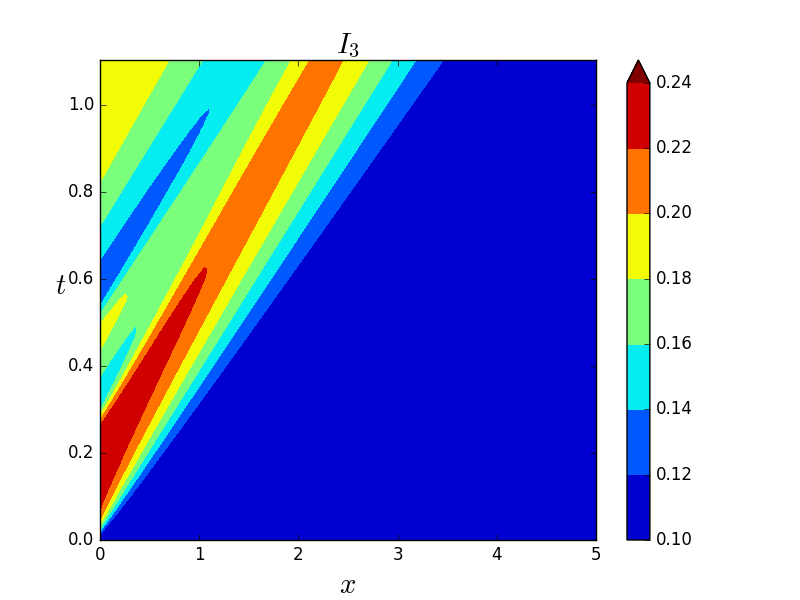}
  \\
  \includegraphics[width=7cm, height = 3cm]{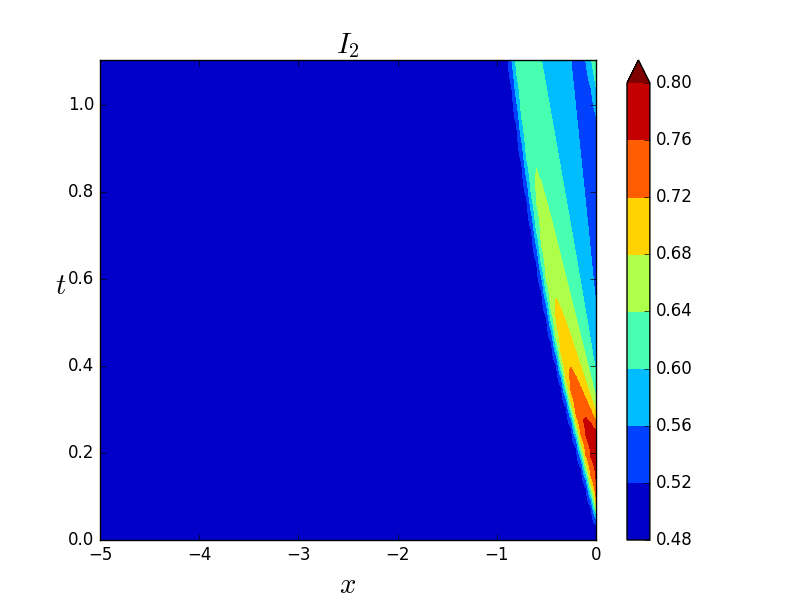}
  \includegraphics[width=7cm, height = 3cm]{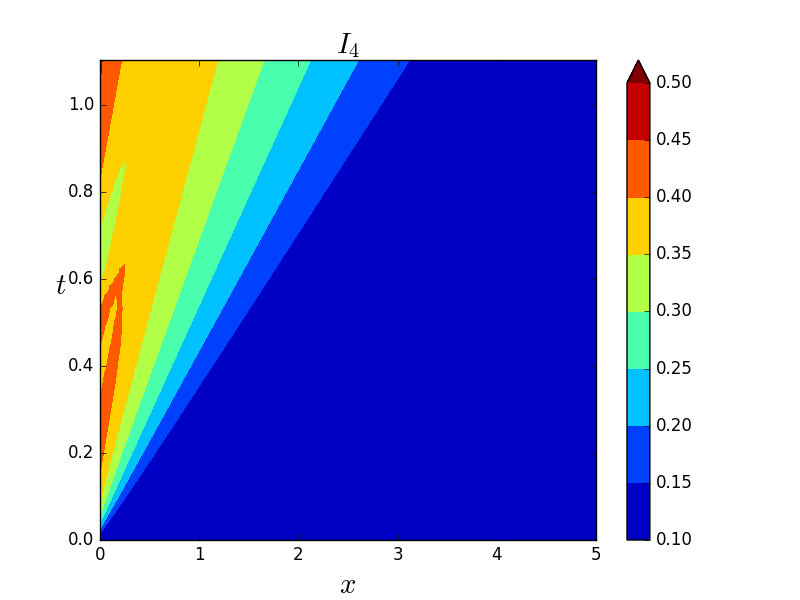}
  \caption{The solution, for Subsection~\ref{sse:constant-ic},
    obtained with the Riemann solver. In the first line, the solutions in
    $I_1$ and $I_3$. In the second line, the solutions in
    $I_2$ and $I_4$. (Color online)}
  \label{fig:s4.2:RS-sol}
\end{figure}
% \begin{figure}
%   \centering
%   \includegraphics[width=7cm, height = 3cm]{figures/4.2/Contour-classic-I1-range.png}
%   \includegraphics[width=7cm, height = 3cm]{figures/4.2/Contour-classic-I3-range.png}
%   \\
%   \includegraphics[width=7cm, height = 3cm]{figures/4.2/Contour-classic-I2-range.png}
%   \includegraphics[width=7cm, height = 3cm]{figures/4.2/Contour-classic-I4-range.png}
%   \caption{(SAME RANGE) The solution, for Subsection~\ref{sse:constant-ic},
%     obtained with the Riemann solver. In the first line, the solutions in
%     $I_1$ and $I_3$. In the second line, the solutions in
%     $I_2$ and $I_4$.}
%   \label{fig:s4.2:RS-sol}
% \end{figure}

\begin{figure}
  \centering
  \includegraphics[width=5cm, height = 2cm]{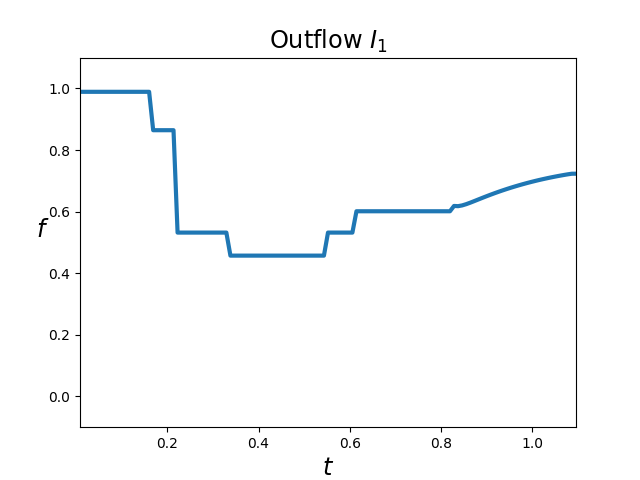}
  \includegraphics[width=5cm, height = 2cm]{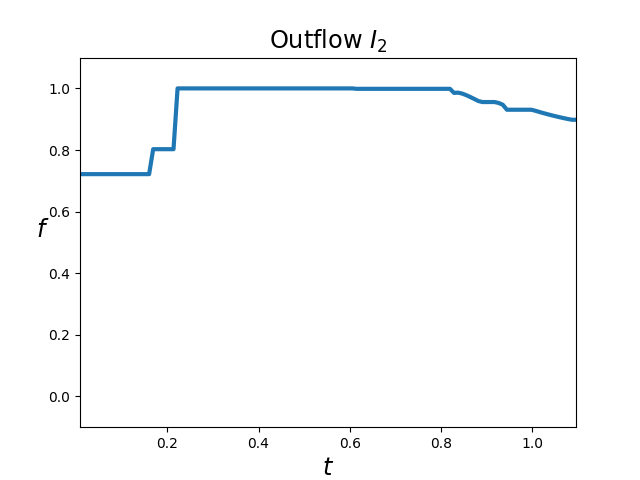}
  \caption{The numerical optimal fluxes at the node in the incoming arcs $I_1$
    and $I_2$, obtained in
    Subsection~\ref{sse:constant-ic}.}
  \label{fig:ss1:optimal-fluxes-sol}
\end{figure}
\begin{figure}
  \centering
  \includegraphics[width=5cm, height = 2cm]{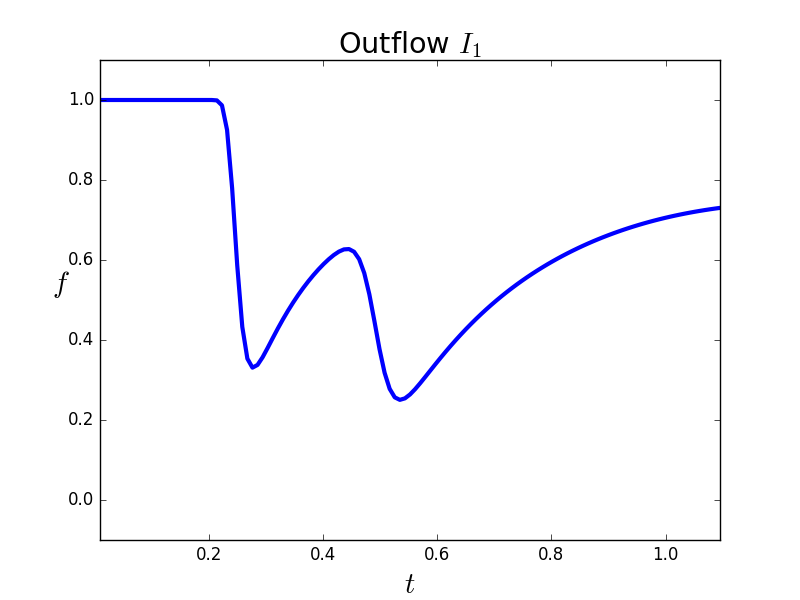}
  \includegraphics[width=5cm, height = 2cm]{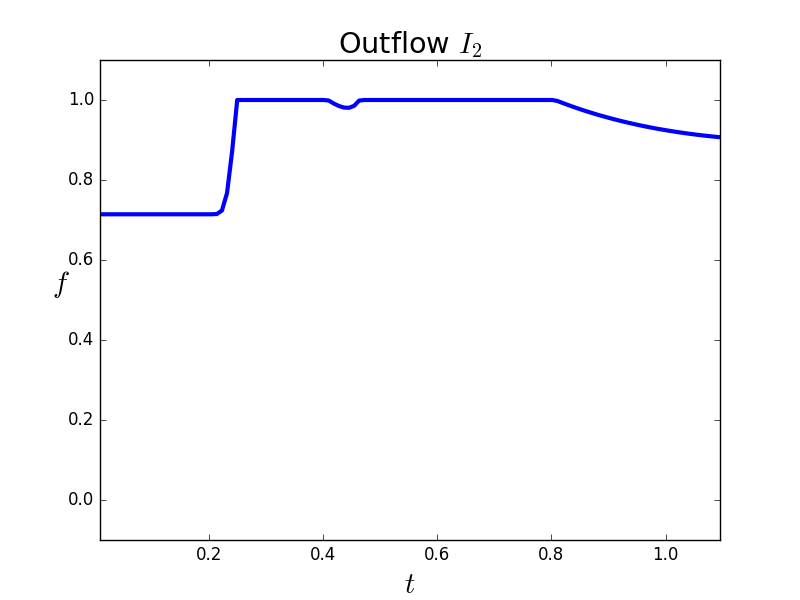}
  \caption{The fluxes at the node in the incoming arcs of solution,
    for Subsection~\ref{sse:constant-ic},
    obtained with the Riemann solver.}
  \label{fig:ss1:RS-fluxes}
\end{figure}

\subsection{Sinusoidal initial conditions}
\label{sse:shock}

% 2x2/simulations/sin2/product/ramp3

Here we consider the case of a sinusoidal initial condition. Indeed we consider
the initial datum
\begin{equation*}
  \overline u_1(x) = 0.5 + 0.3 \cdot \sin(2x),
  \qquad
  \overline u_2(x) = 0.2, \qquad
  \overline u_3(x) = 0.75 + 0.2 \cdot \cos(x), \qquad
  \overline u_4(x) = 0.1, 
\end{equation*}
the distributional matrix
\begin{equation*}
  A = \left(
    \begin{array}{cc}
      0.5 & 0.3
      \\
      0.5 & 0.7
    \end{array}
  \right),
\end{equation*}
and $\mathcal J(f_1, f_2) = f_1 f_2$.
Moreover, we consider the following choice of parameters:
$T = 5$, $M = 20$ and $\delta = 0$, so that the total variation of the fluxes is not
taken into account.
In Figure~\ref{fig:s4.3:optimal-fluxes} the numerical optimal solution is
represented. Note that in the arc $I_2$ a shock with negative speed
appears at time $t \sim 0$ and comes back to the node at time $t \sim 1$.
Moreover in the arc $I_1$ the right-hand shock behaves in a different way with
respect to the right-hand shock of 
the solution obtained with the Riemann solver;
see Figure~\ref{fig:s4.3:RS-sol}.
The numeric cost~(\ref{eq:numeric-cost}) for the optimal solution
is approximately $3.053$
while the cost for the solution obtained with the corresponding Riemann
solvers, see Figure~\ref{fig:s4.3:RS-sol}, is approximately $3.015$.
In Figures~\ref{fig:ss2:optimal-fluxes-sol} and~\ref{fig:ss2:RS-fluxes}
respectively the optimal fluxes and the fluxes obtained with the Riemann
solver are drawn. Here we observe that
the total variation of the optimal fluxes is bigger
than the total variation of the solution obtained through the Riemann solver.
This is not an unexpected feature, since optimal costs are achieved with the
price of bigger oscillations.
\begin{figure}
  \centering
  \includegraphics[width=7cm, height = 3cm]{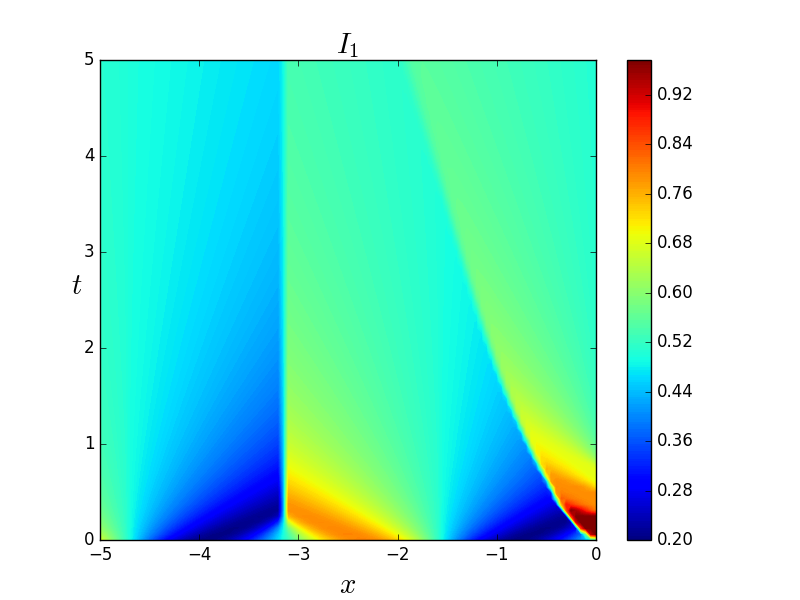}
  \includegraphics[width=7cm, height = 3cm]{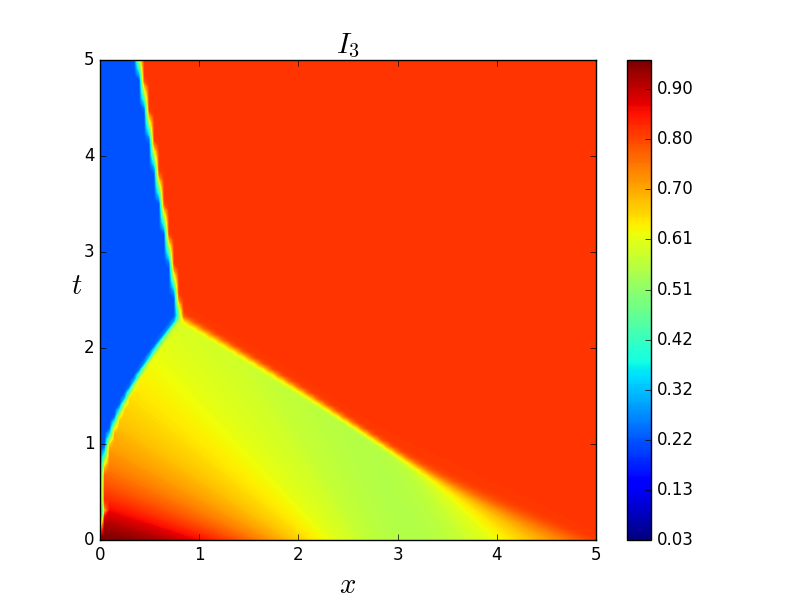}
  \\
  \includegraphics[width=7cm, height = 3cm]{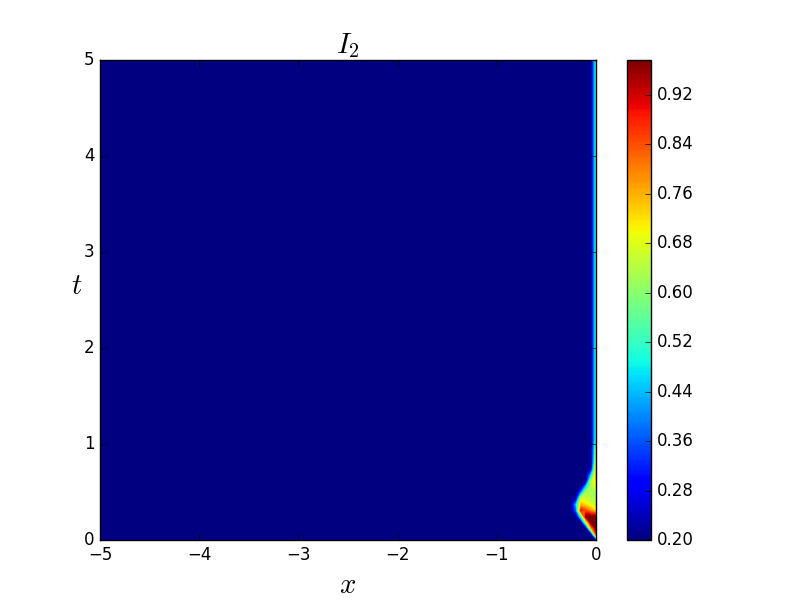}
  \includegraphics[width=7cm, height = 3cm]{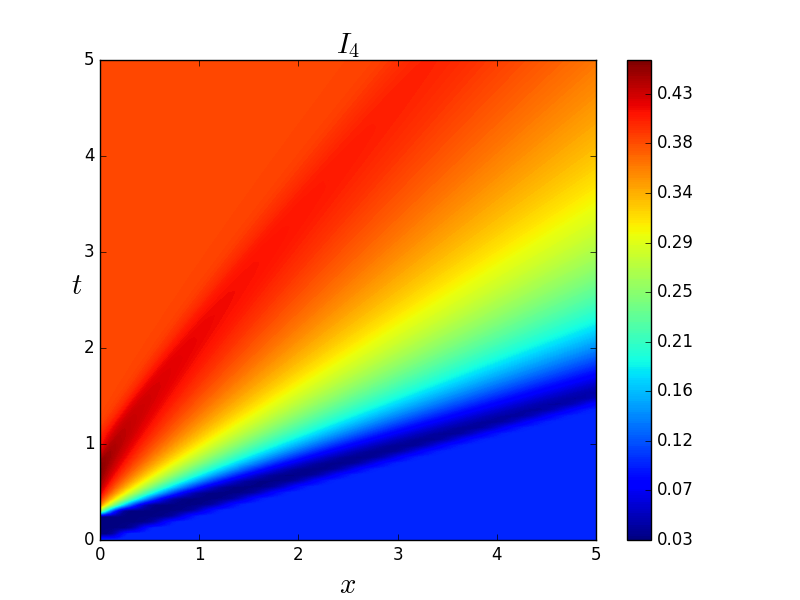}  
  \caption{The numerical optimal solution obtained in
    Subsection~\ref{sse:shock}.  In the first line, the solutions in
    $I_1$ and $I_3$. In the second line, the solutions in
    $I_2$ and $I_4$. (Color online)}
  \label{fig:s4.3:optimal-fluxes}
\end{figure}
% \begin{figure}
%   \centering
%   \includegraphics[width=7cm, height = 3cm]{figures/4.3/Contour-I1-range.png}
%   \includegraphics[width=7cm, height = 3cm]{figures/4.3/Contour-I3-range.png}
%   \\
%   \includegraphics[width=7cm, height = 3cm]{figures/4.3/Contour-I2-range.png}
%   \includegraphics[width=7cm, height = 3cm]{figures/4.3/Contour-I4-range.png}  
%   \caption{(SAME RANGE) The numerical optimal solution obtained in
%     Subsection~\ref{sse:shock}.  In the first line, the solutions in
%     $I_1$ and $I_3$. In the second line, the solutions in
%     $I_2$ and $I_4$.}
%   \label{fig:s4.3:optimal-fluxes}
% \end{figure}
\begin{figure}
  \centering
  \includegraphics[width=7cm, height = 3cm]{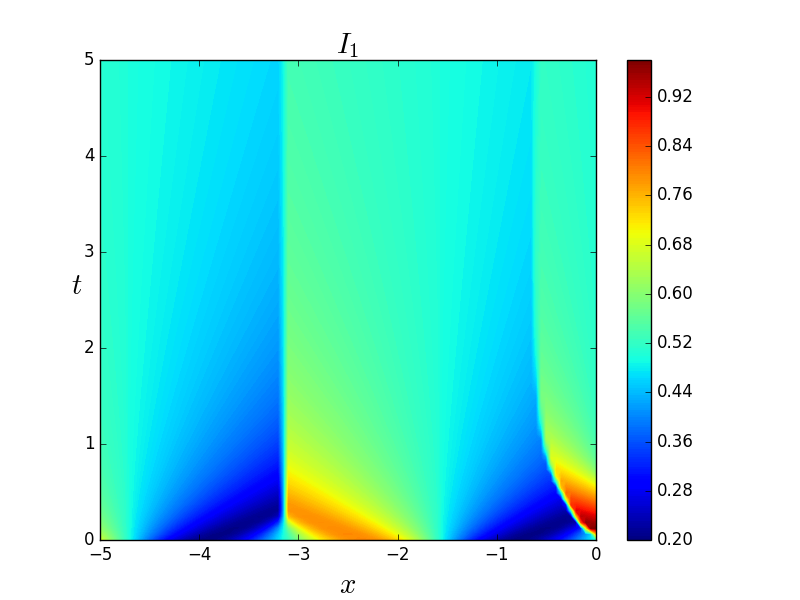}
  \includegraphics[width=7cm, height = 3cm]{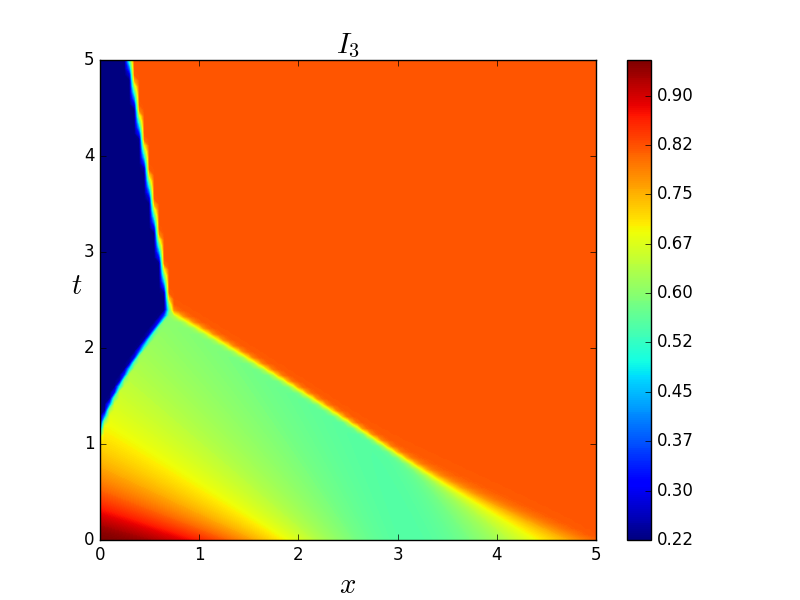}
  \\
  \includegraphics[width=7cm, height = 3cm]{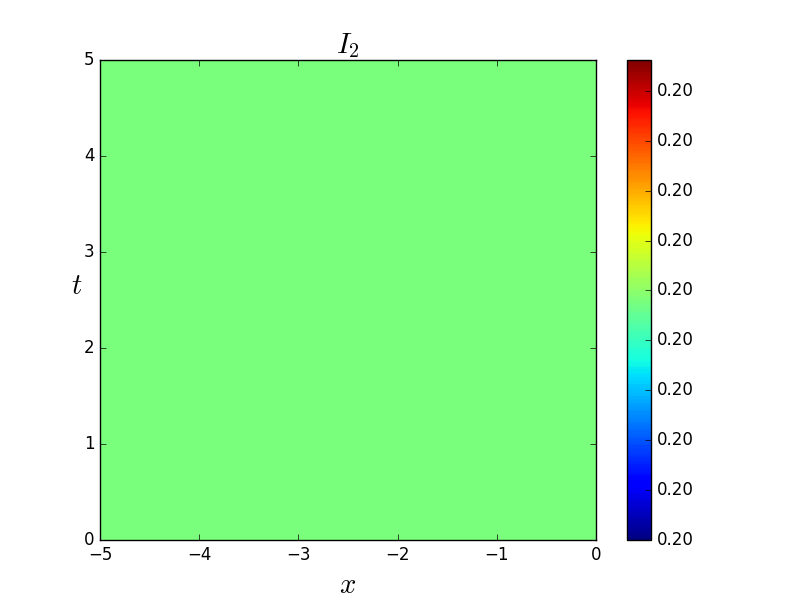}
  \includegraphics[width=7cm, height = 3cm]{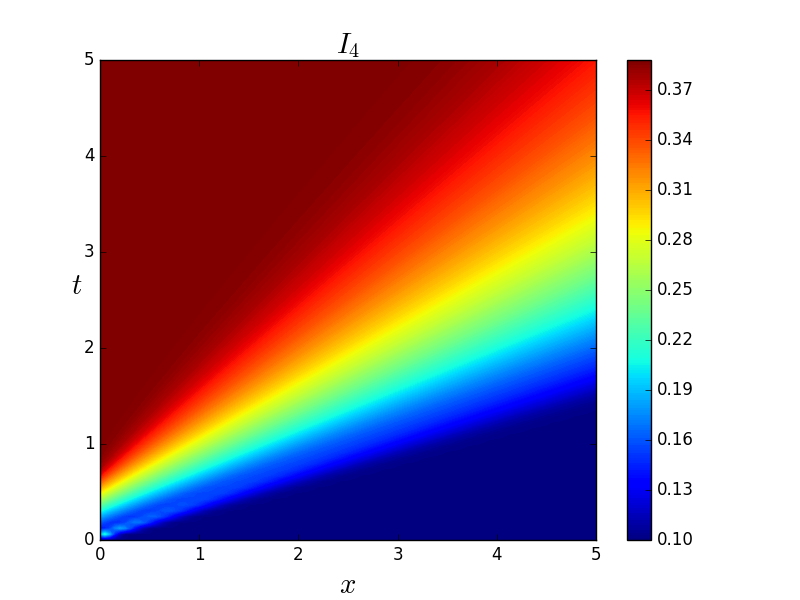}
  \caption{The solution, for Subsection~\ref{sse:shock},
    obtained with the Riemann solver. (Color online)}
  \label{fig:s4.3:RS-sol}
\end{figure}
% \begin{figure}
%   \centering
%   \includegraphics[width=7cm, height = 3cm]{figures/4.3/Contour-classic-I1-range.png}
%   \includegraphics[width=7cm, height = 3cm]{figures/4.3/Contour-classic-I3-range.png}
%   \\
%   \includegraphics[width=7cm, height = 3cm]{figures/4.3/Contour-classic-I2-range.png}
%   \includegraphics[width=7cm, height = 3cm]{figures/4.3/Contour-classic-I4-range.png}
%   \caption{(SAME RANGE) The solution, for Subsection~\ref{sse:shock},
%     obtained with the Riemann solver.}
%   \label{fig:s4.3:RS-sol}
% \end{figure}

\begin{figure}
  \centering
  \includegraphics[width=5cm, height = 2cm]{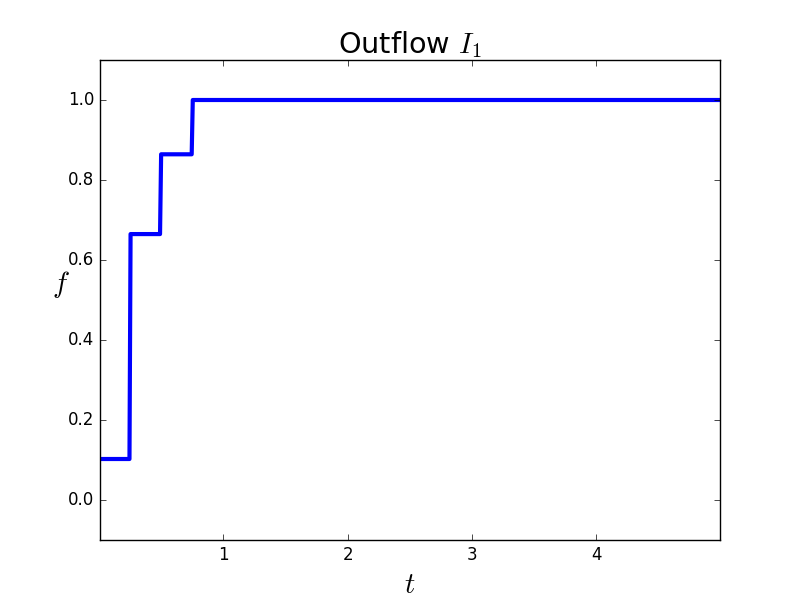}
  \includegraphics[width=5cm, height = 2cm]{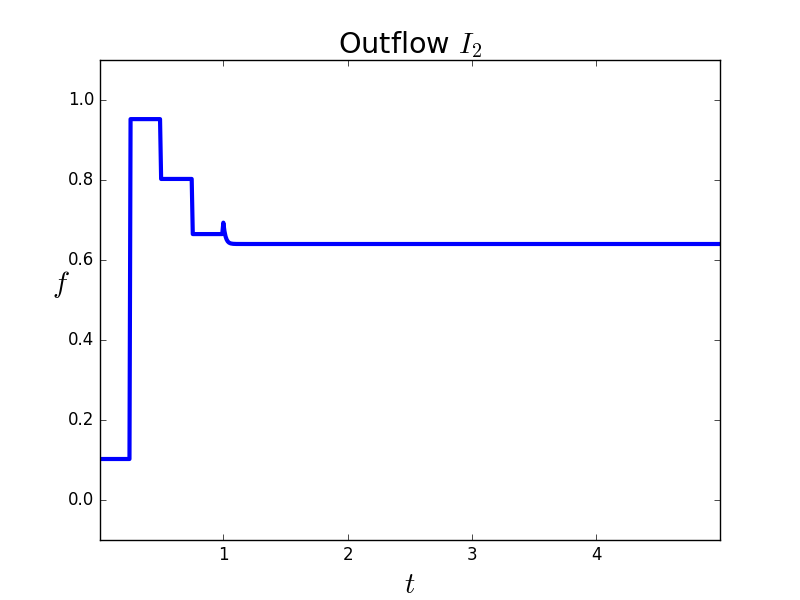}
  \caption{The numerical optimal fluxes at the node in the incoming arcs $I_1$
    and $I_2$, obtained in
    Subsection~\ref{sse:shock}.}
  \label{fig:ss2:optimal-fluxes-sol}
\end{figure}
\begin{figure}
  \centering
  \includegraphics[width=5cm, height = 2cm]{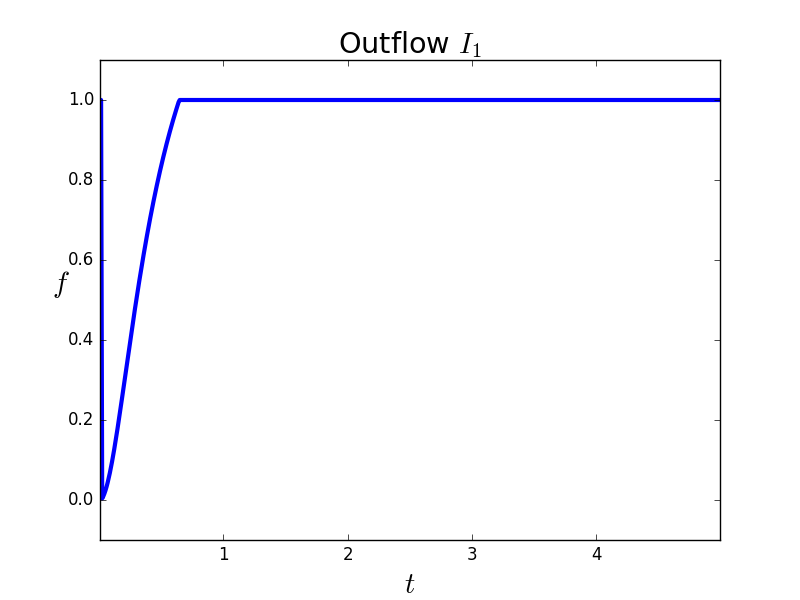}
  \includegraphics[width=5cm, height = 2cm]{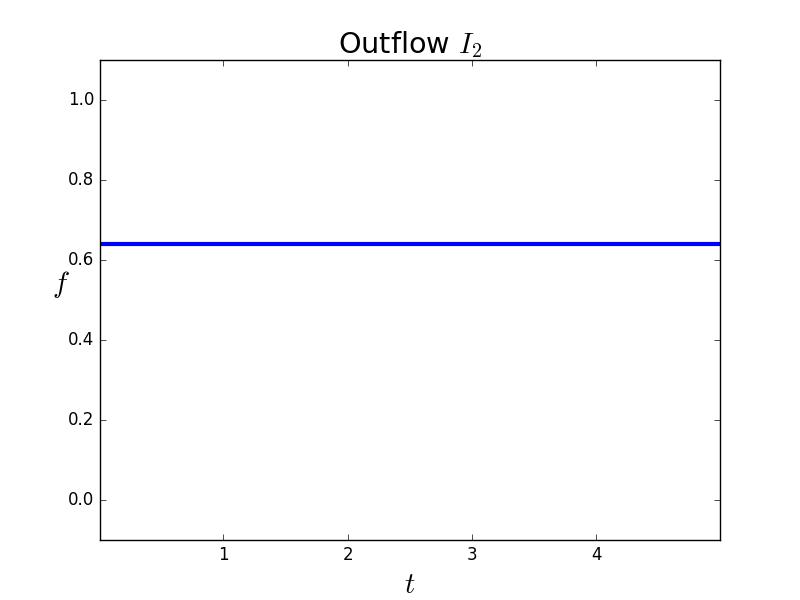}
  \caption{The fluxes at the node in the incoming arcs of solution,
    for Subsection~\ref{sse:shock},
    obtained with the Riemann solver.}
  \label{fig:ss2:RS-fluxes}
\end{figure}

% \section{Conclusions}\label{sec:concl} 

% This paper should be viewed as a first step towards developing sound control strategies 
% at the nodal points of a network.

%
% biblio
%  

\end{document}